\documentclass[a4paper,10pt]{article}
\usepackage{color, xcolor}
\usepackage[margin=2.3cm]{geometry}
\usepackage[pdftex]{graphicx}
\usepackage{amsfonts, amsmath, fancyref, amsthm, bm}
\usepackage{mathtools}
\usepackage{comment}
\usepackage[toc,page]{appendix}
\usepackage{enumitem}
\usepackage{setspace}
\usepackage{amsthm}
\usepackage{anyfontsize} 
\usepackage{hyperref}
\usepackage{mathrsfs}
\hypersetup{colorlinks=true,
    linkcolor=red,
    filecolor=magenta,      
    urlcolor=cyan,
   }
\newtheorem{theorem}{Theorem}[section]
\newtheorem{corollary}[theorem]{Corollary}
\newtheorem{lemma}[theorem]{Lemma}
\newtheorem{definition}{Definition}[section]
\newtheorem{proposition}[theorem]{Proposition}
\newtheorem{assum}{Assumption}

\newtheorem*{remark}{Remark}
\def\one{\mbox{1\hspace{-4.25pt}\fontsize{12}{14.4}\selectfont\textrm{1}}}
\numberwithin{equation}{section}

\begin{document}

\title{\textbf{Convergence of blanket times for sequences of random walks on critical random graphs}}
\date{}
\author{George Andriopoulos \\ \normalsize Mathematics Institute, University of Warwick, CV4 7AL, UK \\ \small g.andriopoulos@warwick.ac.uk \thanks{This work is supported by EPSRC as part of the MASDOC DTC at the University of Warwick. Grant No. EP/HO23364/1.}}

\maketitle

\begin{abstract} 
Under the assumption that sequences of graphs equipped with resistances, associated measures, walks and local times converge in a suitable Gromov-Hausdorff topology, we establish asymptotic bounds on the distribution of the $\varepsilon$-blanket times of the random walks in the sequence. The precise nature of these bounds ensures convergence of the $\varepsilon$-blanket times of the random walks if the $\varepsilon$-blanket time of the limiting diffusion is continuous with probability one at $\varepsilon$. This result enables us to prove annealed convergence in various examples of critical random graphs, including critical Galton-Watson trees, the Erd\H{o}s-R\'enyi random graph in the critical window and the configuration model in the scaling critical window. 

We highlight that proving continuity of the $\varepsilon$-blanket time of the limiting diffusion relies on the scale invariance of a finite measure that gives rise to realizations of the limiting compact random metric space, and therefore we expect our results to hold for other examples of random graphs with a similar scale invariance property.
\end{abstract}

\section{Introduction}

A simple random walk on a finite connected graph $G$ with at least two vertices is a reversible Markov chain that starts at some vertex $v\in G$, and at each step moves with equal probability to any vertex adjacent to its present position. The mixing and the cover time of the random walk are among the graph parameters which have been extensively studied. To these parameters, Winkler and Zuckerman \cite{zuckermanmultiple} added the $\varepsilon$-blanket time variable (an exact definition will be given later in \eqref{blank1}) as the least time such that the walk has spent at every vertex at least an $\varepsilon$ fraction of time as much as expected at stationarity. Then, the $\varepsilon$-blanket time of $G$ is defined as the expected $\varepsilon$-blanket time variable maximized over the starting vertex. 

The necessity of introducing and studying  the blanket time arises mainly from applications in computer science. For example, suppose that a  limited access to a source of information is randomly transferred from (authorized) user to user in a network. How long does it take for each user to own the information for as long as it is supposed to? To answer this question under the assumption that each user has to be active processing the information equally often involves the consideration of the blanket time. To a broader extent, viewing the internet as a (directed) graph, where every edge represents a link, a web surfer can be regarded as a walker who visits and records the sites at random. In a procedure that resembles Google's  PageRank (PR), one wishes to rank a website according to the amount of time such walkers spend on it. A way to produce such an estimate is to rank the website according to the number of visits. The blanket time is the first time at which we expect this estimate to become relatively accurate.

Obviously, for every $\varepsilon\in (0,1)$, the $\varepsilon$-blanket time is larger than the cover time since one has to wait for all the vertices to have been visited at least once. Winkler and Zuckerman \cite{zuckermanmultiple} made the conjecture that, for every $\varepsilon\in (0,1)$, the $\varepsilon$-blanket time and the cover time are equivalent up to universal constants that depend only on $\varepsilon$ and not on the particular underlying graph $G$. This conjecture was resolved by Ding, Lee and Peres \cite{ding2011cover} who provided a highly non-trivial connection between those graph parameters and the discrete Gaussian free field (GFF) on $G$ using Talagrand's theory of majorizing measures. Recall that the GFF on $G$ with vertex set $V$ is a centered Gaussian process $(\eta_v)_{v\in V}$ with $\eta_{v_0}=0$, for some $v_0\in V$, and covariance structure given by the Green kernel of the random walk killed at $v_0$. 

Recent years have witnessed a growing interest in studying the geometric and analytic properties of random graphs partly motivated by applications in research areas ranging from sociology and systems biology to interacting particle systems as well as by the need to present convincing models to gain insight into real-world networks. One aspect of this develpoment consists of examining the metric structure and connectivity of random graphs at criticality, that is precicely when we witness the emergence of a giant component that has size proportional to the number of vertices of the graph. Several examples of trees, including critical Galton-Watson trees, possess the Brownian Continuum Random Tree (CRT) as their scaling limit. A program \cite{bhamidi2014scaling} has been launched in the last few years having as its general aim to prove that the maximal components in the critical regime of a number of fundamental random graph models, with their distances scaling like $n^{1/3}$, fall into the basin of attraction of the Erd\H{o}s-R\'enyi random graph. Their scaling limit is a multiple of the scaling limit of the Erd\H{o}s-R\'enyi random graph in the critical window, that is a tilted version of the Brownian CRT where a finite number of vertices have been identified. Two of the examples that belong to the Erd\H{o}s-R\'enyi universality class are the configuration model in the critical scaling window and critical inhomogeneous random graphs, where different vertices have different proclivity to form edges, as it was shown in the recent work of \cite{bhamidi2016geometry} and \cite{bhamidi2017inhomo} respectively.

In \cite{croydon2012convergence}, Croydon, Hambly and Kumagai established criteria for the convergence of mixing times for random walks on general sequences of finite graphs. Furthermore, they applied their mixing time results in a number of examples of random graphs, such as self-similar fractal graphs with random weights, critical Galton-Watson trees, the critical Erd\H{o}s-R\'enyi random graph and the range of high-dimensional random walk. Motivated by their approach, starting   with the strong assumption that the sequences of graphs, associated measures, walks and local times converge appropriately, we provide asymptotic bounds on the distribution of the blanket times of the random walks in the sequence. 

To state the aforementioned assumption, we continue by introducing the graph theoretic framework in which we work. Firstly, Let $G=(V(G),E(G))$ be a finite connected graph with at least two vertices, where $V(G)$ denotes the vertex set of $G$ and $E(G)$ denotes the edge set of $G$. We endow the edge set $E(G)$ with a symmetric weight function $\mu^G: V(G)^2\rightarrow \mathbb{R_{+}}$ that satisfies $\mu_{xy}^G>0$ if and only if $\{x,y\}\in E(G)$. Now, the weighted random walk associated with $(G,\mu^G)$ is the Markov chain $((X^G_t)_{t\ge 0},\mathbf{P}_{x}^G,x\in V(G))$ with transition probabilities $(P_G(x,y))_{x,y\in V(G)}$ given by 
\[
P_{G}(x,y):=\frac{\mu^G_{xy}}{\mu_x^G},
\]
where $\mu^G_x=\sum_{y\in V(G)} \mu_{xy}^G$. One can easily check that this Markov chain is reversible and has stationary distribution given by 
\[
\pi^G(A):=\frac{\sum_{x\in A} \mu_x^G}{\sum_{x\in V(G)} \mu_x^G},
\]
for every $A\subseteq V(G)$. The process $X^G$ has corresponding local times $(L_t^G(x))_{x\in V(G),t\ge 0}$ given by $L_0^G(x)=0$, for every $x\in V(G)$, and, for $t\ge 1$
\[
L_t^G(x):=\frac{1}{\mu_x^G} \sum_{i=0}^{t-1} \mathbf{1}_{\{X_i^G=x\}}.
\]
The simple random walk on this graph is a Markov chain with transition probabilities given by $P(x,y):=1/\text{deg}(x)$ for all $x\in V(G)$, such that $\{x,y\}\in E(G)$. Observe that the simple random walk on $G$ is a weighted random walk on $G$ that assigns a constant unit weight to each edge. 

To endow $G$ with a metric, we can choose $d_G$ to be the shortest path distance, which counts the number of edges in the shortest path between a pair of vertices in $G$. But this is not the most convenient choice in many examples. Another typical graph distance that arises from the view of $G$ as an electrical network equipped with conductances $(\mu^G_{xy})_{\{x,y\}\in E(G)}$ is the so-called resistance metric. For $f, g: V(G)\rightarrow \mathbb{R}$ let
\begin{equation} \label{resist1}
\mathcal{E}_G(f,g):=\frac{1}{2} \sum_{\substack{x,y\in V(G): \\ \{x,y\}\in E(G)}} (f(x)-f(y))(g(x)-g(y)) \mu_{xy}^G
\end{equation}
denote the Dirichlet form associated with the process $X^G$. Note that the sum in the expression above counts each edge twice. One can give the following interpretation of $\mathcal{E}_G(f,f)$ in terms of electrical networks. Given a voltage $f$ on the network, the current flow $I$ associated with $f$ is defined as $I_{xy}:=\mu_{xy}^G (f(x)-f(y))$, for every $\{x,y\}\in E(G)$. Then, the energy dissipation of a wire connecting $x$ and $y$ is $\mu^G_{xy} (f(x)-f(y))^2$. So, $\mathcal{E}_G(f,f)$ is the total energy dissipation of $G$. We define the resistance operator on disjoint sets $A, B\in V(G)$ through the formula
\begin{equation} \label{resist2}
R_{G}(A,B)^{-1}:=\inf \{\mathcal{E}_G(f,f): f: V(G)\rightarrow \mathbb{R}, f|_{A}=0, f|_{B}=1\}.
\end{equation}
Now, the distance on the vertices of $G$ defined by $R_G(x,y):=R_G(\{x\},\{y\})$, for $x\neq y$, and $R_G(x,x):=0$ is indeed a metric on the vertices of $G$. For a proof and a treatise on random walks on electrical networks see \cite[Chapter 9]{levin2017markov}.

For some $\varepsilon\in (0,1)$, define the $\varepsilon$-blanket time variable by 
\begin{equation} \label{blank1}
\tau_{\text{bl}}^G(\varepsilon):=\inf\{t\ge 0: m^G L_t^G(x)\ge \varepsilon t, \ \forall x\in V(G)\},
\end{equation}
where $m^G$ is the total mass of the graph with respect to the measure $\mu^G$, i.e. $m^G:=\sum_{x\in V(G)} \mu_x^G$. 
Taking the mean over the random walk started from the worst possible vertex defines the $\varepsilon$-blanket time, i.e.
\[
t_{\text{bl}}^G(\varepsilon):=\max_{x\in V(G)} \mathbf{E}_x \tau_{\text{bl}}^G(\varepsilon).
\]
Secondly, let $(K,d_K)$ be a compact metric space and let $\pi$ be a Borel measure of full support on $(K,d_K)$. Take $((X_t)_{t\ge 0},\mathbf{P}_x,x\in K)$ to be a $\pi$-symmetric Hunt process that admits local times $(L_t(x))_{x\in K,t\ge 0}$ continuous at $x$, uniformly over compact time intervals in $t$, $\mathbf{P}_x$-a.s. for every $x\in K$. A Hunt process is a strong Markov process that possesses useful properties such as the right-continuity and the existense of the left limits of sample paths (for definitions and other properties see \cite[Appendix A.2]{fukushima2010dirichlet}). Analogously, it is possible to define the $\varepsilon$-blanket time variable of $K$ as
\begin{equation} \label{blacken1}
\tau_{\text{bl}}(\varepsilon):=\inf \{t\ge 0: L_t(x)\ge \varepsilon t, \ \forall x\in K\}
\end{equation}
and check that is a non-trivial quantity (see Proposition \ref{nontrivexp}).

The following assumption encodes the information that, properly rescaled, the discrete state spaces, invariant measures, random walks, and local times, converge to $(K,d_k)$, $\pi$, $X$, and $(L_t(x))_{x\in K,t\in [0,T]}$ respectively, for some fixed $T>0$. This formulation will be described in terms of the extended Gromov-Hausdorff topology constructed in Section \ref{extendedsec}.

\begin{assum} \label{Assum1}
Fix $T>0$. Let $(G^n)_{n\ge 1}$ be a sequence of finite connected graphs that have at least two vertices, for which there exist sequences of real numbers $(\alpha(n))_{n\ge 1}$ and $(\beta(n))_{n\ge 1}$, such that
\[
\left(\left(V(G^n),\alpha(n) d_{G^n},\rho^n\right),\pi^n,\left(X^n_{\beta(n) t}\right)_{t\in [0,T]},\left(L^n_{\beta(n) t}(x)\right)_{\substack{x\in V(G^n), \\ t\in [0,T]}}\right)\longrightarrow \left(\left(K,d_K,\rho\right),\pi,X,\left(L_t(x)\right)_{\substack{x\in K, \\ t\in [0,T]}}\right)
\]
in the sense of the extended pointed Gromov-Hausdorff topology, where $\rho^n\in V(G^n)$ and $\rho\in K$ are distinguished points. In the above expression the definition of the discrete local times is extended to all positive times by linear interpolation.
\end{assum}

In most of the examples that will be discussed later, we will consider random graphs. In this context, we want to verify that the previous convergence holds in distribution. Our first conclusion is the following.

\begin{theorem}  \label{Mth}
Suppose that Assumption \ref{Assum1} holds in such a way that the time and space scaling factors satisfy $\alpha(n) \beta(n)=m^{G^n}$, for every $n\ge 1$. Then, for every $\varepsilon\in (0,1)$, $\delta\in (0,1)$ and $t\in [0,T]$,
\begin{equation} \label{blacken2}
\limsup_{n\rightarrow \infty} \mathbf{P}_{\rho_n}^n \left(\beta(n)^{-1} \tau_{\textnormal{bl}}^n(\varepsilon)\le t\right)\le \mathbf{P}_{\rho} \left(\tau_{\textnormal{bl}}(\varepsilon(1-\delta))\le t\right),
\end{equation} 
\begin{equation} \label{blacken3}
\liminf_{n\rightarrow \infty} \mathbf{P}_{\rho_n}^n \left(\beta(n)^{-1} \tau_{\textnormal{bl}}^n(\varepsilon)\le t\right)\ge \mathbf{P}_{\rho} \left({\tau}_{\textnormal{bl}}(\varepsilon)<t\right),
\end{equation}
where $\mathbf{P}_{\rho}$ is the law of $X$ on $K$, started from $\rho$.
\end{theorem}

The mapping $\varepsilon\mapsto \tau_{\textnormal{bl}}(\varepsilon)$ is increasing in $(0,1)$, so it posseses left and right limits at each point. It also becomes clear that if $\tau_{\textnormal{bl}}(\varepsilon)$ is continuous with probability one at $\varepsilon$, then
$\lim_{\delta\to 0}\tau_{\textnormal{bl}}(\varepsilon(1-\delta))=\tau_{\textnormal{bl}}(\varepsilon)$,
holds with probability one. Letting $\delta\to 0$ on both \eqref{blacken2} and \eqref{blacken3} demonstrates that $\beta(n)^{-1} \tau^n_{\textnormal{bl}}(\varepsilon)\to \tau_{\textnormal{bl}}(\varepsilon)$ in distribution.

\begin{corollary} \label{verification}
Suppose that Assumption \ref{Assum1} holds in such a way that the time and space scaling factors satisfy $\alpha(n) \beta(n)=m^{G^n}$, for every $n\ge 1$. If $\tau_{\textnormal{bl}}(\varepsilon)$ is continuous with probability one at $\varepsilon$, then
$
\beta(n)^{-1} \tau^n_{\textnormal{bl}}(\varepsilon)\to \tau_{\textnormal{bl}}(\varepsilon)
$
in distribution.
\end{corollary}

To demonstrate our main results consider first $T$, a critical Galton-Watson tree (with finite variance $\sigma^2$). The following result on the cover time of the simple random walk was obtained by Aldous (see \cite[Proposition 15]{aldous1991random}), which we apply to the blanket time in place of the cover time. The two parameters are equivalent up to universal constants as was conjectured in \cite{zuckermanmultiple} and proved in \cite{ding2011cover}. 

\begin{theorem} [\textbf{Aldous \cite{aldous1991random}}]
Let $T$ be a critical Galton-Watson tree (with finite variance $\sigma^2$). For any $\delta>0$ there exists $A=A(\delta,\varepsilon,\sigma^2)>0$ such that 
\[
P(A^{-1} k^{3/2}\le t_{\textnormal{bl}}^{T}(\varepsilon)\le A k^{3/2}||T|\in [k,2k])\ge 1-\delta,
\]
for every $\varepsilon\in (0,1)$.
\end{theorem}

Now let $G(n,p)$ be the resulting subgraph of the complete graph on $n$ vertices obtained by $p$-bond percolation. If $p=n^{-1}+\lambda n^{-4/3}$ for some $\lambda\in \mathbb{R}$, that is when we are in the so-called critical window, the largest connected component $\mathcal{C}_1^n$, as a graph, converges to a random compact metric space $\mathcal{M}$ that can be constructed directly from the Brownian CRT $\mathcal{T}_e$ (see the work of \cite{addario2012continuum}). The following result on the blanket time of the simple random walk on $\mathcal{C}_1^n$ is due to Barlow, Ding, Nachmias and Peres \cite{barlow2011evolution}.

\begin{theorem} [\textbf{Barlow, Ding, Nachmias, Peres \cite{barlow2011evolution}}]
Let $\mathcal{C}_1^n$ be the largest connected component of $G(n,p)$, $p=n^{-1}+\lambda n^{-4/3}$, $\lambda\in \mathbb{R}$ fixed. For any $\delta>0$ there exists $B=B(\delta,\varepsilon)>0$ such that
\[
P(B^{-1} n\le t_{\textnormal{bl}}^{\mathcal{C}_1^n}(\varepsilon)\le B n)\ge 1-\delta,
\]
for every $\varepsilon\in (0,1)$.
\end{theorem}

Our contribution refines the previous existing tightness results on the order of the blanket time. In what follows $\mathbb{P}_{\rho^n}$, $n\ge 1$ as well as $\mathbb{P}_{\rho}$ are the annealed measures, that is the probability measures obtained by integrating out the randomness of the state spaces involved. We remark also here that we can apply our main result due to the recent work of \cite{croydon2016scaling} that generalises previous work done in \cite{croydon2008convergence} and \cite{croydon2012scaling}.
 
\begin{theorem} \label{finish1}
Let $\mathcal{T}_n$ be a critical Galton-Watson tree (with finite variance) conditioned to have total progeny $n+1$. Fix $\varepsilon\in (0,1)$. If $\tau_{\textnormal{bl}}^n(\varepsilon)$ is the $\varepsilon$-blanket time variable of the simple random walk on $\mathcal{T}_n$, started from its root $\rho^n$, then 
\[
\mathbb{P}_{\rho^n}\left(n^{-3/2} \tau_{\textnormal{bl}}^n(\varepsilon)\le t\right)\to \mathbb{P}_{\rho}\left(\tau_{\textnormal{bl}}^e(\varepsilon)\le t\right),
\]
for every $t\ge 0$, where $\tau_{\textnormal{bl}}^e(\varepsilon)\in (0,\infty)$ is the $\varepsilon$-blanket time variable of the Brownian motion on $\mathcal{T}_e$, started from a distinguished point $\rho\in \mathcal{T}_e$. Equivalently, for every $\varepsilon\in (0,1)$, $n^{-3/2} \tau_{\textnormal{bl}}^n(\varepsilon)$ under $\mathbb{P}_{\rho^n}$ converges weakly to $\tau_{\textnormal{bl}}^e(\varepsilon)$ under $\mathbb{P}_{\rho}$.
 \end{theorem}
 
 \begin{theorem}
 Fix $\varepsilon\in (0,1)$. If $\tau_{\textnormal{bl}}^{n}(\varepsilon)$ is the $\varepsilon$-blanket time variable 
of the simple random walk on $\mathcal{C}_1^n$, started from its root $\rho^n$, then 
\[
\mathbb{P}_{\rho^n}\left(n^{-1} \tau_{\textnormal{bl}}^n(\varepsilon)\le t\right)\to \mathbb{P}_{\rho}\left(\tau_{\textnormal{bl}}^{\mathcal{M}}(\varepsilon)\le t\right),
\]
for every $t\ge 0$, where $\tau_{\textnormal{bl}}^{\mathcal{M}}(\varepsilon)\in (0,\infty)$ is the $\varepsilon$-blanket time variable of the Brownian motion on $\mathcal{M}$, started from $\rho$. 
 \end{theorem}
 
 Moreover, to present our last result we consider the configuration model. Let $M^n(d)$ be the random multigraph labelled by $[n]$ such that the $i$-th vertex has degree $d_i$, $i\ge 1$, for every $1\le i\le n$, which is constructed as follows. Assign $d_i$ half-edges to each vertex $i$, labelling them in an arbitrary way. Then, the configuration model is produced by a uniform pairing of the half-edges to create full edges. If the degree sequence satisfies certain conditions that would be precicely given later in Assumption \ref{gracias}, it was shown in the work of \cite{dhara2017critical} that the largest connected component $M_1^n(d)$, is of order $n^{2/3}$. Recently in \cite{bhamidi2016geometry} its scaling limit, $\mathcal{M}_D$, was proven to exist and to belong to the Erd\H{o}s-R\'enyi universality class.
 
 \begin{theorem} \label{finish3}
 Fix $\varepsilon\in (0,1)$. If $\tau_{\textnormal{bl}}^{n}(\varepsilon)$ is the $\varepsilon$-blanket time variable of the simple random walk on $M_1^n(d)$, started from its root $\rho^n$, then 
\[
\mathbb{P}_{\rho^n}\left(n^{-1} \tau_{\textnormal{bl}}^n(\varepsilon)\le t\right)\to \mathbb{P}_{\rho}\left(\tau_{\textnormal{bl}}^{\mathcal{M}_D}(\varepsilon)\le t\right),
\]
for every $t\ge 0$, where $\tau_{\textnormal{bl}}^{\mathcal{M}_D}(\varepsilon)\in (0,\infty)$ is the $\varepsilon$-blanket time variable of the Brownian motion on $\mathcal{M}_D$, started from $\rho$. 
\end{theorem}
 
The paper is organized as follows. In Section \ref{extendedsec}, we introduce the extended Gromov-Hausdorff topology and derive some useful properties. In Section \ref{blankbounds}, we prove Theorem \ref{Mth} under Assumption \ref{Assum1} and present Assumption \ref{Assum3}, a weaker sufficient assumption when the sequence of spaces is equipped with resistance metrics. In Section \ref{maintheorems}, we verify the assumptions of Corollary \ref{verification}, and therefore prove convergence of blanket times for the series of critical random graphs mentioned above, thus effectively proving Theorem \ref{finish1}-Theorem \ref{finish3}. The main tool we employ to prove continuity of the $\varepsilon$-blanket time of the diffusion on the limiting spaces that appear in the statements of the preceding theorems is to exploit scale invariance properties of the finite measures, such as  It\^{o}'s excursion measure (see Section \ref{needed} for key facts of its theory), that give rise to realizations of those spaces. For that reason we believe our results to easily transfer when considering Galton-Watson trees with critical offspring distribution in the domain of attraction of a stable law with index $\alpha\in (1,2)$ (see \cite[Theorem 4.3]{le2006random}) and random stable looptrees (see \cite[Theorem 4.1]{curien2014loop}). Also, we hope our work to be seen as a stepping stone to deal with the more delicate problem of establishing convergence in distribution of the rescaled cover times of the discrete-time walks in each application of our main result. See \cite[Remark 7.4]{croydon2015moduli} for a thorough discussion on the demanding nature of this project.

\section{Extended Gromov-Hausdorff topologies} \label{extendedsec}

In this section we define an extended Gromov-Hausdorff distance between quadruples consisting of a compact metric space , a Borel measure, a time-indexed right-continuous path with left-hand limits and a local time-type function. This allows us to make precise the assumption under which we are able to prove convergence of blanket times for the random walks on various models of critical random graphs. In Lemma \ref{2.2}, we give an equivalent characterization of the assumption that will be used in Section \ref{blankbounds} when proving distributional limits for the rescaled blanket times. Also, Lemma \ref{2.3} will be useful when it comes checking that the examples we treat satisfy the assumption.         

Let $(K,d_K)$ be a non-empty compact metric space. For a fixed $T>0$, let $X^K$ be a path in $D([0,T],K)$, the space of c\`adl\`ag functions, i.e. right-continuous functions with left-hand limits, from $[0,T]$ to $K$. We say that a function $\lambda$ from $[0,T]$ onto itself is a time-change if it is strictly increasing and continuous. Let $\Lambda$ denote the set of all time-changes. If $\lambda\in \Lambda$, then $\lambda(0)=0$ and $\lambda(T)=T$. We equip $D([0,T],K)$ with the Skorohod metric $d_{J_1}$ defined as follows:
\[
d_{J_1}(x,y):=\inf_{\lambda\in \Lambda} \bigg\{\sup_{t\in [0,T]} |\lambda(t)-t|+\sup_{t\in [0,T]} d_{K} (x(\lambda(t)),y(t))\bigg\},
\]
for $x,y\in D([0,T],K)$. The idea behind going from the uniform metric to the Skorohod metric $d_{J_1}$ is to say that two paths are close if they are uniformly close in $[0,T]$, after allowing small perturbations of time.
Moreover, $D([0,T],K)$ endowed with $d_{J_1}$ becomes a separable metric space (see \cite[Theorem 12.2]{billingsley2013convergence}). Let $\mathcal{P}(K)$ denote the space of Borel measures on $K$. If $\mu,\nu\in \mathcal{P}(K)$ we set
\[
d_{P}(\mu,\nu)=\inf\{\varepsilon>0: \mu(A)\le \nu(A^{\varepsilon})+\varepsilon\text{ and }\nu(A)\le \mu(A^{\varepsilon})+\varepsilon,\text{ for any }A\in \mathcal{M}(K)\},
\]
where $\mathcal{M}(K)$ is the set of all closed subsets of $K$. This expression gives the standard Prokhorov metric between $\mu$ and $\nu$. Moreover, it is known , see \cite{daley2007introduction} Appendix A.2.5, that $(\mathcal{P}(K),d_{P})$ is a Polish metric space, i.e. a complete and separable metric space, and the topology generated by $d_P$ is exactly the topology of weak convergence, the convergence against bounded and continuous functionals. 

Let $\pi^K$ be a Borel measure on $K$ and $L^K=(L_t^K(x))_{x\in K,t\in [0,T]}$ be a jointly continuous function of $(t,x)$ taking positive real values. Let $\mathbb{K}$ be the collection of quadruples $(K,\pi^K,X^K,L^K)$. We say that two elements $(K,\pi^K,X^K,L^K)$ and $(K',\pi^{K'},X^{K'},L^{K'})$ of $\mathbb{K}$ are equivalent if there exists an isometry $f:K\rightarrow K'$ such that 
\begin{itemize}

\item $\pi^{K}\circ f^{-1}=\pi^{K'}$, 

\item $f\circ X^K=X^{K'},$ which is a shorthand of $f(X_t^K)=X_t^{K'}$, for every $t\in [0,T]$.

\item $L_t^{K'}\circ f=L_t^K$, for every $t\in [0,T]$, which is a shorthand of $L_t^{K'}(f(x))=L_t^K(x)$, for every $t\in [0,T]$, $x\in K$.

\end{itemize}
Not to overcomplicate our notation, we will often identify an equivalence class of $\mathbb{K}$ with a particular element of it. We now introduce a distance $d_{\mathbb{K}}$ on $\mathbb{K}$ by setting
\begin{align*}
d_{\mathbb{K}}&((K,\pi^K,X^K,L^K),(K',\pi^{K'},X^{K'},L^{K'}))
\\
&:=\inf_{Z,\phi,\phi',\mathcal{C}}\bigg\{d_P^Z(\pi^K\circ \phi^{-1},\pi^{K'}\circ \phi'^{-1})+d_{J_1}^Z(\phi(X_t^K),\phi'(X_t^{K'}))
\\
&+\sup_{(x,x')\in \mathcal{C}}\bigg(d_{Z}(\phi(x),\phi'(x'))+\sup_{t\in [0,T]} |L_t^K(x)-L_t^{K'}(x')|\bigg)\bigg\},
\end{align*}
where the infimum is taken over all metric spaces $(Z,d_Z)$, isometric embeddings $\phi:K\rightarrow Z$, $\phi':K'\rightarrow Z$ and correspondences $\mathcal{C}$ between $K$ and $K'$. A correspondence between $K$ and $K'$ is a subset of $K\times K'$, such that for every $x\in K$ there exists at least one $x'$ in $K'$ such that $(x,x')\in \mathcal{C}$ and conversely for every $x'\in K'$ there exists at least one $x\in K$ such that $(x,x')\in \mathcal{C}$. In the above expression $d^Z_P$ is the standard Prokhorov distance between Borel measures on $Z$, and $d_{J_1}^Z$ is the Skorohod metric $d_{J_1}$ between c\`adl\`ag paths on $Z$. 

	In the following proposition we check that the definition of $d_{\mathbb{K}}$ induces a metric and that the resulting metric space is separable. The latter fact will be used repeatedly later when it comes to applying Skorohod's represantation theorem on sequences of random graphs to prove statements regarding their blanket times or the cover times. Before proceeding to the proof of Proposition \ref{Prop1.1}, let us first make a few remarks about the ideas behind the definition of $d_{\mathbb{K}}$. The first term along with the Hausdorff distance on $Z$ between $\phi(K)$ and $\phi'(K')$ is that used in the Gromov-Hausdorff-Prokhorov distance  for compact metric spaces (see \cite[Section 2.2, (6)]{abraham2013note}). Though, in our definition of $d_{\mathbb{K}}$ we did not consider the Hausdorff distance between the embedded compact metric spaces $K$ and $K'$, since it is absorbed by the first part of the third term in the expression for $d_{\mathbb{K}}$. Recall here the equivalent definition of the standard Gromov-Hausdorff distance via correspondences as a way to relate two compact metric spaces (see \cite[Theorem 7.3.25]{burago2001course}). The motivation for the second term comes from \cite{croydon2012scaling}, where the author defined a distance between pairs of compact length spaces (for a definition of a length space see \cite[Definition 2.1.6]{burago2001course}) and continuous paths on those spaces. The restriction on length spaces is not necessary, as we will see later, on proving that $d_{\mathbb{K}}$ provides a metric. Considering c\`adl\`ag paths instead of continuous paths and replacing the uniform metric with the Skorohod metric $d_{J_1}$ allows us to prove separability without assuming that $(K,d_K)$ is a non-empty compact length space. The final term was first introduced in \cite[Section 6]{duquesne2005probabilistic} to define a distance between spatial trees equipped with a continuous function. 

\begin{proposition} \label{Prop1.1}
$(\mathbb{K},d_{\mathbb{K}})$ is a separable metric space.
\end{proposition}

\begin{proof}
That $d_{\mathbb{K}}$ is non-negative and symmetric is obvious. To prove that is also finite, for any choice of $(K,\pi^K,X^K,L^K)$, $(K',\pi^{K'},X^{K'},L^{K'})$ consider the disjoint union $Z=K\sqcup K'$ of $K$ and $K'$. Then, set $d_{Z}(x,x'):=\text{diam}_{K}(K)+\text{diam}_{K'}(K')$, for any $x\in K$, $x'\in K'$, where 
\[
\text{diam}_{K}(K)=\sup_{y,z\in K} d_K(y,z)
\]
denotes the diameter of $K$ with respect to the metric $d_K$. Since $K$ and $K'$ are compact their diameters are finite. Therefore, $d_Z$ is finite for any $x\in K$, $x'\in K'$. To conclude that $d_{\mathbb{K}}$ is finite, simply suppose that $\mathcal{C}=K\times K'$. 

Next, we show that $d_{\mathbb{K}}$ is positive-definite. Let $(K,\pi^K,X^K,L^K), (K',\pi^{K'},X^{K'},L^{K'})$ be in $\mathbb{K}$ such that $d_{\mathbb{K}}((K,\pi^K,X^K,L^K),(K',\pi^{K'},X^{K'},L^{K'}))=0$. Then, for every $\varepsilon>0$ there exist $Z,\phi,\phi',\mathcal{C}$ such that the sum of the quantities inside the infimum in the definition of $d_{\mathbb{K}}$ is bounded above by $\varepsilon$. Furthermore, there exists $\lambda_{\varepsilon}\in \Lambda$ such that the sum of the quantities inside the infimum in the definition of $d_{J_1}^Z$ is bounded above by $2 \varepsilon$. Recall that for every $t\in [0,T]$, $L^K_t:K\to \mathbb{R}_{+}$ is a continuous function and since $K$ is a compact metric space, then it is also uniformly continuous. Therefore, there exists a $\delta\in (0,\varepsilon]$ such that 
\begin{equation} \label{2.1}
\sup_{\substack{x_1,x_2\in K:\\ d_{K}(x_1,x_2)<\delta}} \sup_{t\in [0,T]} |L_t^K(x_1)-L_t^K(x_2)|\le \varepsilon.
\end{equation}
Now, let $(x_i)_{i\ge 1}$ be a dense sequence of disjoint elements in $K$. Since $K$ is compact, there exists an integer $N_{\varepsilon}$ such that the collection of open balls $(B_{K}(x_i,\delta))_{i=1}^{N_{\varepsilon}}$ covers $K$. Defining $A_1=B_{K}(x_1,\delta)$ and $A_i=B_{K}(x_i,\delta)\setminus \cup_{j=1}^{i-1} B_{K}(x_j,\delta)$, for $i=2,...,N_{\varepsilon}$, we have that $(A_i)_{i=1}^{N_{\varepsilon}}$ is a disjoint cover of $K$. Consider a function $f_{\varepsilon}:K\rightarrow K'$ by setting
\[
f_{\varepsilon}(x):=x_i'
\]
on $A_i$, where $x_i'$ is chosen such that $(x_i,x_i')\in \mathcal{C}$, for $i=1,...,N_{\varepsilon}$. Note that by definition $f_{\varepsilon}$ is a measurable function defined on $K$. For any $x\in K$, such that $x\in A_i$ for some $i=1,...,N_{\varepsilon}$, we have that 
\begin{align} \label{mine1}
d_{Z}(\phi(x),\phi'(f_{\varepsilon}(x)))&=d_{Z}(\phi(x),\phi'(x_i')) \nonumber 
\\
&\le d_{Z}(\phi(x),\phi(x_i))+d_{Z}(\phi(x_i),\phi'(x_i'))\le \delta+\varepsilon\le 2 \varepsilon.
\end{align}
From \eqref{mine1}, it follows that for any $x\in K$ and $y\in K$
\begin{align*}
|d_{Z}(\phi(x),\phi(y))-d_{Z}(\phi'(f_{\varepsilon}(x)),\phi'(f_{\varepsilon}(y))|&\le d_{Z}(\phi(y),\phi'(f_{\varepsilon}(y)))+d_{Z}(\phi(x),\phi'(f_{\varepsilon}(x)))
\\
&\le 2 \varepsilon+2 \varepsilon=4 \varepsilon.
\end{align*}
This immediately yields 
\begin{equation} \label{2.2}
\sup_{x,y\in K} |d_{K}(x,y)-d_{K'}(f_{\varepsilon}(x),f_{\varepsilon}(y))|\le 4 \varepsilon.
\end{equation}
From \eqref{2.2}, we deduce the bound
\begin{equation} \label{2.3}
d_{P}^{K'}(\pi^{K}\circ f_{\varepsilon}^{-1},\pi^{K'})\le 5 \varepsilon
\end{equation}
for the Prokhorov distance between $\pi^K\circ f_{\varepsilon}^{-1}$ and $\pi^{K'}$ in $K'$. Using \eqref{2.1} and the fact that the last quantity inside the infimum in the definition of $d_{\mathbb{K}}$ is bounded above by $\varepsilon$, we deduce 
\begin{equation} \label{2.4}
\sup_{x\in K,t\in [0,T]} |L_t^K(x)-L_t^{K'}(f_{\varepsilon}(x))|\le 2 \varepsilon.
\end{equation}
Using \eqref{mine1} and the fact that the second quantity in the infimum is bounded above by $\varepsilon$, we deduce that for any $t\in [0,T]$
\begin{align*}
d_{Z}(\phi'(f_{\varepsilon}(X^K_{\lambda_{\varepsilon}(t)})),\phi'(X_t^{K'}))&\le d_{Z}(\phi'(f_{\varepsilon}(X^K_{\lambda_{\varepsilon}(t)})),\phi(X^K_{\lambda_{\varepsilon}(t)}))+d_{Z}(\phi(X^K_{\lambda_{\varepsilon}(t)}),\phi'(X_t^{K'}))
\\
&\le 2 \varepsilon+2 \varepsilon=4 \varepsilon.
\end{align*}
Therefore,
\begin{equation} \label{Sk1}
\sup_{t\in [0,T]} d_{K'}(f_{\varepsilon}(X_{\lambda_{\varepsilon}(t)}^K),X_t^{K'})\le 4 \varepsilon.
\end{equation}
Using a diagonalization argument we can find a sequence $(\varepsilon_n)_{n\ge 1}$ such that $f_{\varepsilon_n}(x_i)$ converges to some limit $f(x_i)\in K'$, for every $i\ge 1$. From \eqref{2.2} we immediately get that $d_{K}(x_i,x_j)=d_{K'}(f(x_i),f(x_j))$, for every $i,j\ge 1$.  By \cite[Proposition 1.5.9]{burago2001course}, this map can be extended continuously to the whole $K$. This shows that $f$ is distance-preserving. Reversing the roles of $K$ and $K'$, we are able to find also a distance-preserving map from $K'$ to $K$. Hence $f$ is an isometry. We are now able to check that $\pi^K\circ f^{-1}=\pi^{K'}$, $L_t^{K'}\circ f=L_t^K$, for all $t\in [0,T]$, and $f\circ X^K=X^{K'}$. Since $f_{\varepsilon_n}(x_i)$ converges to $f(x_i)$ in $K'$, we can find $\varepsilon'\in (0,\varepsilon]$ such that $d_{K'}(f_{\varepsilon'}(x_i),f(x_i))\le \varepsilon$, for $i=1,...,N_{\varepsilon}$. Recall that $(x_i)_{i=1}^{N_{\varepsilon}}$ is an $\varepsilon$-net in $K$. Then, for $i=1,...,N_{\varepsilon}$, such that $x\in A_i$, using \eqref{2.2} and the fact that $f$ is an isometry, we deduce
\begin{equation} \label{2.5}
d_{K'}(f_{\varepsilon'}(x),f(x))\le d_{K'}(f_{\varepsilon'}(x),f_{\varepsilon'}(x_i))+d_{K'}(f_{\varepsilon'}(x_i),f(x_i))+d_{K'}(f(x_i),f(x))\le 7 \varepsilon.
\end{equation}
This, combined with \eqref{2.3} implies 
\[
d_{P}^{K'}(\pi^K\circ f^{-1},\pi^{K'})\le  d_{P}^{K'}(\pi^K\circ f^{-1},\pi^{K}\circ f_{\varepsilon'}^{-1})+d_{P}^{K'}(\pi^K\circ f_{\varepsilon'}^{-1},\pi^{K'})\le 12 \varepsilon.
\]
Since $\varepsilon>0$ was arbitrary, $\pi^{K}\circ f^{-1}=\pi^{K'}$. Moreover, from \eqref{2.4} and \eqref{2.5} we have that 
\begin{align*}
&\sup_{x\in K,t\in [0,T]} |L_t^K(x)-L_t^{K'}(f(x))|
\\
&\le \sup_{x\in K,t\in [0,T]} |L_t^{K}(x)-L_t^{K'}(f_{\varepsilon'}(x))|+\sup_{x\in K,t\in [0,T]} |L_t^{K'}(f_{\varepsilon'}(x))-L_t^{K'}(f(x))|
\\
&\le 2 \varepsilon+\sup_{\substack{x_1',x_2'\in K':\\ d_{K'}(x_1',x_2')\le 7\varepsilon}} \sup_{t\in [0,T]} |L_t^{K'}(x_1')-L_t^{K'}(x_2')|.
\end{align*}
Now, this and the uniform continuity of $L^{K'}$ (replace $L^K$ by $L^{K'}$ in \eqref{2.1}) gives $L_t^{K'}\circ f=L_t^K$, for all $t\in [0,T]$. Finally, we verify that $f\circ X^K=X^{K'}$. For any $t\in [0,T]$
\[
d_{K'}(f(X_{\lambda_{\varepsilon}(t)}^K),X_t^{K'})\le d_{K'}(f(X_{\lambda_{\varepsilon}(t)}^K),f_{\varepsilon'}(X_{\lambda_{\varepsilon}(t)}^K))+d_{K'}(f_{\varepsilon'}(X_{\lambda_{\varepsilon}(t)}^K),X_t^{K'})\le 7 \varepsilon+4 \varepsilon=11 \varepsilon, 
\]
where we used \eqref{Sk1} and \eqref{2.5}. Therefore,
\begin{equation} \label{Sk2}
\sup_{t\in [0,T]} d_{K'}(f(X_{\lambda_{\varepsilon}(t)}^K),X_t^{K'})\le 11 \varepsilon.
\end{equation}
Recall that $\sup_{t\in [0,T]} |\lambda_{\varepsilon}(t)-t|\le 2 \varepsilon$. From this and \eqref{Sk2}, it follows that for every $t\in [0,T]$, there exists a sequence $(z_n)_{n\ge 1}$, such that $z_n\rightarrow t$ and $d_{K'}(f(X_{z_n}^{K}),X_t^{K'})\rightarrow 0$, as $n\rightarrow \infty$. If $t$ is a continuity point  of $f\circ X^K$, then $d_{K'}(f(X_{z_n}^{K}),f(X_t^K))\rightarrow 0$, as $n\rightarrow \infty$. Thus, $f(X_t^K)=X_t^{K'}$. If $f\circ X^K$ has a jump at $t$ and $(z_n)_{n\ge 1}$ has a subsequence $(z_{n_k})_{k\ge 1}$, such that $z_{n_k}\ge t$, for any $k\ge 1$, then $d_{K'}(f(X_{z_{n_k}}^{K}),X_t^{K'})\rightarrow 0$, as $n\rightarrow \infty$, and $d_{K'}(f(X_{z_{n_k}}^{K}),f(X_t^K))\rightarrow 0$, as $n\rightarrow \infty$. Therefore, $f(X_t^K)=X_t^{K'}$. Otherwise, $z_n<t$, for $n$ large enough and $d_{K'}(f(X_{z_n}^{K}),f(X_{t-}^{K}))\rightarrow 0$, as $n\rightarrow \infty$, which implies $f(X_{t-}^{K})=X_t^{K'}$. Essentially, what we have proved is that if $f\circ X^K$ has a jump, then either $f(X^K_t)=X_t^{K'}$ or $f(X^K_{t-})=X_t^{K'}$. But, since $X^{K'}$ is c\`adl\`ag, $f\circ X^K=X^{K'}$. This completes the proof that the quadruples $(K,\pi^K,X^K,L^K)$ and $(K',\pi^{K'},X^{K'},L^{K'})$ are equivalent in $(\mathbb{K},d_{\mathbb{K}})$, and consequently that $d_{\mathbb{K}}$ is positive-definite. 

For the triangle inequality we follow the proof of \cite[Proposition 7.3.16]{burago2001course}, which proves the triangle inequality for the standard Gromov-Hausdorff distance. Let $\mathcal{K}^i=(K^i,\pi^i,X^i,L^i)$ be an element of $(\mathbb{K},d_{\mathbb{K}})$ for $i=1,2,3$. Suppose that 
\[
d_{\mathbb{K}}(\mathcal{K}^1,\mathcal{K}^2)<\delta_1.
\]
Thus, there exists a metric space $Z_1$, isometric embeddings $\phi_{1,1}: K^1\rightarrow Z_1$, $\phi_{2,1}:K^2\rightarrow Z_1$ and a correspondence $\mathcal{C}_1$ between $K^1$ and $K^2$ such that the sum of the quantities inside the infimum that defines $d_{\mathbb{K}}$ is bounded above by $\delta_1$. Similarly, if
\[
d_{\mathbb{K}}(\mathcal{K}^2,\mathcal{K}^3)<\delta_2,
\]
there exists a metric space $Z_2$, isometric embeddings $\phi_{2,2}: K^2\rightarrow Z_2$, $\phi_{2,3}: K^3\rightarrow Z_2$ and a correspondence $\mathcal{C}_2$ between $K^2$ and $K^3$ such that the sum of the quantities inside the infimum that defines $d_{\mathbb{K}}$ is bounded above by $\delta_2$. Next, we set $Z=Z_1\sqcup Z_2$ to be the disjoint union of $Z_1$ and $Z_2$ and we define a distance on $Z$ in the following way. Let $d_{Z|Z_i\times Z_i}=d_{Z_i}$, for $i=1,2$, and for $x\in Z_1$, $y\in Z_2$ set 
\[
d_{Z}(x,y):=\inf_{z\in K^2} \{d_{Z_1}(x,\phi_{2,1}(z))+d_{Z_2}(\phi_{2,2}(z),y)\}.
\]
It is obvious that $d_Z$ is symmetric and non-negative. It is also easy to check that $d_Z$ satisfies the triangle inequality. Identifying points that are separated by zero distance and slightly abusing notation, we turn $(Z,d_{Z})$ into a metric space, which comes with isometric embeddings $\phi_i$ of $Z_i$ for $i=1,2$. Using the triangle inequality of the Prokhorov metric on $Z$, gives us that  
\[
d_P^Z(\pi^1\circ (\phi_1\circ \phi_{1,1})^{-1},\pi^3\circ (\phi_2\circ \phi_{3,2})^{-1})
\]
\[
\le d_P^Z(\pi^1\circ (\phi_1\circ \phi_{1,1})^{-1},\pi^2\circ (\phi_1\circ \phi_{2,1})^{-1})+d_P^Z(\pi^2\circ (\phi_1\circ \phi_{2,1})^{-1},\pi^3\circ (\phi_2\circ \phi_{3,2})^{-1}).
\]
Now, since $\phi_1(\phi_{2,1}(y))=\phi_2(\phi_{2,2}(y))$, for all $y\in K^2$, we deduce 
\begin{equation} \label{tria1}
d_P^Z(\pi^1\circ (\phi_1\circ \phi_{1,1})^{-1},\pi^3\circ (\phi_2\circ \phi_{3,2})^{-1}) \le d_P^{Z_1}(\pi^1\circ \phi_{1,1}^{-1},\pi^2\circ \phi_{2,1}^{-1})+d_P^{Z_2}(\pi^2\circ \phi_{2,2}^{-1},\pi^3\circ \phi_{3,2}^{-1}).
\end{equation}
A similar bound also applies for the embedded c\`adl\`ag paths. Namely, using the same methods as above, we deduce
\begin{equation} \label{tria2}
d_{J_1}^Z((\phi_1\circ \phi_{1,1})(X^1),(\phi_2\circ \phi_{3,2})(X^3))\le d_{J_1}^Z(\phi_{1,1}(X^1),\phi_{2,1}(X^2))+d_{J_1}^Z(\phi_{2,2}(X^2),\phi_{3,2}(X^3)).
\end{equation}
Now, let
\[
\mathcal{C}:=\{(x,z)\in K^1\times K^3: (x,y)\in \mathcal{C}_1,(y,z)\in \mathcal{C}_2,\text{ for some }y\in K^2\}.
\]
Observe that $\mathcal{C}$ is a correspondence between $K^1$ and $K^3$. Then, if $(x,z)\in \mathcal{C}$, there exists $y\in K^2$ such that $(x,y)\in \mathcal{C}_1$ and $(y,z)\in \mathcal{C}_2$, and noting again that $\phi_1(\phi_{2,1}(y))=\phi_2(\phi_{2,2}(y))$, for all $y\in K^2$, we deduce 
\begin{equation} \label{tria3}
d_Z(\phi_1(\phi_{1,1}(x)),\phi_2(\phi_{3,2}(z)))\le d_{Z_1}(\phi_{1,1}(x),\phi_{2,1}(y))+d_{Z_2}(\phi_{2,2}(y),\phi_{3,2}(z)).
\end{equation}
Using the same arguments one can prove a corresponding bound involving $L^i$, $i=1,2,3$. Namely, if $(x,z)\in \mathcal{C}$, there exists $y\in K^2$ such that $(x,y)\in \mathcal{C}_1$ and $(y,z)\in \mathcal{C}_2$, and moreover
\begin{equation} \label{tria4}
\sup_{t\in [0,T]}|L_t^1(x)-L_t^3(z)|\le \sup_{t\in [0,T]} |L_t^1(x)-L_t^2(y)|+\sup_{t\in [0,T]} |L_t^2(y)-L_t^3(z)|.
\end{equation}
Putting \eqref{tria1}, \eqref{tria2}, \eqref{tria3} and \eqref{tria4} together gives
\[
d_{\mathbb{K}}(\mathcal{K}^1,\mathcal{K}^3)\le \delta_1+\delta_2,
\]
and the triangle inequality follows. Thus, $(\mathbb{K},d_{\mathbb{K}})$ forms a metric space.

To finish the proof, we need to show that $(\mathbb{K},d_{\mathbb{K}})$ is separable. Let $(K,\pi,X,L)$ be an element of $\mathbb{K}$. First, let $K^n$ be a finite $n^{-1}$-net of $K$, which exists since $K$ is compact. Furthermore, we can endow $K^n$ with a metric $d_{K^n}$, such that $d_{K^n}(x,y)\in \mathbb{Q}$, and moreover $|d_{K^n}(x,y)-d_K(x,y)|\le n^{-1}$, for every $x,y \in K^n$. Since, $K^n$ is a finite $n^{-1}$-net of $K$ we can choose a partition for $K$, $(A_x)_{x\in K^n}$, such that $x\in A_x$, and $\text{diam}_{K}(A_x)\le 2 n^{-1}$. We can even choose the partition in such a way that $A_x$ is measurable for all $x\in K^n$ (see for example the definition of $(A_i)_{i=1}^{N_{\varepsilon}}$ after \eqref{2.1}). Next, we construct a Borel measure $\pi^n$ in $K^n$ that takes rational mass at each point, i.e. $\pi^n(\{x\})\in \mathbb{Q}$, and $|\pi^n(\{x\})-\pi(A_x)|\le n^{-1}$. Define $\varepsilon_n$ by
\[
\varepsilon_n:=\sup_{\substack{s,t\in [0,T]:\\ |s-t|\le n^{-1}}} \sup_{\substack{x,x'\in K:\\ d_{K}(x,x')\le n^{-1}}} |L_s(x)-L_t(x')|.
\]
By the joint continuity of $L$, $\varepsilon_n\rightarrow 0$, as $n\rightarrow \infty$. Let $0=s_0<s_1<\cdot \cdot \cdot <s_r=T$ be a set of rational times such that $|s_{i+1}-s_i|\le n^{-1}$, for $i=0,...,r-1$. Choose $L_{s_i}^n(x)\in \mathbb{Q}$ with $|L_{s_i}^n(x)-L_{s_i}(x)|\le n^{-1}$, for every $x\in K^n$. We interpolate linearly between the finite collection of rational time points in order to define $L^n$ to the whole domain $K^n\times [0,T]$. Let $\mathcal{C}^n:=\{(x,x')\in K\times K^n: d_K(x,x')\le n^{-1}\}$. Clearly $\mathcal{C}^n$ defines a correspondence between $K$ and $K^n$. Let $(x,x')\in \mathcal{C}^n$ and $s\in [s_i,s_{i+1}]$, for some $i=0,...,r-1$. Then, using the triangle inequality we observe that
\begin{equation} \label{correction1}
|L_s^n(x)-L_s(x')|\le |L_s^n(x)-L_s(x)|+|L_s(x)-L_s(x')|\le |L_s^n(x)-L_s(x)|+\varepsilon_n.
\end{equation}
Since we interpolated linearly to define $L^n$ beyond rational time points on the whole space $K^n\times [0,T]$ we have that
\begin{equation} \label{correction2}
|L_s^n(x)-L_s(x)|\le |L_{s_{i+1}}^n(x)-L_s(x)|+|L_{s_i}^n(x)-L_s(x)|.
\end{equation}
Applying the triangle inequality again yields
\begin{align*}
|L_{s_i}^n(x)-L_s(x)|&\le |L_{s_i}^n(x)-L_{s_i}(x)|+|L_{s_i}(x)-L_s(x)|
\\
&\le n^{-1}+\varepsilon_n.
\end{align*} 
The same upper bound applies for $|L_{s_{i+1}}^n(x)-L_s(x)|$, and from \eqref{correction1} and \eqref{correction2} we conclude that for $(x,x')\in \mathcal{C}^n$ and $s\in [s_i,s_{i+1}]$, for some $i=0,...,r-1$,
\[
|L_s^n(x)-L_s(x')|\le 2 n^{-1}+3 \varepsilon_n.
\]
For $X\in D([0,T],K)$ and $A\subseteq [0,T]$ put
\[
w(X;A):=\sup_{s,t\in A} d_K(X_t,X_s).
\]
Now, for $\delta\in (0,1)$, define the c\`adl\`ag modulus to be 
\[
w'(X;\delta):=\inf_{\Sigma} \max_{1\le i\le k} w(X;[t_{i-1},t_i)),
\]
where the infimum is taken over all partitions $\Sigma=\{0=t_0<t_1<\cdot \cdot \cdot <t_k=T\}$, $k\in \mathbb{N}$, with $\min_{1\le i\le k}(t_i-t_{i-1})>\delta$.  For a function to lie in $D([0,T],K)$, it is necessary and sufficient to satisfy $w'(X;\delta)\to 0$, as $\delta\to 0$. Let $B_n$ be the set of functions having a constant value in $K^n$ over each interval $[(u-1)T/n,uT/n)$, for some $n\in \mathbb{N}$ and also a value in $K^n$ at time $T$. Take $B=\cup_{n\ge 1} B_n$, and observe that is countable. Clearly, putting $z=(z_u)_{u=0}^{n}$, with $z_u=uT/n$, for every $u=0,...,n$ satisfies $0=z_0<z_1<\cdot \cdot \cdot <z_n=T$. Let $T_z: D([0,T],K)\to D([0,T],K)$ be the map that is defined in the following way. For $X\in D([0,T],K)$ take $T_zX$ to have a constant value $X(z_{u-1})$ over the interval $[z_{u-1},z_u)$ for $1\le u\le n$ and the value $X(T)$ at $t=T$. From an adaptation of \cite[Lemma 3, p.127]{billingsley2013convergence}, considering c\`adl\`ag paths that take values on metric spaces, we have that 
\begin{equation} \label{separ1}
d_{J_1}(T_zX,X)\le T n^{-1}+w'(X;T n^{-1}).
\end{equation} 
Also, there exists $X^n\in B_n$, for which
\begin{equation} \label{separ2}
d_{J_1}(T_zX,X^n)\le T n^{-1}.
\end{equation}
Combining \eqref{separ1} and \eqref{separ2}, we have that 
\[
d_{J_1}(X^n,X)\le d_{J_1}(X^n,T_zX)+d_{J_1}(T_zX,X)\le 2 T n^{-1}+w'(X;T n^{-1}).
\]
With the choice of the sequence $(K^n,\pi^n,X^n,L^n)$, we find that 
\[
d_{\mathbb{K}}((K^n,\pi^n,X^n,L^n),(K,\pi,X,L))\le (4 +2 T) n^{-1}+3 \varepsilon_n+w'(X;T n^{-1}).
\] 
Recalling that $w'(X;T n^{-1})\to 0$, as $n\to \infty$, and noting that our sequence was drawn from a countable subset of $\mathbb{K}$ completes the proof of the proposition.

\end{proof}

Fix $T>0$. Let $\tilde{\mathbb{K}}$ be the space of quadruples of the form $(K,\pi^K,X^K,L^K)$, where $K$ is a non-empty compact pointed metric space with distinguished vertex $\rho$, $\pi^K$ is a Borel measure on $K$, $X^K=(X^K_t)_{t\in [0,K]}$ is a c\`adl\`ag path on $K$ and $L^K=(L_t(x))_{x\in K,t\in [0,T]}$ is a jointly continuous positive real-valued function of $(t,x)$. We say that two elements of $\tilde{\mathbb{K}}$, say $(K,\pi^K,X^K,L^K)$ and $(K',\pi^{K'},X^{K'},L^{K'})$, are equivalent if and only there is a root-preserving isometry $f:K\to K'$, such that $f(\rho)=\rho'$, $\pi^K\circ f^{-1}=\pi^{K'}$, $f\circ X^K=X^{K'}$ and $L_t^{K'}\circ f=L_t^K$, for every $t\in [0,T]$. It is possible to define a metric on the equivalence classes of $\tilde{\mathbb{K}}$ by imposing in the definition of $d_{\mathbb{K}}$ that the infimum is taken over all correspondences that contain $(\rho,\rho')$. The incorporation of distinguished points to the extended Gromov-Hausdorff topology leaves the proof of Proposition \ref{Prop1.1} unchanged and it is possible to show that $(\tilde{\mathbb{K}},d_{\tilde{\mathbb{K}}})$ is a separable metric space.

The aim of the following lemmas is to establish a sufficient condition for Assumption \ref{Assum1} to hold, as well as to show that if Assumption \ref{Assum1} holds then we can isometrically embed the rescaled graphs, measures, random walks and local times into a common metric space such that they all converge to the relevant objects. To be more precise we formulate this last statement in the next lemma.

\begin{lemma} \label{lem2.2}
If Assumption \ref{Assum1} is satisfied, then we can find isometric embeddings of $(V(G^n),d_{G^n})_{n\ge 1}$ and $(K,d_K)$ into a common metric space $(F,d_F)$ such that 
\begin{equation} \label{extra}
\lim_{n\rightarrow \infty} d_H^F(V(G^n),K)=0, \qquad \lim_{n\to \infty} d_F(\rho^n,\rho)=0,
\end{equation}
where $d_H^F$ is the standard Hausdorff distance between $V(G^n)$ and $K$, regarded as subsets of $(F,d_F)$,
\begin{equation} \label{2.6}
\lim_{n\rightarrow \infty} d_P^F(\pi^n,\pi)=0,
\end{equation}
where $d_P^F$ is the standard Prokhorov distance between $V(G^n)$ and $K$, regarded as subsets of $(F,d_F)$,
\begin{equation} \label{2.7}
\lim_{n\rightarrow \infty} d_{J_1}^F(X^n,X)=0,
\end{equation}
where $d_{J_1}^F$ is the Skorohod $d_{J_1}$ metric between $V(G^n)$ and $K$, regarded as subsets of $(F,d_F)$. Also,
\begin{equation} \label{2.8}
\lim_{\delta\rightarrow 0} \limsup_{n\rightarrow \infty} \sup_{\substack{x^n\in V(G^n),x\in K: \\ d_F(x^n,x)<\delta}} \sup_{t\in [0,T]} |L^n_{\beta(n) t}(x^n)-L_t(x)|=0.
\end{equation}
For simplicity we have identified the measures and the random walks in $V(G^n)$ with their isometric embeddings in $(F,d_F)$.
\end{lemma}

\begin{proof}
Since Assumption \ref{Assum1} holds, for each $n\ge 1$ we can find metric spaces $(F_n,d_n)$, isometric embeddings $\phi_n:V(G^n)\rightarrow F_n$, $\phi_n':K\rightarrow F_n$ and correspondences $\mathcal{C}^n$ (that contain $(\rho^n,\rho)$) between $V(G^n)$ and $K$ such that (identifying the relevant objects with their embeddings)
\begin{align} \label{2.9}
d_{P}^{F_n}(\pi^n,\pi)+d_{J_1}^{F_n}(X^n,X)+\sup_{(x,x')\in \mathcal{C}^n}\bigg (d_n(x,x')+\sup_{t\in [0,T]}|L^n_{\beta(n) t}(x)-L_t(x')|\bigg )\le \varepsilon_n,
\end{align}
where $\varepsilon_n\rightarrow 0$, as $n\rightarrow \infty$. Now, let $F=\sqcup_{n\ge 1} F_n$, be the disjoint union of $F_n$, and define the distance $d_{F|F_n\times F_n}=d_{n}$, for $n\ge 1$, and for $x\in F_n$, $x'\in F_{n'}$, $n\neq n'$
\[
d_F(x,x'):=\inf_{y\in K}\{d_n(x,y)+d_{n'}(y,x')\}.
\]
This distance, as the distance that was defined in order to prove the triangle inequality in Proposition \ref{Prop1.1}, is symmetric and non-negative, so identifying points that are separated by a zero distance, we turn $(F,d_F)$ into a metric space, which comes with natural isometric embeddings of $(V(G^n),d_{G^n})_{n\ge 1}$ and $(K,d_K)$. In this setting, under the appropriate isometric embeddings \eqref{extra}, \eqref{2.6} and \eqref{2.7} readily hold from \eqref{2.9}. Thus, it only remains to prove \eqref{2.8}. For every $x\in V(G^n)$, since $\mathcal{C}^n$ is a correspondence in $V(G^n)\times K$, there exists an $x'\in K$ such that $(x,x')\in \mathcal{C}^n$. Then, \eqref{2.9} implies that $d_F(x,x')\le \varepsilon_n$. Now, let $(y,y')\in \mathcal{C}^n$, $(z,z')\in \mathcal{C}^n$ and note that 
\begin{align*}
&\sup_{t\in [0,T]} |L_{\beta(n) t}^n(y)-L_{\beta(n) t}^n(z)|\\
&\le \sup_{t\in [0,T]} |L_{\beta(n) t}^n(y)-L_t(y')|+\sup_{t\in [0,T]} |L_{\beta(n) t}^n(z)-L_t(z')|+\sup_{t\in [0,T]} |L_t(y')-L_t(z')|\\
&\le 2 \varepsilon_n+\sup_{t\in [0,T]} |L_t(y')-L_t(z')|.
\end{align*}
For any $\delta>0$ and $y, z\in V(G^n)$, such that $d_{G^n}(y,z)<\delta$, we have that 
\[
d_K(y',z')\le d_F(y,y')+d_F(z,z')+d_{G^n}(y,z)<2 \varepsilon_n+\delta.
\]
Therefore,
\begin{align} \label{2.10}
&\sup_{\substack{y,z\in V(G^n): \\ d_{G^n}(y,z)<\delta}} \sup_{t\in [0,T]} |L_{\beta(n) t}^n(y)-L_{\beta(n) t}^n(z)|\nonumber \\ 
&\le 2 \varepsilon_n+\sup_{\substack{y,z\in K: \\ d_K(y,z)< 2 \varepsilon_n+\delta}} \sup_{t\in [0,T]} |L_t(y)-L_t(z)|.
\end{align}
Also, for every $x\in K$ there exists an $x'\in V(G^n)$ such that $(x',x)\in \mathcal{C}^n$, and furthermore $d_F(x',x)\le \varepsilon_n$. Let $x^n\in V(G^n)$ such that $d_F(x^n,x)<\delta$. Then, 
\[
d_F(x^n,x')\le d_F(x^n,x)+d_F(x',x)<2 \epsilon_n+\delta.
\]
More generally, if we denote by $B_F(x,r)$, the open balls of radius $r$, centered in $x$, we have the following inclusion
\[
B_F(x,\delta)\cap V(G^n)\subseteq B_F(x',2 \varepsilon_n+\delta)\cap V(G^n).
\]
For $x\in K$, and $x'\in V(G^n)$ with $d_F(x',x)\le \varepsilon_n$, using \eqref{2.9}, we deduce
\begin{align*}
&\sup_{t\in [0,T]} |L_{\beta(n) t}^n(x^n)-L_t(x)|\\
&\le \sup_{t\in [0,T]} |L_{\beta(n) t}^n(x^n)-L_{\beta(n) t}^n(x')|+\sup_{t\in [0,T]} |L_{\beta(n) t}^n(x')-L_t(x)|\\
&\le \varepsilon_n+\sup_{t\in [0,T]} |L_{\beta(n) t}^n(x^n)-L_{\beta(n) t}^n(x')|.
\end{align*}
Since $x^n \in B_F(x',2 \varepsilon_n+\delta)\cap V(G^n)$, taking the supremum over all $x^n\in V(G^n)$ and $x\in K$, for which $d_F(x^n,x)<\delta$ and using \eqref{2.10}, we deduce 
\begin{align*}
&\sup_{\substack{x^n\in V(G^n),x\in K: \\ d_F(x^n,x)<\delta}} \sup_{t\in [0,T]} |L_{\beta(n) t}^n(x^n)-L_t(x)|\\
&\le \varepsilon_n+\sup_{\substack{y,z\in V(G^n): \\ d_{G^n}(y,z)<2 \varepsilon_n+\delta}} \sup_{t\in [0,T]} |L_{\beta(n) t}^n(y)-L_{\beta(n) t}^n(z)|\\
&\le 3 \varepsilon_n+\sup_{\substack{y,z\in K: \\ d_K(y,z)< 4 \varepsilon_n+\delta}} \sup_{t\in [0,T]} |L_t(y)-L_t(z)|.
\end{align*}
Using the continuity of $L$, as $n\rightarrow \infty$
\begin{equation} \label{explicit}
\limsup_{n\rightarrow \infty} \sup_{\substack{x^n\in V(G^n),x\in K: \\ d_F(x^n,x)<\delta}} \sup_{t\in [0,T]} |L_{\beta(n) t}^n(x^n)-L_t(x)|\le \sup_{\substack{y,z\in K: \\ d_K(y,z)\le \delta}} \sup_{t\in [0,T]} |L_t(y)-L_t(z)|.
\end{equation}
Again appealing to the continuity of $L$, the right-hand side converges to 0, as $\delta\rightarrow 0$. Thus, we showed that \eqref{2.8} holds, and this finishes the proof of Lemma \ref{lem2.2}.

\end{proof}

In the process of proving \eqref{2.8} we established a useful equicontinuity property. We state and prove this property in the next corollary.

\begin{corollary}
Fix $T>0$ and suppose that Assumption \ref{Assum1} holds. Then, 
\begin{equation} \label{2.11}
\lim_{\delta\rightarrow 0} \limsup_{n\rightarrow \infty} \sup_{\substack{y,z\in V(G^n): \\ d_{G^n}(y,z)<\delta}} \sup_{t\in [0,T]} |L_{\beta(n) t}^n(y)-L_{\beta(n) t}^n(z)|=0. 
\end{equation}
\end{corollary}

\begin{proof}
As we hinted upon when deriving \eqref{explicit}, using the continuity of $L$,
\[
\limsup_{n\rightarrow \infty} \sup_{\substack{y,z\in V(G^n): \\ d_{G^n}(y,z)<\delta}} \sup_{t\in [0,T]} |L_{\beta(n) t}^n(y)-L_{\beta(n) t}^n(z)|\le \sup_{\substack{y,z\in K: \\ d_K(y,z)\le \delta}} \sup_{t\in [0,T]} |L_t(y)-L_t(z)|.
\]
Sending $\delta\rightarrow 0$ gives the desired result.

\end{proof}

Next, we prove that if we reverse the conclusions of Lemma \ref{lem2.2}, more specifically if  \eqref{extra}-\eqref{2.8} hold, then also Assumption \ref{Assum1} holds.

\begin{lemma} \label{lem2.3}
Suppose that \eqref{extra}-\eqref{2.8} hold. Then so does Assumption \ref{Assum1}.
\end{lemma}

\begin{proof}
There exist isometric embeddings of $(V(G^n),d_{G^n})_{n\ge 1}$ and $(K,d_K)$ into a common metric space $(F,d_F)$, under which the assumptions \eqref{extra}-\eqref{2.8} hold. Since \eqref{extra} gives the convergence of spaces under the Hausdorff metric, \eqref{2.6} gives the convergence of measures under the Prokhorov metric and \eqref{2.7} gives the convergence of paths under $d_{J_1}$, it only remains to check the uniform convergence of local times. Let $\mathcal{C}^n$ be the set of all pairs $(x,x')\in K\times V(G^n)$, for which $d_F(x,x')\le n^{-1}$. Since \eqref{extra} holds, $\mathcal{C}^n$ are correspondences for $n\ge 1$. Then, for $(x,x')\in \mathcal{C}^n$
\[
\sup_{t\in [0,T]} |L_{\beta(n) t}^n(x')-L_t(x)|\le \sup_{\substack{x^n\in V(G^n),x\in K: \\ d_F(x^n,x)<n^{-1}}} \sup_{t\in [0,T]} |L^n_{\beta(n) t}(x^n)-L_t(x)|,
\]
and using \eqref{2.8} completes the proof.

\end{proof}

\section{Blanket time-scaling and distributional bounds} \label{blankbounds}

In this section, we show that under Assumption \ref{Assum1}, and as a consequence of the local time convergence in Lemma \ref{emblem}, we are able to establish asymptotic bounds on the distribution of the blanket times of the graphs in the sequence. The argument for the cover time-scaling was provided first in \cite[Corollary 7.3]{croydon2015moduli} by restricting to the unweighted Sierpi\'nski gasket graphs. The argument is applicable to any other model as long as the relevant assumptions are satisfied.

\subsection{Proof of Theorem \ref{Mth}} \label{Sec2.2}

First, let us check that the $\varepsilon$-blanket time variable of $K$ as written in \eqref{blacken1} is well-defined.   

\begin{proposition} \label{nontrivexp}
Fix $\varepsilon\in (0,1)$. For every $x\in K$, $\mathbf{P}_x$-a.s. we have that $\tau_{\textnormal{bl}}(\varepsilon)\in (0,\infty)$.
\end{proposition}

\begin{proof}
Fix $x,y\in K$. There is a strictly positive $\mathbf{P}_x$-probability that $L_t(x)>0$ for $t$ large enough, which is a consequence of \cite[Lemma 3.6]{marcus1992sample}. From the joint continuity of $(L_t(x))_{x\in K,t\ge 0}$, there exist $r\equiv r(x)$, $\delta\equiv \delta(x)>0$ and $t_{*}\equiv t_{*}(x)<\infty$ such that 
\begin{equation} \label{stem1}
\mathbf{P}_x\left(\inf_{z\in B(x,r)} L_{t_{*}}(z)>\delta\right)>0.
\end{equation}
Now, let $\tau_{x,y}(t_{*}):=\inf\{t>t_{*}+\tau_x: X_t=y\}$, where $\tau_x:=\inf \{t>0: X_t=x\}$ is the hitting time of $x$. In other words, $\tau_{x,y}(t_{*})$ is the first hitting of $y$ by $X$ after time $t_{*}+\tau_x$. The commute time identity for a resistance derived in the proof of \cite[Lemma 2.9]{croydon2016time} yields that $\mathbf{E}_x \tau_y+\mathbf{E}_y \tau_x=R(x,y) \pi(K)$, for every $x,y\in K$, which in turn deduces that $\mathbf{E}_x \tau_y<\infty$. Applying this observation about the finite first moments of hitting times, one can check that $\tau_{x,y}(t_{*})<\infty$, $\mathbf{P}_y$-a.s., and also that
\[
\mathbf{P}_y\left(\inf_{z\in B(x,r)} L_{\tau_{x,y}(t_{*})}(z)>\delta\right)>0.
\]
This simply follows from an application of \eqref{stem1} and the Strong Markov property. The additivity of local times and the Strong Markov property implies that 
\begin{equation} \label{insecure}
\liminf_{t\to \infty} \inf_{z\in B(x,r)} \frac{L_t(z)}{t}\ge \left(\sum_{i=1}^{\infty}\xi_i^1\right) \left(\sum_{i=1}^{\infty} \xi_i^2\right)^{-1},
\end{equation}
where $(\xi_i^1)_{i\ge 1}$ are independent random variables distributed as $\inf_{z\in B(x,r)} L_{\tau_{x,y}(t_{*})}(z)$
and $(\xi_i^2)_{i\ge 1}$ are independent random variables distributed as $\tau_{x,y}(t_{*})$. The strong law of large numbers deduces that the right-hand side of \eqref{insecure} converges to $\mathbf{E}_y \left(\inf_{z\in B(x,r)} L_{\tau_{x,y}(t_{*})}(z)\right) (\mathbf{E}_y \tau_{x,y}(t_{*}))^{-1}$, $\mathbf{P}_x$-a.s. Using the occupation density formula and the commute time identity for a resistance, observe that 
\begin{align*}
\mathbf{E}_y L_{\tau_{x,y}(t_{*})}(x) \cdot (\mathbf{E}_y \tau_{x,y}(t_{*}))^{-1}=\mathbf{E}_x L_{\tau_y}(x) \cdot (\mathbf{E}_y \tau_{x,y}(t_{*}))^{-1}
=R(x,y)\cdot (R(x,y) \pi(K))^{-1}=\pi(K)^{-1},
\end{align*}
and therefore, fixing some $\varepsilon_{*}\in (\varepsilon,1)$, the joint continuity of the local time process lets it to be deduced that
\begin{equation} \label{stem2}
\liminf_{t\to \infty} \inf_{z\in B(x,r)} \frac{L_t(z)}{t}\ge \varepsilon_{*},
\end{equation}
$\mathbf{P}_x$-a.s. To extend this statement that holds uniformly over $B(x,r)$, $\mathbf{P}_x$-a.s., we use the compactness of $K$. Consider the open cover $(B(x,r))_{x_i\in K}$ the open cover for $K$, which say admits a finite subcover $(B(x_i,r_i))_{i=1}^{N}$. Since the right-hand of \eqref{stem2} is greater than $\varepsilon_{*}$, $\mathbf{P}_x$-a.s., the result clearly follows as
\[
\lim_{t\to \infty} \frac{L_t(x)}{t}=\min_{1\le i\le N} \lim_{t\to \infty} \inf_{x\in B(x_i,r_i)} \frac{L_t(x)}{t}.
\]  

\end{proof}

We are now ready to prove one of our main results.

\begin{proof}[Proof of Theorem \ref{Mth}]
Let $\varepsilon\in (0,1)$, $\delta\in (0,1)$ and $t\in [0,T]$. Suppose that $t<{\tau}_{\textnormal{bl}}(\varepsilon(1-\delta))$. Then, there exists a $y\in K$ for which $L_t(y)<\varepsilon (1-\delta) t$. Using the Skorohod representation theorem, we can suppose that the conclusions of Lemma \ref{lem2.2} holds in an almost-surely sense. From \eqref{extra}, there exists $y^n\in V(G^n)$ such that, for $n$ large enough, $d_F(y^n,y)<2 \varepsilon$. Then, the local convergence at \eqref{2.8} implies that, for $n$ large enough
\[
\alpha(n) L_{\beta(n) t}^n(y^n)\le L_t(y)+\varepsilon \delta t.
\]
Thus, for $n$ large enough, it follows that $\alpha(n) L_{\beta(n) t}^n(y^n)\le L_t(y)+\varepsilon \delta t< \varepsilon t$. Using the time and space scaling identity, we deduce $m^{G^n} L^n_{\beta(n) t}(y^n)<\varepsilon \beta(n) t$, for $n$ large enough, which in turn implies that $\beta(n) t\le \tau_{\text{bl}}^n(\varepsilon)$, for $n$ large enough. As a consequence, we get that $\tau_{\text{bl}}(\varepsilon (1-\delta))\le \liminf_{n\rightarrow \infty} \beta(n)^{-1} \tau_{\text{bl}}^n(\varepsilon)$, which proves \eqref{blacken2}.

Assume now that $\tau_{\text{bl}}(\varepsilon(1+\delta))<t$. Then, for some $\tau_{\text{bl}}(\varepsilon(1+\delta))\le t_0<t$, it is the case that $L_{t_0}(x)\ge \varepsilon(1+\delta) t_0$, for every $y\in K$. As in the previous paragraph, using the Skorohod represantation theorem, we suppose that the conclusions of Lemma \ref{lem2.2} holds almost-surely. From \eqref{extra}, for every $y^n\in V(G^n)$, there exists a $y\in K$ such that, for $n$ large enough, $d_F(y^n,y)<2 \varepsilon$. From the local convergence statement at \eqref{2.8}, we have that, for $n$ large enough
\[
\alpha(n) L_{\beta(n) t_0}^n(y^n)\ge L_{t_0}(y)-\varepsilon \delta t_0.
\] 
Therefore, for $n$ large enough, it follows that $\alpha(n) L^n_{\beta(n) t_0}(y^n)\ge L_{t_0}(y)-\varepsilon \delta t_0\ge \varepsilon t_0$, for every $y\in K$. As before, using the time and space scaling identity yields $m^{G^n} L^n_{\beta(n) t_0}(y^n)\ge \varepsilon \beta(n) t_0$, for every $y^n\in V(G^n)$ and large enough $n$, which in turn implies that $\beta(n) t_0\ge \tau^n_{\text{bl}}(\varepsilon)$, for $n$ large enough. As consequence we get that $\limsup_{n\rightarrow \infty} \beta(n)^{-1} \tau^n_{\text{bl}}(\varepsilon)\le \tau_{\text{bl}}(\varepsilon(1+\delta))$, from which \eqref{blacken3} follows by observing that $\tau_{\text{bl}}$ is right-continuous at $\varepsilon$.

\end{proof} 

\subsection{Local time convergence}

To check that Assumption \ref{Assum1} holds we need to verify that the convergence of the rescaled local times in \eqref{2.8}, as suggested by Lemma \ref{lem2.2}. Due to work done in a more general framework in \cite{croydon2016time}, we can weaken the local convergence statement of \eqref{2.8} and replace it by the equicontinuity condition of \eqref{2.11}. In \eqref{resist2} we defined a resistance metric on a graph viewed as an electrical network. Next, we give the definition of a regular resistance form and its associated resistance metric for arbitrary non-empty sets, which is a combination of \cite[Definition 2.1]{croydon2016time} and \cite[Definition 2.2]{croydon2016time}.

\begin{definition} \label{resistf}
Let $K$ be a non-empty set. A pair $(\mathcal{E},\mathcal{K})$ is called a regular resistance form on $K$ if the following six conditions are satisfied.
\begin{enumerate}[label=\roman{*})]

\item $\mathcal{K}$ is a linear subspace of the collection of functions $\{f:K\to \mathbb{R}\}$ containing constants, and $\mathcal{E}$ is a non-negative symmetric quadratic form on $\mathcal{K}$ such that $\mathcal{E}(f,f)=0$ if and only if $f$ is constant on $K$.

\item Let $\sim$ be an equivalence relation on $\mathcal{K}$ defined by saying $f\sim g$ if and only if the difference $f-g$ is constant on $K$. Then, $(\mathcal{K}/\sim,\mathcal{E})$ is a Hilbert space.

\item If $x\neq y$, there exists $f\in \mathcal{K}$ such that $f(x)\neq f(y)$.

\item For any $x,y\in K$,
\[
R(x,y):=\sup \left\{\frac{|f(x)-f(y)|^2}{\mathcal{E}(f,f)}: f\in \mathcal{K}, \ \mathcal{E}(f,f)>0\right\}<\infty.
\]

\item If $\bar{f}:=(f\wedge 1)\vee 0$, then $f\in \mathcal{K}$ and $\mathcal{E}(\bar{f},\bar{f})\le \mathcal{E}(f,f)$ for any $f\in \mathcal{K}$.

\item If $\mathcal{K}\cap C_0(K)$ is dense in $C_0(K)$ with respect to the supremum norm on $K$, where $C_0(K)$ denotes the space of compactly supported, continuous (with respect to $R$) functions on $K$.

\end{enumerate}
\end{definition}

It is the first five conditions that have to be satisfied in order for the pair $(\mathcal{E},\mathcal{K})$ to define a resistance form. If in addition the sixth condition is satisfied then $(\mathcal{E},\mathcal{K})$ defines a regular resistance form. Note that the fourth condition can be rewritten as $R(x,y)^{-1}=\inf\{\mathcal{E}(f,f): f:K\to \mathbb{R},f(x)=0,f(y)=1\}$, and it can be proven that it is actually a metric on $K$ (see \cite[Proposition 3.3]{kigami2012resistance}).  It also clearly resembles the effective resistance on $V(G)$ as defined in \eqref{resist2}. More specifically, taking $\mathcal{K}:=\{f:V(G)\to \mathbb{R}\}$ and $\mathcal{E}_G$ as defined in \eqref{resist1} one can prove that the pair $(\mathcal{E}_G,\mathcal{K})$ satisfies the six conditions of Definition \ref{resistf}, and therefore is a regular resistance form on $V(G)$ with associated resistance metric given by \eqref{resist2}. For a detailed proof of this fact see \cite[Example 1.2.5]{fukushima2010dirichlet}. Finally, in this setting given a regular Dirichlet form, standard theory gives us the existence of an associated Hunt process $X=((X_t)_{t\ge 0},\mathbf{P}_x,x\in K)$ that is defined uniquely everywhere (see \cite[Theorem 7.2.1]{fukushima2010dirichlet} and \cite[Theorem 9.9]{kigami2012resistance}).

Suppose that the discrete state spaces $(V(G^n))_{n\ge 1}$ are equipped with resistances $(R_{G^n})_{n\ge 1}$ as defined in \eqref{resist2} and that the limiting non-empty metric space $K$, that appears in Assumption \ref{Assum1}, is equipped with a resistance metric $R$ as in Definition \ref{resistf}, such that 
\begin{itemize}

\item $(K,R)$ is compact,

\item $\pi$ is a Borel measure of full support on $(K,R)$,

\item $X=((X_t)_{t\ge 0},\mathbf{P}_x,x\in K)$ admits local times $L=(L_t(x))_{x\in K,t\ge 0}$ continuous at $x$, uniformly over compact intervals in $t$, $\mathbf{P}_x$ -a.s. for every $x\in K$. 
\end{itemize}

In the following extra assumption we input the information encoded in the first three conclusions of Lemma \ref{lem2.2}, given that we work in a probabilistic setting instead. For simplicity as before we identify the various objects with their embeddings.

\begin{assum} \label{Assum3}
Fix $T>0$. Let $(G^n)_{n\ge 1}$ be a sequence of finite connected graphs that have at least two vertices, for which there exist sequences of real numbers $(\alpha(n))_{n\ge 1}$ and $(\beta(n))_{n\ge 1}$, such that
\[
\left(\left(V(G^n),\alpha(n) R_{G^n},\rho^n\right),\pi^n,\left(X^n_{\beta(n) t}\right)_{t\in [0,T]}\right)\longrightarrow \left(\left(K,R,\rho\right),\pi,X\right)
\]
in the sense of the pointed extended pointed Gromov-Hausdorff topology, where $\rho^n\in V(G^n)$ and $\rho\in K$ are distinguished points. Furthermore, suppose that for every $\varepsilon>0$ and $T>0$,
\begin{equation} \label{emb3}
\lim_{\delta\to 0} \limsup_{n\to \infty} \sup_{x\in V(G^n)} \mathbf{P}_x^n\left(\sup_{\substack{y,z\in V(G^n): \\ R_{G^n}(y,z)<\delta}} \sup_{t\in [0,T]} \alpha(n) |L_{\beta(n) t}^n(y)-L_{\beta(n) t}^n(z)|\ge \varepsilon\right)=0.
\end{equation}
\end{assum}
It is Assumption \ref{Assum3} we have to verify in the examples of random graphs we will consider later. As we prove below in the last lemma of this subsection, if Assumption \ref{Assum3} holds then the finite dimensional local times converge in distribution (see \eqref{2.8}). Given that $(V(G^n),R_{G^n})_{n\ge 1}$ and $(K,R)$ can be isometrically embedded into a common metric space $(F,d_F)$ such that $X^n$ under $\mathbf{P}_{\rho^n}^n$ converges weakly to the law of $X$ under $\mathbf{P}_{\rho}$ on $D([0,T],F)$ (see Lemma \ref{lem2.2}), 
we can couple $X^n$ started from $\rho^n$ and $X$ started from $\rho$ into a common probability space such that $(X^n_{\beta(n) t})_{t\in [0,T]}\to (X_t)_{t\in [0,T]}$ in $D([0,T],F)$, almost-surely.
Denote by $\mathbf{P}$ the joint probability measure under which the convergence above holds. Proving the convegence of finite dimensional distributions of local times is then an application of three lemmas that appear in \cite{croydon2016time}, which we summarize below.

\begin{lemma} [\textbf{Croydon, Hambly, Kumagai \cite{croydon2016time}}] \label{three}
For every $x\in F$, $\delta>0$, introduce the function $f_{\delta,x}(y):=\max\{0,\delta-d_F(x,y)\}$. Then, under Assumption \ref{Assum3},
\begin{enumerate}[label=\roman{*})]

\item $\mathbf{P}$-a.s., for each $x\in K$ and $T>0$, as $\delta\to 0$,

\[
\sup_{t\in [0,T]} \left|\frac{\int_{0}^{t} f_{\delta,x}(X_s) ds}{\int_K f_{\delta,x}(y) \pi(dy)}
-L_t(x)
\right|\to 0.
\]
  
\item $\mathbf{P}$-a.s., for each $x\in K$, $T>0$ and $\delta>0$, as $n\to \infty$,

\[
\sup_{t\in [0,T]} \left|\frac{\int_{0}^{t} f_{\delta,x}(X_s) ds}{\int_K f_{\delta,x}(y) \pi(dy)}
-\frac{\int_{0}^{t} f_{\delta,x}(X_{\beta(n) s}^n) ds}{\int_{V(G^n)} f_{\delta,x}(y) \pi^n(dy)}
\right|\to 0.
\]

\item For each $x\in K$ and $T>0$, if $x^n\in V(G^n)$ is such that $d_F(x^n,x)\to 0$, as $n\to \infty$, then 

\[
\lim_{\delta\to 0} \limsup_{n\to \infty} \mathbf{P}\left(\sup_{t\in [0,T]} \left|\frac{\int_{0}^{t} f_{\delta,x}(X_{\beta(n) s}^n) ds}{\int_{V(G^n)} f_{\delta,x}(y) \pi^n(dy)}-\alpha(n) L_{\beta(n) t}^n(x^n)\right|>\varepsilon \right)=0.
\]
\end{enumerate}
\end{lemma}

By applying the conclusions of Lemma \ref{three}, one deduces that for any $x\in K$ and $T>0$, if $x^n\in V(G^n)$ such that $d_F(x^n,x)\to 0$, as $n\to \infty$, then $(\alpha(n) L^n_{\beta(n) t}(x^n))_{t\in [0,T]}\to (L_t(x))_{t\in [0,T]}$ in $\mathbf{P}$-probability in $C([0,T],\mathbb{R})$. This result extends to finite collections of points, and this is enough to establish the convergence of finite dimensional distributions of local times.

\begin{lemma} \label{emblem}
Suppose that Assumption \ref{Assum3} holds. Then, if the finite collections $(x_i^n)_{i=1}^{k}$ in $V(G^n)$, for $n\ge 1$, are such that $d_F(x_i^n,x_i)\to 0$, as $n\to \infty$, for some $(x_i)_{i=1}^{k}$ in $K$, then it holds that 
\begin{equation} \label{emb4}
(\alpha(n) L_{\beta(n) t}^n(x_i^n))_{i=1,...,k,t\in [0,T]}\to (L_t(x_i))_{i=1,...,k,t\in [0,T]},
\end{equation}
in distribution in $C([0,T],\mathbb{R}^k)$.
\end{lemma}
\section{Examples} \label{maintheorems}

In this section we demonstrate that it is possible to apply our main results in a number of examples where the graphs, and the limiting spaces are random. These examples include critical Galton-Watson trees, the critical Erd\H{o}s-R\'enyi random graph and the critical regime of the configuration model. The aforementioned models of sequences of random graphs exhibit a mean-field behavior at criticality in the sense that the scaling exponents for the walks, and consequently for the local times, are a multiple of the volume and the diameter of the graphs. In the first few pages of each subsection we quickly survey some of the key features of each example that will be helpful when verifying Assumption \ref{Assum3}. Our method used in proving continuity of the blanket time of the limiting diffusion is generic in the sense that it applies on each random metric measure space and a corresponding $\sigma$-finite measure that generates realizations of the random metric measure space in such a way that rescaling the $\sigma$-finite measure by a constant factor results in generating the same space with its metric and measure perturbed by a multiple of this constant factor. 

\subsection{Critical Galton-Watson trees} \label{CGWT}

We start by briefly describing the connection between critical Galton-Watson trees and the Brownian continuum random tree (CRT). Let $\xi$ be a mean 1 random variable with variance $0<\sigma_{\xi}^2<+\infty$, whose distribution is aperiodic (its support generates the lattice $\mathbb{Z}$, not just a strict subgroup of $\mathbb{Z}$). Let $\mathcal{T}_n$ be a Galton-Watson tree with offspring distribution $\xi$ conditioned to have total number of vertices $n+1$, which is well-defined from the aperiodicity of the distribution of $\xi$. Then, it is the case that 
\begin{equation} \label{con1}
\left(V(\mathcal{T}_n),n^{-1/2} d_{\mathcal{T}_n}\right)\rightarrow \left(\mathcal{T}_e,d_{\mathcal{T}_e}\right),
\end{equation}
in distribution with respect to the Gromov-Hausdorff distance between compact metric spaces, where $d_{\mathcal{T}_n}$ is the shortest path distance on the vertex set $V(\mathcal{T}_n)$ (see \cite{aldous1993continuum} and \cite{le1993uniform}). To describe the limiting object in \eqref{con1}, let $(e_t)_{0\le t\le 1}$ denote the normalised Brownian excursion, which is informally a linear Brownian motion, started from zero, conditioned to remain positive in $(0,1)$ and come back to zero at time 1. We extend the definition of $(e_t)_{0\le t\le 1}$ by setting $e_t=0$, if $t>1$. We define a distance $d_{\mathcal{T}_e}$ in $[0,1]$ by setting 
\begin{equation} \label{dist}
d_{\mathcal{T}_e}(s,t)=e(s)+e(t)-2 \min_{r\in [s\wedge t,s\vee t]} e(r).
\end{equation}
Introducing the equivalence relation $s\sim t$ if and only if $e(s)=e(t)=\min_{r\in [s\wedge t,s\vee t]} e(r)$ and defining $\mathcal{T}_e:=[0,1]/\sim $ it is possible to check that $(\mathcal{T}_e,d_{\mathcal{T}_e})$ is almost-surely a compact metric space. Moreover, $(\mathcal{T}_e,d_{\mathcal{T}_e})$ is a random real tree , called the CRT. For the notion of compact real trees coded by functions and a proof of the previous result see \cite[Section 2]{le2006random}. There is a natural Borel measure upon $\mathcal{T}_e$, $\pi^e$ say, which is the image measure on $\mathcal{T}_e$ of the Lebesgue measure on $[0,1]$ by the canonical projection of $[0,1]$ onto $\mathcal{T}_e$. 

Upon almost-every realization of the metric measure space $(\mathcal{T}_e,d_{\mathcal{T}_e},\pi^e)$, it is possible to define a corresponding Brownian motion $X^e$. The way this can be done is described in \cite[Section 2.2]{croydon2008convergence}. Now if we denote by $\mathbf{P}_{\rho^n}^{\mathcal{T}_n}$ the law of the simple random walk in $\mathcal{T}_n$, started from a distinguished point $\rho^n$, and by $\pi^n$ the stationary probability measure, then as it was shown in \cite{croydon2010scaling} the scaling limit in \eqref{con1} can be extended to the distributional convergence of 
\[
\bigg(\big(V(\mathcal{T}_n),n^{-1/2} d_{\mathcal{T}_n},\rho^n\big),\pi^n(n^{1/2} \cdot ),\mathbf{P}_{\rho^n}^{\mathcal{T}_n}\big((n^{-1/2} X^n_{\lfloor n^{3/2} t \rfloor})_{t\in [0,1]}\in \cdot \big)\bigg)
\]
to $\bigg(\big(\mathcal{T}_e,d_{\mathcal{T}_e},\rho\big),\pi^e,\mathbf{P}_{\rho}^e\bigg)$, where $\mathbf{P}_{\rho}^e$ is the law of $X^e$, started from a distinguished point $\rho$. This convergence described in \cite{croydon2010scaling} holds after embedding all the relevant objects nicely into a Banach space. We can reformulate this result in terms of the pointed extended Gromov-Hausdorff topology that incorporates distinguished points. Namely
\begin{equation}\label{con2}
\left(\left(V(\mathcal{T}_n),n^{-1/2} d_{\mathcal{T}_n},\rho^n\right),\pi^n,\left(n^{-1/2} X^n_{\lfloor n^{3/2} t \rfloor}\right)_{t\in [0,1]}\right)\longrightarrow \left(\left(\mathcal{T}_e,d_{\mathcal{T}_e},\rho\right),\pi^e,\left(X^e_{(\sigma_{\xi}/2) t}\right)_{t\in [0,1]}\right),
\end{equation}
in distribution in an extended pointed Gromov-Hausdorff sense. 

Next, we introduce the contour function of $\mathcal{T}_n$. Informally, it encodes the trace of the motion of a particle that starts from the root at time $t=0$ and then explores the tree from left to right, moving continuously at unit speed along its edges. Formally, we define a function first for integer arguments as follows:
\[
f(0)=\rho^n.
\]
Given $f(i)=v$, we define $f(i+1)$ to be, if possible, the leftmost child that has not been visited yet, let's say $w$. If no child is left unvisited, we let $f(i+1)$ be the parent of $v$. Then, the contour function of $\mathcal{T}_n$, is defined as the distance of $f(i)$ from the root $\rho_n$, i.e.
\[
V_n(i):=d_{\mathcal{T}_n}(\rho^n,f(i)), \qquad 0\le i\le 2 n.
\]
We extend $V_n$ continuously to $[0,2 n]$ by interpolating linearly between integral points. The following theorem is due to Aldous.

\begin{theorem} [\textbf{Aldous \cite{aldous1993continuum}}]
Let $v_n$ denote the normalized contour function of $\mathcal{T}_n$, defined by 
\[
v_n(s):=\frac{V_n(2 n s)}{\sqrt{n}}, \qquad 0\le s\le 1.
\]
Then, the following convergence holds in distribution in $C([0,1])$:
\[
v_n\xrightarrow{(d)} v:=\frac{2}{\sigma_{\xi}} e,
\] 
where $e$ is a normalized Brownian excursion.
\end{theorem}

An essential tool in what follows will be a universal concentration estimate of the fluctuations of local times that holds uniformly over compact time intervals. For the statement of this result let 
\[
r(\mathcal{T}_n):=\sup_{x,y\in V(\mathcal{T}_n)} d_{\mathcal{T}_n}(x,y)
\]
denote the diameter of $\mathcal{T}_n$ in the resistance metric and $m(\mathcal{T}_n)$ denote the total mass of $\mathcal{T}_n$. Also, we introduce the rescaled resistance metric $\tilde{d}_{\mathcal{T}_n}(x,y):=r(\mathcal{T}_n)^{-1} d_{\mathcal{T}_n}(x,y)$.

\begin{theorem} [\textbf{Croydon \cite{croydon2015moduli}}] \label{prepar1}
For every $T>0$, there exist constants $c_1$ and $c_2$ not depending on $\mathcal{T}_n$ such that 
\begin{equation} \label{proof1}
\sup_{y,z\in V(\mathcal{T}_n)} \mathbf{P}_{\rho^n}^{\mathcal{T}_n} \left(r(\mathcal{T}_n)^{-1} \sup_{t\in [0,T]} \left|L_{r(\mathcal{T}_n) m(\mathcal{T}_n) t}^n(y)-L_{r(\mathcal{T}_n) m(\mathcal{T}_n) t}^n(z) \right|\ge \lambda \sqrt{\tilde{d}_{\mathcal{T}_n}(y,z)}\right)\le c_1 e^{-c_2 \lambda}
\end{equation}
for every $\lambda\ge 0$. Moreover, the constants can be chosen in such a way that only $c_1$ depends on $T$.
\end{theorem}

We remark here that the product $m(\mathcal{T}_n) r(\mathcal{T}_n)$, that is the product of the volume and the diameter of the graph, which is also the maximal commute time of the random walk, gives the natural time-scaling for the various models of sequences of critical random graphs we are going to consider. The last ingredient we are going to make considerable use of is the tightness of the sequence $||v_n||_{H_{\alpha}}$ of H\"older norms, for some $\alpha>0$. The proof of \eqref{proof2} is based on Kolmogorov's continuity criterion (and its proof to get uniformity in $n$). 

\begin{theorem} [\textbf{Janson and Marckert \cite{janson2005convergence}}] \label{prepar2}
There exists $\alpha\in (0,1/2)$ such that for every $\varepsilon>0$ there exists a finite real number $K_{\varepsilon}$ such that
\begin{equation} \label{proof2}
P\left(\sup_{s,t\in [0,1]} \frac{|v_n(s)-v_n(t)|}{|t-s|^{\alpha}}\le K_{\varepsilon}\right)\ge 1-\varepsilon,
\end{equation}
uniformly on $n$.
\end{theorem}

\begin{remark}
Building upon \cite{gitten2003state}, Janson and Marckert proved this precise estimate on the geometry of the trees when the offspring distribution has finite exponential moments. Relaxing this strong condition to only a finite variance assumption, the recent work of Marzouk and more specifically \cite[Lemma 1]{marzouk2018snake} implies that Theorem \ref{prepar2} holds for the normalized height function of $\mathcal{T}_n$, which constitutes an alternative encoding of the trees. That Theorem \ref{prepar2} can be stated as well in the terms of the normalized contour function of $\mathcal{T}_n$, with only a finite variance assumption, is briefly achieved using that the normalized contour process is arbitrarily close to a time-changed normalized height process. See the equation that appears as (15) in \cite[Theorem 1.7]{legall2005applications} and refer to \cite[Section 1.6]{legall2005applications} for a detailed discussion.
\end{remark}

Since $(\mathcal{T}_n)_{n\ge 1}$ is a collection of graph trees it follows that the shortest path distance $d_{\mathcal{T}_n}$, $n\ge 1$ is identical to the resistance metric on the vertex set $V(\mathcal{T}_n)$, $n\ge 1$. In this context, we make use of the full machinery provided by the theorems above in order to prove that the local times are equicontinuous with respect to the annealed law, which is formally defined for suitable events as 
\begin{equation} \label{exemplary}
\mathbb{P}_{\rho^n}(\cdot):=\int \mathbf{P}_{\rho^n}^{\mathcal{T}_n}(\cdot) P(d \mathcal{T}_n).
\end{equation}

\begin{proposition} \label{own}
For every $\varepsilon>0$ and $T>0$,
\[
\lim_{\delta\rightarrow 0} \limsup_{n\to \infty} \mathbb{P}_{\rho^n} \left(\sup_{\substack{y,z\in V(\mathcal{T}_n): \\ n^{-1/2} d_{\mathcal{T}_n}(y,z)<\delta}} \sup_{t\in [0,T]} n^{-1/2} |L_{n^{3/2} t}^n(y)-L_{n^{3/2} t}^n(z)|\ge \varepsilon\right)=0.
\]
\end{proposition}

\begin{proof}
Let us define, similarly to $d_{\mathcal{T}_e}$, the distance $d_{v_n}$ in $[0,1]$ by setting
\[
d_{v_n}(s,t):=v_n(s)+v_n(t)-2 \min_{r\in [s\wedge t,s\vee t]} v_n(r).
\]
Using the terminology introduced to describe the CRT, it is a fact that $\mathcal{T}_n$ coincides with the tree coded by the function $n^{1/2} v_n$, equipped with the metric $n^{1/2} d_{v_n}$ and the equivalence relation $s\sim t$ if and only if $v_n(s)=v_n(t)=\min_{r\in [s\wedge t,s\vee t]} v_n(r)$. We denote by $p_{v_n}: [0,1]\rightarrow \mathcal{T}_n$ the canonical projection that maps every time point in $[0,1]$ to its equivalence class on $\mathcal{T}_n$. From Theorem \ref{prepar2}, there exist $C>0$ and $\alpha\in (0,1/2)$ such that 
\begin{equation} \label{proof3}
d_{v_n}(s,t)\le C |t-s|^{\alpha}, \qquad \forall s,t\in [0,1]
\end{equation}
with probability arbitrarily close to 1. We condition on $v_n$, assuming that it satisfies \eqref{proof3}. Given $t_1,t_2\in [0,1]$, with $2 n t_1$ and $2 n t_2$ integers, such that $p_{v_n}(t_1)=y$ and $p_{v_n}(t_2)=z$, the total length of the path between $y$ and $z$ is, using \eqref{proof3}
\begin{align*}
V_n(2 n t_1)+V_n(2 n t_2)-2 \min_{r\in [t_1\wedge t_2,t_1\vee t_2]} V_n(2 n r)&=n^{1/2} (v_n(t_1)+v_n(t_2)-2 \min_{r\in [t_1\wedge t_2,t_1\vee t_2]} v_n(r))\\
&=n^{1/2} d_{v_n}(t_1,t_2)\le C n^{1/2} |t_1-t_2|^{\alpha}.
\end{align*}
By Theorem \ref{prepar1}, for any fixed $p>0$,
\begin{align*}
&\mathbf{E}_{\rho^n}^{\mathcal{T}_n} \left[\sup_{t\in [0,T]} r(\mathcal{T}_n)^{-1} \left|L_{r(\mathcal{T}_n) m(\mathcal{T}_n) t}^n(y)-L_{r(\mathcal{T}_n) m(\mathcal{T}_n) t}^n(z) \right|^p\right]\\
&=\int_{0}^{\infty} \mathbf{P}_{\rho^n}^{\mathcal{T}_n} \left(\sup_{t\in [0,T]}  r(\mathcal{T}_n)^{-1} \left|L_{r(\mathcal{T}_n) m(\mathcal{T}_n) t}^n(y)-L_{r(\mathcal{T}_n) m(\mathcal{T}_n) t}^n(z) \right|\ge \varepsilon^{1/p}\right) d \varepsilon\\
&\le c_1 \int_{0}^{\infty} e^{-c_2 \frac{\varepsilon^{1/p}}{\sqrt{r(\mathcal{T}_n)^{-1} d_{\mathcal{T}_n}(y,z)}}} d \varepsilon.
\end{align*}
Changing variables, $\lambda^{1/p}=\frac{\varepsilon^{1/p}}{\sqrt{r(\mathcal{T}_n)^{-1} d_{\mathcal{T}_n}(y,z)}}$ yields
\[
\int_{0}^{\infty} e^{-c_2 \frac{\varepsilon^{1/p}}{\sqrt{r(\mathcal{T}_n)^{-1} d_{\mathcal{T}_n}(y,z)}}} d \varepsilon=(r(\mathcal{T}_n)^{-1} d_{\mathcal{T}_n}(y,z))^{p/2} \cdot \int_{0}^{\infty}e^{-c_2 \lambda^{1/p}} d \lambda\le c_3 (r(\mathcal{T}_n)^{-1} d_{\mathcal{T}_n}(y,z))^{p/2},
\]
where $c_3$ is a constant depending only on $p$. Therefore, 
\begin{equation} \label{proof4}
\mathbf{E}_{\rho^n}^{\mathcal{T}_n} \left[\sup_{t\in [0,T]} r(\mathcal{T}_n)^{-1} \left|L_{r(\mathcal{T}_n) m(\mathcal{T}_n) t}^n(y)-L_{r(\mathcal{T}_n) m(\mathcal{T}_n) t}^n(z) \right|^p\right]\le c_3 (r(\mathcal{T}_n)^{-1} d_{\mathcal{T}_n}(y,z))^{p/2}.
\end{equation}
Conditioned on the event that $v_n$ satisfies \eqref{proof3}, observe that the total length of the path between $y$ and $z$ is bounded above by $C n^{1/2} |t_1-t_2|^{\alpha}$, and consequently the diameter of $\mathcal{T}_n$ in the resistance metric is bounded above by a multiple of $n^{1/2}$. More specifically, $r(\mathcal{T}_n)\le C n^{1/2} 2^{\alpha}$. Moreover, $m(\mathcal{T}_n)=2n$. Hence, by \eqref{proof4}, we have shown that, conditioned on $v_n$ satisfying \eqref{proof3}, for any fixed $p>0$,

\begin{align*}
\mathbf{E}_{\rho^n}^{\mathcal{T}_n} \left[\sup_{t\in [0,T]} n^{-1/2} \left|L_{n^{3/2} t}^n(y)-L_{n^{3/2} t}^n(z) \right|^p\right]&\le c_4  (n^{-1/2} d_{\mathcal{T}_n}(y,z))^{p/2} 2^{-\alpha p}\\
&\le c_5 \left|(t_1/4)-(t_2/4)\right|^{\alpha p/2}.
\end{align*}
This holds for all $t_1$ and $t_2$, with $2 n t_1$ and $2 n t_2$ integers, such that $p_{v_n}(t_1)=y$ and $p_{v_n}(t_2)=z$. Choosing $p$ such that $\alpha p>2$, and using the moment condition (13.14) of \cite[Theorem 13.5]{billingsley2013convergence} yields that, under the event that $v_n$ satisfies \eqref{proof3}, the sequence of the rescaled local times $\{(n^{-1/2} L^n_{n^{3/2} t}(2 n \cdot))_{t\in [0,T]}\}_{n\ge 1}$ is tight in $C[0,1]$. Therefore, for every $\varepsilon>0$ and $T>0$,
\[
\lim_{\delta\rightarrow 0} \limsup_{n\to \infty} \mathbb{P}_{\rho^n} \left(\sup_{\substack{y,z\in V(\mathcal{T}_n): \\ n^{-1/2} d_{\mathcal{T}_n}(y,z)<\delta}} \sup_{t\in [0,T]} n^{-1/2} |L_{n^{3/2} t}^n(y)-L_{n^{3/2} t}^n(z)|\ge \varepsilon\Bigg |\sup_{s,t\in [0,1]} \frac{d_{v_n}(s,t)}{|t-s|^{\alpha}}\le C \right)=0.
\]
To complete the proof, note that 
\[
\mathbb{P}_{\rho^n}(A)\le \mathbb{P}_{\rho^n}\left(A\Bigg|\sup_{s,t\in [0,1]} \frac{d_{v_n}(s,t)}{|t-s|^{\alpha}}\le C\right)+P(d_{v_n}(s,t)>C |t-s|^{\alpha}, \ \forall s,t\in [0,1]),
\]
measurable $A\subseteq \mathbb{K}$. The desired result now follows using the tightness of the rescaled local times, conditioned on $v_n$ satisfying \eqref{proof3}, which was shown before, and using the fact that the second probability on the right-hand side above, by \eqref{proof3}, is arbitrarily small.

\end{proof}

\subsubsection{It\^o's excursion theory of Brownian motion} \label{needed}

We recall some key facts of It\^{o}'s excursion theory of reflected Brownian motion. Our main interest here lies on the scaling property of the It\^{o} excursion measure. Let $(L_t^0)_{t\ge 0}$ denote the local time process at level 0 of the reflected Brownian motion $(|B_t|)_{t\ge 0}$. It can be shown that 
\[
L_t^0=\lim_{\varepsilon\rightarrow 0}\frac{1}{2 \varepsilon} \int_{0}^{t} \one_{[0,\varepsilon]}(|B_s|) ds,
\]
for every $t\ge 0$, a.s. The local time process at level 0 is increasing, and its set of points of increase coincides with the set of time points for which the reflected Brownian is identical to zero. Now, introducing the right-continuous inverse of the local time process at level 0, i.e.
\[
\tau_k:=\inf\{t\ge 0: L_t^0>k\},
\]
for every $k\ge 0$, we have that the set of points of increase of $(L_t^0)_{t\ge 0}$ alternatively belong to the set 
\[
\{\tau_k:k\ge 0\}\cup \{\tau_{k-}: k\in D\},
\]
where $D$ is the set of countable discontinuities of the mapping $k\mapsto \tau_k$. Then, for every $k\in D$ we define the excursion $(e_k(t))_{t\ge 0}$ with excursion interval $(\tau_{k-},\tau_k)$ away from 0 as 
\[
e_k(t):=
\begin{cases} 
   |B_{t+\tau_{k-}}|  & \text{if } 0\le t\le \tau_k-\tau_{k-}, \\
   0 & \text{if } t>\tau_k-\tau_{k-}.
\end{cases}
\]
Let $E$ denote the space of excursions, namely the space of functions $e\in C(\mathbb{R}_+,\mathbb{R}_+)$, satisfying $e(0)=0$ and $\zeta(e):=\sup\{s>0: e(s)>0\}\in (0,\infty)$. By convention $\sup \emptyset=0$. Observe here that for every $k\in D$, $e_k\in E$, and moreover $\zeta(e_k)=\tau_k-\tau_{k-}$. The main theorem of It\^o's excursion theory adapted in our setting is the existence of a $\sigma$-finite measure $\mathbb{N} (d e)$ on the space of positive excursions of linear Brownian motion, such that the point measure 
\[
\sum_{k\in D}\delta_{(k,e_k)}(ds \ de)
\]
is a Poisson measure on $\mathbb{R}_{+}\times E$, with intensity $ds \otimes \mathbb{N}(d e)$. The It\^o excursion measure has the following scaling property. For every $a>0$ consider the mapping $\Theta_a: E\rightarrow E$ defined by setting $\Theta_a(e)(t):=\sqrt{a}e(t/a)$, for every $e\in E$, and $t\ge 0$. Then, we have that 
\begin{equation} \label{sc1}
\mathbb{N}\circ \Theta^{-1}_a=\sqrt{a} \mathbb{N}
\end{equation}
Versions of the It\^o excursion measure $\mathbb{N}(d e)$ under different conditionings are possible. For example one can define conditionings with respect to the height or the length of the excursion. For our purposes we focus on the fact that there exists a unique collection of probability measures $(\mathbb{N}_s: s>0)$ on $E$, such that $\mathbb{N}_s(\zeta=s)=1$, for every $s>0$, and moreover for every measurable event $A\in E$,
\begin{equation} \label{sc2}
\mathbb{N}(A)=\int_{0}^{\infty} \mathbb{N}_s(A) \frac{ds}{2 \sqrt{2 \pi s^3}}.
\end{equation}
In other words $\mathbb{N}_s(d e)$ is the It\^o excursion measure $\mathbb{N}(d e)$, conditioned on the event $\{\zeta=s\}$. We might write $\mathbb{N}_1=\mathbb{N}(\cdot |\zeta=1)$ to denote law of the normalized Brownian excursion. It is straightforward from \eqref{sc1} and \eqref{sc2} to check that $\mathbb{N}_s$ satisfies the scaling property
\begin{equation} \label{sc3}
\mathbb{N}_s\circ \Theta_a^{-1}=\mathbb{N}_{as},
\end{equation}
for every $s>0$ and $a>0$. To conclude our recap on It\^o's excursion theory we highlight the fact that for every $t>0$ the process $(e(t+r))_{r\ge 0}$ is Markov under the conditional probability measure $\mathbb{N}(\cdot |\zeta>t)$. The transition kernel of the process is the same with the one of a Brownian motion killed upon the first time it hits zero. 

\subsubsection{Continuity of blanket times of Brownian motion on the CRT} \label{unes}

We are primarily interested in proving continuity of the $\varepsilon$-blanket time variable of the Brownian motion on the CRT. For every $\varepsilon \in (0,1)$ we let 

\[
\mathcal{A}_{\varepsilon}:=\left\{(\mathcal{T}_e)_{e\in E}: \mathbf{P}_{\rho}^{e}\left(\tau_{\text{bl}}^e(\varepsilon-)=\tau_{\text{bl}}^e(\varepsilon)\right)=1\right\}
\]
denote the the collection of random trees coded by positive excursions that have continuous $\varepsilon$-blanket times $\mathbf{P}_{\rho}^{e}$-a.s. Recall that the mapping $\varepsilon\mapsto \tau_{\text{bl}}^e(\varepsilon)$ is increasing in $(0,1)$, and therefore it posseses left and right limits, so $\mathcal{A}_{\varepsilon}$ is well-defined for every $\varepsilon\in (0,1)$. Moreover, $\varepsilon\mapsto \tau_{\text{bl}}^e(\varepsilon)$ has at most a countably infinite number of discontinuities $\mathbf{P}_{\rho}^e$-a.s. Using Fubini, we immediately get
\[
\int_{0}^{1} \int \mathbf{P}_{\rho}^{e}\left(\tau_{\text{bl}}^e(\varepsilon-)\neq \tau_{\text{bl}}^e(\varepsilon)\right) \mathbb{N}(d e) d \varepsilon=\mathbf{E}_{\mathbb{P}} \left[\int_{0}^{1} \one_{\left\{\tau_{\text{bl}}^e(\varepsilon-)\neq \tau_{\text{bl}}^e(\varepsilon)\right\}} d \varepsilon \right]=0,
\]
where by $\mathbf{E}_{\mathbb{P}_{\rho}}$, we denote the expectation with respect to the measure
\begin{equation} \label{sc5}
\mathbb{P}_{\rho}(\cdot):=\int \mathbf{P}_{\rho}^{e}(\cdot) \mathbb{N}(d e).
\end{equation}
Therefore, if we denote the Lebesgue measure by $\lambda$, we deduce that for $\lambda$-a.e. $\varepsilon$
\[
\int \left(1-\mathbf{P}_{\rho}^{e}\left(\tau_{\text{bl}}^e(\varepsilon-)=\tau_{\text{bl}}^e(\varepsilon)\right)\right) \mathbb{N}(d e)=0.
\]
The fact that $\mathbb{N}(d e)$ is a sigma-finite measure on $E$ yields that for $\lambda$-a.e. $\varepsilon$, $\mathbb{N}$-a.e. $e$
\begin{equation} \label{sc6}
\mathbf{P}_{\rho}^{e}\left(\tau_{\text{bl}}^e(\varepsilon-)=\tau_{\text{bl}}^e(\varepsilon)\right)=1.
\end{equation}
Thus, we inferred that for $\lambda$-a.e. $\varepsilon$, $\mathbb{N}$-a.e. $e$, $\mathcal{T}_e\in \mathcal{A}_{\varepsilon}$. 

To be satisfactory for our purposes, we need to improve this statement to hold for every $\varepsilon\in (0,1)$. Let $e\in E$ and consider the random real tree $(\mathcal{T}_e,d_{\mathcal{T}_e},\pi^e)$, where $d_{\mathcal{T}_e}$ is defined as in \eqref{dist} and $\pi^e$ is the image measure on $\mathcal{T}_e$ of the Lebesgue measure on $[0,\zeta]$ by the canonical projection $p_e$ of $[0,\zeta]$ onto $\mathcal{T}_e$. First, we observe that by applying the mapping $\Theta_a$ to $e$ for some $a>0$, results in perturbing $d_{\mathcal{T}_e}$ by a factor of $\sqrt{a}$ and $\pi^e$ by a factor of $a$. Recalling that $\Theta_a(e)(t)=\sqrt{a} e(t/a)$, for every $t\ge 0$, it is immediate that $d_{\mathcal{T}_{\Theta_a(e)}}(s,t)=\sqrt{a} d_{\mathcal{T}_e}(s,t)$, for every $s,t\in [0,\zeta]$. To see how $\pi^e$ rescales consider the set 
\[
I=\{t\in [0,\zeta]: p_e(t)\in A\},
\]
where $A$ is a Borel set of $\mathcal{T}_e$. Then, by definition $\pi^e(A)=\lambda(I)$. Moreover, $\pi^{\Theta_a(e)}(A)=\lambda(I')$, where 
\[
I=\{t\in [0,a \zeta]: p_e(t)\in A\}.
\]
By the scaling property of the Lebesgue measure, we have that $\pi^{\Theta_a(e)}(A)=a \lambda(I)$, and therefore $\pi^{\Theta_a(e)}(A)=a \pi^e(A)$. For simplicity, for the random real tree $\mathcal{T}= (\mathcal{T}_e,d_{\mathcal{T}_e},\pi^e)$, we write $\Theta_a \mathcal{T}$ to denote the resulting random real tree $(\mathcal{T}_e,\sqrt{a} d_{\mathcal{T}_e},a \pi^e)$ after rescaling. 

Next, if the Brownian motion on $\mathcal{T}$ admits local times $(L_t(x))_{x\in \mathcal{T}_e,t\ge 0}$ that are jointly continuous $\mathbf{P}_{\rho}^e$-a.s., then it is the case that the Brownian motion on $\Theta_a \mathcal{T}$ admits local times distributed as $(\sqrt{a} L_{a^{-3/2} t}(x))_{x\in \mathcal{T}_e,t\ge 0}$ that are jointly continuous $\mathbf{P}_{\rho}^e$-a.s. To justify this check that, for every $t>0$
\[
\int \sqrt{a} L_{a^{-3/2 } t}(x) \pi^{\Theta_a(e)}(d x)=a^{3/2} \int L_{a^{-3/2} t}(x) \pi^ e(d x)=t.
\]
Now, for every $\varepsilon\in (0,1)$ fraction of time and every scalar parameter $a>1$, for the $\varepsilon$-blanket time variable of the Brownian motion on $\Theta_a \mathcal{T}$ as defined in \eqref{blacken1}, we have that 
\begin{align*}
\tau^{\Theta_a(e)}_{\text{bl}}(a^{-1} \varepsilon) &\,{\buildrel (d) \over =}\, \inf \{t\ge 0: \sqrt{a} L_{a^{-3/2} t}(x)\ge \varepsilon a^{-1} t,\ \forall x\in \mathcal{T}_e\} \\
&\,{\buildrel (d) \over =}\, \inf \{t\ge 0: L_{a^{-3/2} t}(x)\ge \varepsilon a^{-3/2} t,\ \forall x\in \mathcal{T}_e\}\ \,{\buildrel (d) \over =}\, a^{-3/2} \tau_{\text{bl}}^e(\varepsilon).
\end{align*}
This fact asserts that $\mathcal{T}\in \mathcal{A}_{\varepsilon}$ if and only if $\Theta_a \mathcal{T}\in \mathcal{A}_{a^{-1} \varepsilon}$. In other words, $\tau_{\text{bl}}^e(\varepsilon)$ is continuous at $\varepsilon$,  $\mathbf{P}_{\rho}^{e}$-a.s. if and only if $\tau_{\text{bl}}^{\Theta_a(e)}(a^{-1} \varepsilon)$ is continuous at $\varepsilon$, $\mathbf{P}_{\rho}^e$-a.s. Using the precise way in which the blanket times above relate as well as the scaling properties of the usual and the normalized It\^o excursion we prove the following proposition.

\begin{proposition} \label{pr2}
For every $\varepsilon\in (0,1)$, $\mathbb{N}$-a.e. $e$, $\tau_{\textnormal{bl}}^e(\varepsilon)$ is continuous at $\varepsilon$, $\mathbf{P}_{\rho}^e$-a.s. Moreover, $\mathbb{N}_1$-a.e. $e$, $\tau_{\textnormal{bl}}^e(\varepsilon)$ is continuous at $\varepsilon$, $\mathbf{P}_{\rho}^e$-a.s.
\end{proposition}

\begin{proof}
Fix $\varepsilon\in (0,1)$. We choose $a>1$ in such a way that $a^{-1} \varepsilon\in \Omega_0$, where $\Omega_0$ is the set for which the assertion in \eqref{sc6} holds $\lambda$-a.e. $\varepsilon$. Namely, $\mathbb{N}$-a.e. $e$, $\mathcal{T}\in \mathcal{A}_{a^{-1} \varepsilon}$. Using the scaling property of It\^o's excursion measure as quoted in \eqref{sc1} yields $\sqrt{a} \mathbb{N}$-a.e. $e$, $\Theta_a \mathcal{T}\in \mathcal{A}_{a^{-1} \varepsilon}$, and consequently $\mathbb{N}$-a.e. $e$, $\mathcal{T}\in \mathcal{A}_{\varepsilon}$, where we exploited the fact that $\Theta_a \mathcal{T}\in \mathcal{A}_{a^{-1} \varepsilon}$ if and only if $\mathcal{T}\in \mathcal{A}_{\varepsilon}$. Since $\varepsilon\in (0,1)$ was arbitrary, this establishes the first conclusion that $\mathbb{N}$-a.e. $e$, $\tau_{\text{bl}}^e(\varepsilon)$ is continuous at $\varepsilon$, $\mathbf{P}_{\rho}^e$-a.s.

What remains now is to prove a similar result but with $\mathbb{N}(d e)$ replaced with its version conditioned on the length of the excursion. Following the same steps we used in order to prove \eqref{sc6}, we infer that for $\lambda$-a.e. $\varepsilon$, $\mathbb{N}(\cdot |\zeta\in [1,2])$-a.e. $e$, $\mathcal{T}\in \mathcal{A}_{\zeta^{-1} \varepsilon}$, and consequently $\lambda$-a.e. $\varepsilon$, $\mathbb{N}(\cdot |\zeta\in [1,2])$-a.e. $e$, $\Theta_{\zeta} \mathcal{T}\in \mathcal{A}_{\varepsilon}$. Using the scaling property of the normalised It\^o excursion measure (see \eqref{sc3}), we deduce that $\lambda$-a.e. $\varepsilon$, $\mathbb{N}_1$-a.e. $e$, $\mathcal{T}\in \mathcal{A}_{\varepsilon}$, where $\mathbb{N}_1$ is the law of the normalized Brownian excursion. To conclude, we proceed using the same argument as in the first part of the proof.

Fix an $\varepsilon\in (0,1)$ and choose $a>1$ such that $a^{-1} \varepsilon\in \Phi_0$, where $\Phi_0$ is the set for which the assertion $\mathbb{N}_1$-a.e. $e$, $\mathcal{T}\in \mathcal{A}_{\varepsilon}$ holds $\lambda$-a.e $\varepsilon$. Namely, $\mathbb{N}_1$-a.e. $e$, $\mathcal{T}\in \mathcal{A}_{a^{-1} \varepsilon}$, which from the scaling property of the normalised It\^o excursion measure yields $a \mathbb{N}_1$-a.e. $e$, $\Theta_a \mathcal{T}\in \mathcal{A}_{a^{-1} \varepsilon}$. As before this gives us that $\mathbb{N}_1$-a.e. $e$, $\mathcal{T}\in \mathcal{A}_{\varepsilon}$, or in other words that $\mathbb{N}_1$-a.e. $e$, $\tau_{\text{bl}}^e(\varepsilon)$ is continuous at $\varepsilon$, $\mathbf{P}_{\rho}^e$-a.s.

\end{proof}

Since the space in which the convergence in \eqref{con2} takes place is separable we can use Skorohod's coupling to deduce that there exists a common metric space $(F,d_F)$ and a joint probability measure $\tilde{\mathbf{P}}$ such that, as $n\to \infty$,
\[
d_H^F(V(\tilde{\mathcal{T}}_n),\tilde{\mathcal{T}}_e)\to 0, \qquad d_P^F(\tilde{\pi}^n,\tilde{\pi}^e)\to 0, \qquad d_F(\tilde{\rho}^n,\tilde{\rho})\to 0, \qquad \tilde{\mathbf{P}}\text{ -a.s.},
\]
where $(V(\mathcal{T}_n),\pi^n,\rho^n)\,{\buildrel (d) \over =}\, (V(\tilde{\mathcal{T}_n}),\tilde{\pi}^n,\tilde{\rho}^n)$ and $(\mathcal{T}_e,\pi^e,\rho)\,{\buildrel (d) \over =}\, (\tilde{\mathcal{T}_e},\tilde{\pi}^e,\tilde{\rho})$. Moreover, $X^n$ under $\mathbf{P}_{\tilde{\rho}^n}^{\tilde{\mathcal{T}_n}}$ converges weakly to the law of $X^e$ under $\mathbf{P}_{\tilde{\rho}}^{\tilde{e}}$ on $D([0,1],F)$. 

In Proposition \ref{own} we proved equicontinuity of the local times with respect to the annealed law. Then, reexamining the proof of Lemma \ref{emblem}, one can see that in this case $L^n$ under $\mathbb{P}_{\tilde{\rho}^n}(\cdot):= \int \mathbf{P}_{\tilde{\rho}^n}^{\tilde{\mathcal{T}_n}}(\cdot ) d \mathbf{\tilde{P}}$ will converge weakly to $L$ under $\mathbb{P}_{\tilde{\rho}}(\cdot):=\int \mathbf{P}_{\tilde{\rho}}^{\tilde{e}}(\cdot ) d \mathbf{\tilde{P}}$ in the sense of the local convergence as stated in \eqref{emb4}. It was this precise statement that was used extensively in the derivation of asymptotic distributional bounds for the rescaled blanket times in Section \ref{Sec2.2}. Then, the statement of Theorem \ref{Mth} translates as follows. For every $\varepsilon\in (0,1)$, $\delta\in (0,1)$ and $t\in [0,1]$,
\begin{equation*} 
\limsup_{n\rightarrow \infty} \int \mathbf{P}_{\tilde{\rho}^n}^{\tilde{\mathcal{T}_n}} \left(n^{-3/2} \tau_{\textnormal{bl}}^n(\varepsilon)\le t\right) d \mathbf{\tilde{P}}\le \int \mathbf{P}_{\tilde{\rho}}^{\tilde{e}} \left(\tau_{\textnormal{bl}}^e(\varepsilon(1-\delta))\le t\right) d \mathbf{\tilde{P}},
\end{equation*} 
\begin{equation*} 
\liminf_{n\rightarrow \infty} \int \mathbf{P}_{\tilde{\rho}^n}^{\tilde{\mathcal{T}_n}} \left(n^{-3/2} \tau_{\textnormal{bl}}^n(\varepsilon)\le t\right) d \mathbf{\tilde{P}}\ge \int \mathbf{P}_{\tilde{\rho}}^{\tilde{e}} \left({\tau}_{\textnormal{bl}}^e(\varepsilon)<t\right) d \mathbf{\tilde{P}}.
\end{equation*}
From Proposition \ref{pr2} and the dominated convergence theorem we have that for every $\varepsilon\in (0,1)$ and $t\in [0,1]$,
\begin{align*}
\lim_{\delta\to 0} \int \mathbf{P}_{\tilde{\rho}}^{\tilde{e}} \left(\tau_{\textnormal{bl}}^e(\varepsilon(1-\delta))\le t\right) d \mathbf{\tilde{P}}&=\lim_{\delta\to 0} \int \mathbf{P}_{\tilde{\rho}}^{\tilde{e}} \left({\tau}_{\textnormal{bl}}^e(\varepsilon)<t\right) d \mathbf{\tilde{P}}=\int \mathbf{P}_{\rho}^e\left(\tau_{\text{bl}}^e(\varepsilon)\le t\right) \mathbb{N}(d e)
\\
&=\mathbb{P}_{\rho}\left(\tau_{\text{bl}}^e(\varepsilon)\le t\right).
\end{align*}
Therefore, we deduce that for every $\varepsilon\in (0,1)$ and $t\in [0,1]$,
\begin{align*}
\lim_{n\to \infty} \mathbb{P}_{\rho^n} \left(n^{-3/2} \tau_{\textnormal{bl}}^n(\varepsilon)\le t\right)&=\lim_{n\to \infty} \int \mathbf{P}_{\rho^n}^{\mathcal{T}_n} \left(n^{-3/2} \tau_{\textnormal{bl}}^n(\varepsilon)\le t\right) P(d \mathcal{T}_n)=\mathbb{P}_{\rho}\left(\tau_{\text{bl}}^e(\varepsilon)\le t\right).
\end{align*}
In the theorem below we state the $\varepsilon$-blanket time variable convergence result we have just proved.

\begin{theorem}
Fix $\varepsilon\in (0,1)$. If $\tau_{\textnormal{bl}}^n(\varepsilon)$ is the $\varepsilon$-blanket time variable of the random walk on $\mathcal{T}_n$, started from its root $\rho^n$, then 
\[
\mathbb{P}_{\rho^n}\left(n^{-3/2} \tau_{\textnormal{bl}}^n(\varepsilon)\le t\right)\to \mathbb{P}_{\rho}\left(\tau_{\textnormal{bl}}^e(\varepsilon)\le t\right),
\]
for every $t\ge 0$, where $\tau_{\textnormal{bl}}^e(\varepsilon)\in (0,\infty)$ is the $\varepsilon$-blanket time variable of the Brownian motion on $\mathcal{T}_e$, started from $\rho$. Equivalently, for every $\varepsilon\in (0,1)$, $n^{-3/2} \tau_{\textnormal{bl}}^n(\varepsilon)$ under $\mathbb{P}_{\rho^n}$ converges weakly to $\tau_{\textnormal{bl}}^e(\varepsilon)$ under $\mathbb{P}_{\rho}$.
\end{theorem}

\subsection{The critical Erd\texorpdfstring{\H{o}}{H}s-R\'enyi random graph} \label{lalala}

Our interest in this section shifts to the Erd\H{o}s-R\'enyi random graph at criticality. Take $n$ vertices labelled by $[n]:=\{1,...,n\}$ and put edges between any pair independently with fixed probability $p\in [0,1]$. Denote the resulting random graph by $G(n,p)$. Let $p=c/n$ for some $c>0$. This model exhibits a phase transition in its structure for large $n$. With probability tending to $1$, when $c<1$, the largest connected component has size $O(\log n)$. On the other hand, when $c>1$, we see the emergence of a giant component that contains a positive proportion of the vertices. We will focus here on the critical case $c=1$, and more specifically, on the critical window $p=n^{-1}+\lambda n^{-4/3}$, $\lambda\in \mathbb{R}$. The most significant result in this regime was proven by Aldous \cite{aldous1997brownian}. Fix $\lambda\in \mathbb{R}$ and let $(C_i^n)_{i\ge 1}$ denote the sequence of the component sizes of $G(n,n^{-1}+\lambda n^{-4/3})$. For reasons that are inherent in understanding the structure of the components, we track the surplus of each one, that is the number of vertices that have to be removed in order to obtain a tree. Let $(S_i^n)_{n\ge 1}$ be the sequence of the corresponding surpluses.

\begin{theorem} [\textbf{Aldous \cite{aldous1997brownian}}] \label{PhT}
As $n\to \infty$,
\begin{equation} \label{PhT1}
\left(n^{-2/3} (C_i^n)_{i\ge 1},(S_i^n)_{i\ge 1}\right)\longrightarrow \left((C_i)_{i\ge 1},(S_i)_{i\ge 1}\right)
\end{equation}
in distribution, where the convergence of the first sequence takes place in $\ell^2_{\downarrow}$, the set of positive, decreasing sequences $(x_i)_{i\ge 1}$ with $\sum_{i=1}^{\infty} x_i^2<\infty$. For the second sequence it takes place in the product topology.
\end{theorem} 

The limit is described by stochastic processes that encode various aspects of the structure of the random graph. Consider a Brownian motion with parabolic drift, $(B^{\lambda}_t)_{t\ge 0}$, where
\begin{equation} \label{parabola}
B^{\lambda}_t:=B_t+\lambda t-\frac{t^2}{2}
\end{equation}
and $(B_t)_{t\ge 0}$ is a standard Brownian motion. Then, the limiting sequence $(C_i)_{i\ge 1}$ has the distribution of the ordered sequence of lengths of excursions of the process $B^{\lambda}_t-\inf_{0\le s\le t} B^{\lambda}_s$, that is the parabolic Brownian motion reflected upon its minimum. Finally, $(S_i)_{i\ge 1}$ is recoved as follows. Draw the graph of the reflected process and scatter points on the place according to a rate 1 Poisson process and keep those that fall between the $x$-axis and the function. Then, $S_i$, $i\ge 1$ are the Poisson number of points that fell in the corresponding excursion with length $C_i$. Observe that the distribution of the limit $(C_i)_{i\ge 1}$ depends on the particular value of $\lambda$ chosen. 

The scaling limit of the largest connected component of the Erd\H{o}s-R\'enyi random graph on the critical window arises as a tilted version of the CRT. Given a pointset $\mathcal{P}$, that is a subset of the upper half plane that contains only a finite number of points in any compact subset, and a positive Brownian excursion $e$, we define $\mathcal{P}\cap e$ as the number of points from $\mathcal{P}$ that fall under the area of $e$. We construct a glued metric space $\mathcal{M}_{e,\mathcal{P}}$ as follows. For each point $(t,x)\in \mathcal{P}\cap e$, let $u_{(t,x)}$ be the unique vertex $p_e(t)\in \mathcal{T}_e$ and $v_{(t,x)}$ be the unique vertex on the path from the root to $u_{(t,x)}$ at a distance $x$ from the root. Let $E_{\mathcal{P}}=\{(u_{(t,x)},v_{(t,x)}):(t,x)\in \mathcal{P}\cap e\}$ be the finite set that consists of the pairs of vertex identifications. Let $\{v_i,u_i\}_{i=1,...,k}$ be $k$ pairs of points that belong to $E_{\mathcal{P}}$. We define a quasi-metric on $\mathcal{T}_e$ by setting 
\[
d_{\mathcal{M}_{e,\mathcal{P}}}(x,y):=\min\left\{d_{\mathcal{T}_e}(x,y),\inf_{i_1,...,i_r} \left\{d_{\mathcal{T}_e}(x,u_{i_1})+\sum_{j=1}^{r-1} d_{\mathcal{T}_e}(v_{i_j},u_{i_{j+1}})+d_{\mathcal{T}_e}(v_r,y)\right\}\right\},
\]
where the infimum is taken over $r$ positive integers, and all subsets $\{i_1,...,i_r\}\subseteq \{1,...,k\}$. Moreover, note that the vertices $i_1,...,i_k$ can be chosen to be distinct. The metric defined above gives the shortest distance between $x,y\in \mathcal{T}_e$ when we glue the vertices $v_i$ and $u_i$ for $i=1,...,k$. It is clear that $d_{\mathcal{M}_{e,\mathcal{P}}}$ defines only a quasi-metric since $d_{\mathcal{M}_{e,\mathcal{P}}}(u_i,v_i)=0$, for every $i=1,...,k$, but $u_i\neq v_i$, for every $i=1,...,k$. We define an equivalence relation on $\mathcal{T}_e$ by setting $x\sim_{E_{\mathcal{P}}} y$ if and only if $d_{\mathcal{M}_{e,\mathcal{P}}}(x,y)=0$. This makes the vertex identification explicit and $\mathcal{M}_{e,\mathcal{P}}$ is defined as 
\[
\mathcal{M}_{e,\mathcal{P}}:=(\mathcal{T}_e/\sim_{E_{\mathcal{P}}},d_{\mathcal{M}_{e,\mathcal{P}}}).
\]
To endow $\mathcal{M}_{e,\mathcal{P}}$ with a canonical measure let $p_{e,\mathcal{P}}$ denote the canonical projection from $\mathcal{T}_e$ to the quotient space $\mathcal{T}_e/\sim_{E_{\mathcal{P}}}$. We define $\pi_{e,\mathcal{P}}:=\pi^e\circ p_{e,\mathcal{P}}^{-1}$, where $\pi^e$ is the image measure on $\mathcal{T}_e$ of the Lebesgue measure $\lambda$ on $[0,\zeta]$ by the canonical projection $p_e$ of $[0,\zeta]$ onto $\mathcal{T}_e$. So, $\pi_{e,\mathcal{P}}=\lambda\circ p_e^{-1} \circ p_{e,\mathcal{P}}^{-1}$. We note that the restriction of $p_{e,\mathcal{P}}$ to $\mathcal{T}_e$ is $p_e$.

For every $\zeta>0$, as in \cite{addario2012continuum}, we define a tilted excursion of length $\zeta$ to be a random variable that takes values in $E$ whose distribution is characterized by 
\[
\mathbf{P}(\tilde{e}\in \mathcal{E})=\frac{\mathbf{E}\left(\one_{\{e\in \mathcal{E}\}}\exp\left(\int_{0}^{\zeta} e(t) dt\right)\right)}{\mathbf{E}\left(\exp\left(\int_{0}^{\zeta} e(t) dt\right)\right)},
\]
for every measurable $\mathcal{E}\subseteq E$. We note here that the $\sigma$-algebra on $E$ is the one generated by the open sets with respect to the supremum norm on $C(\mathbb{R}_{+},\mathbb{R}_{+})$. Write $\mathcal{M}^{(\zeta)}$ for the random compact metric space distributed as $(\mathcal{M}_{\tilde{e},\mathcal{P}},2 d_{\mathcal{M}_{\tilde{e},\mathcal{P}}})$, where $\tilde{e}$ is a tilted Brownian excursion of length $\zeta$ and the random pointset of interest $\mathcal{P}$ is a Poisson point process on $\mathbb{R}_{+}^2$ of unit intensity with respect to the Lebesgue measure independent of $\tilde{e}$. We now give an alternative description of $\mathcal{M}_{\tilde{e},\mathcal{P}}$. It is easy to prove that the number $|\mathcal{P}\cap \tilde{e}|$ of vertex identifications is a Poisson random variable with mean $\int_{0}^{\zeta} \tilde{e}(u) du$. Given that $|\mathcal{P}\cap \tilde{e}|=k$, the co-ordinate $u_{(t,x)}$ has density 
\[
\frac{\tilde{e}(u)}{\int_{0}^{\zeta} \tilde{e}(t) dt}
\]
on $[0,\zeta]$, and given $u_{(t,x)}$, its pair $v_{(t,x)}$ is uniformly distributed on $[0,\tilde{e}(u_{(t,x)})]$. The other $k-1$ vertex identifications are distributed accordingly and independently of the pair $(u_{(t,x)},v_{(t,x)})$. After introducing notation, we are in the position to write the limit of the largest connected component, say $\mathcal{C}_1^n$, as $\mathcal{M}^{(C_1)}$, where $C_1$ has the distribution of the length of the longest excursion of the reflected upon its minimum parabolic Brownian motion as defined in \eqref{parabola}. Moreover, the longest excursion, when conditioned to have length $C_1$, is distributed as a tilted excursion $\tilde{e}$ with length $C_1$. The following convergence is a simplified version of \cite[Theorem 2]{addario2012continuum}. As $n\to \infty$, 

\begin{equation}
\left(n^{-2/3} C_1^n,\left(V(\mathcal{C}_1^n),n^{-1/3} d_{\mathcal{C}_1^n}\right)\right)\longrightarrow \left(C_1,\left(\mathcal{M},d_{\mathcal{M}}\right)\right),
\end{equation}
in distribution, where conditional on $C_1$, $\mathcal{M}\,{\buildrel (d) \over =}\,\mathcal{M}^{(C_1)}$. Moreover, it was shown in \cite{croydon2012scaling}, that the discrete-time simple random walks $X^{\mathcal{C}_1^n}$ on $\mathcal{C}_1^n$, started from a distinguished vertex $\rho^n$ satisfy a distributional convergence of the form

\begin{equation}
\left(n^{-1/3} X_{\lfloor n t \rfloor}^{\mathcal{C}_1^n}\right)_{t\ge 0}\to \left(X_t^{\mathcal{M}}\right)_{t\ge 0},
\end{equation}
where $X^{\mathcal{M}}$ is a diffusion on $\mathcal{M}$, started from a distinguished point $\rho\in \mathcal{M}$. The convergence of the associated stationary probability measures, say $\pi^n$, was not directly proven in \cite{croydon2012scaling}, although the hard work to this direction has been done. More specifically, see \cite[Lemma 6.3]{croydon2012scaling}. The results above can be reformulated in the following distributional convergence in terms of the pointed extended Gromov-Hausdorff topology.
\begin{equation} 
\left(\left(V(\mathcal{C}_1^n),n^{-1/3} d_{\mathcal{C}_1^n},\rho^n\right),\pi^n,\left(n^{-1/3} X^{\mathcal{C}_1^n}_{\lfloor n t \rfloor}\right)_{t\ge 0}\right)\longrightarrow \left(\left(\mathcal{M},d_{\mathcal{M}},\rho\right),\pi^{\mathcal{M}},X^{\mathcal{M}}\right).
\end{equation}

Now, we describe how to generate a connected component on a fixed number of vertices. To any such component we can associate a spanning subtree, the depth-first tree by considering the following algorithm. The initial step places the vertex with label 1 in a stack and declares it open. In the next step vertex 1 is declared as explored and is removed from the start of the stack, where we place in increasing order the neighbors of 1 that have not been seen (open or explored) yet. We proceed inductively. When the set of open vertices becomes empty the procedure finishes. It is obvious that the resulting graph that consists of edges between a vertex that was explored at a given step and a vertex that has not been seen yet at the same step, is a tree. For a connected graph $G$ with $m$ vertices, we refer to this tree as the depth-first tree and write $T(G)$. For $i=0,...,m-1$, let $X(i):=|O(i)|-1$ be the number of vertices seen but not yet fully explored at step $i$. The process $(X(i): 0\le i<m)$ is called the depth-first walk of the graph $G$.

Let $\mathbb{T}_m$ be the set of (unordered) tree labelled by $[m]$. For $T\in \mathbb{T}_m$, its associated depth-first tree is $T$ itself. We call an edge permitted by the first-depth procedure run on $T$ if its addition produces the same depth-first tree. Exactly $X(i)$ edges are permitted at step $i$, and therefore the total number of permitted edges is given by 
\[
a(T):=\sum_{i=0}^{m-1} X(i),
\]
which is called the area of $T$. Given a tree $T$ and a connected graph $G$, $T(G)=T$ if and only if $G$ can be obtained from $T$ by adding a subset of permitted edges by the depth-first procedure. Therefore, writing $\mathbb{G}_T$ for the set of connected graphs $G$ that satisfy $T(G)=T$, we have that $\{\mathbb{G}_T: T\in \mathbb{T}_m\}$ is a partition of the connected graphs on $[m]$, and that the cardinality of $\mathbb{G}_T$ is $2^{a(T)}$, since every permitted edge is included or not.

Back to the question on how to generate a connected component, write $G_m^p$ for the graph with the same distribution as $G(m,p)$ conditioned to be connected. Thus, we focus on generating $G_m^p$ instead.

\begin{lemma} [\textbf{Addario-Berry, Broutin, Goldschmidt \cite{addario2012continuum}}] \label{gener}
Fix $p\in (0,1)$. Pick a random tree $\tilde{T}_m^p$ that has a ``tilted'' distribution which is biased in favor of trees with large area. Namely, pick $\tilde{T}_m^p$ in such a way that
\[
P(\tilde{T}_m^p=T)\propto (1-p)^{-a(T)}, \qquad T\in \mathbb{T}_m.
\]
Add to $\tilde{T}_m^p$ each of the $a(\tilde{T}_m^p)$ permitted edges independently with probability $p$. Call the graph generated $\tilde{G}_m^p$. Then, $\tilde{G}_m^p$ has the same distribution as $G_m^p$.
\end{lemma} 

We use $\rho^m$ to denote the root of $\tilde{T}_m^p$. In what follows we give a detailed description on how we can transfer the results proved in Section \ref{CGWT}. We denote by  $\tilde{V}_m=(\tilde{V}_m(i): 0\le i\le 2 m)$ the contour process of $\tilde{T}_m^p$, and by  $\tilde{v}_m=(((m/\zeta)^{-1/2} \tilde{V}_m(2 (m/\zeta) s): 0\le s\le \zeta)$ the rescaled contour process of positive length $\zeta$ as well. We start by showing that, for some $\alpha>0$, the sequence $||\tilde{v}_m||_{H_{\alpha}}$ of H\"older norms is tight.

\begin{lemma} \label{tiltcont1}
Suppose that p=p(m) in such a way that $m p^{2/3}\to \zeta$, as $m\to \infty$. There exists $\alpha\in (0,1/2)$ such that for every $\varepsilon>0$ there exists a finite real number $M_{\varepsilon}$ such that 
\begin{equation} 
P\left(\sup_{s,t\in [0,1]}\frac{|\tilde{v}_m(s)-\tilde{v}_m(t)|}{|t-s|^{\alpha}}\le M_{\varepsilon}\right)\ge 1-\varepsilon.
\end{equation}
\end{lemma}

\begin{proof}
For notational simplity we assume that $\zeta=1$. The general result follows by Brownian scaling. Let $T_m$ be a tree chosen uniformly from $[m]$. Write $V_m$ and $v_m$ for its associated contour process and normalized contour process respectively. We note here that Theorem \ref{prepar2} is stated in the more general situation of Galton-Watson trees with critical offspring distribution that has finite variance and exponential moments, conditioned to have $m$ vertices. If the offspring distribution to be Poisson with mean $1$, then the conditioned tree is a uniformly distributed labelled tree. Then, by Lemma \ref{gener}, if $K_{\varepsilon}$ is the finite real number for which \eqref{proof2} holds,
\[
P\left(\sup_{s,t\in [0,1]}\frac{|\tilde{v}_m(s)-\tilde{v}_m(t)|}{|t-s|^{\alpha}}\ge K_{\varepsilon}\right)=\frac{E\left[\one_{\left\{\sup_{s,t\in [0,1]}\frac{|v_m(s)-v_m(t)|}{|t-s|^{\alpha}}\ge K_{\varepsilon}\right\}} (1-p)^{-a(T_m)}\right]}{E\left[(1-p)^{-a(T_m)}\right]}.
\]
Since $m p^{2/3}\to 1$, as $m\to \infty$, there exists $c>0$ such that $p\le c m^{-3/2}$, for every $m\ge 1$. Moreover, from \cite[Lemma 14]{addario2012continuum} we can find universal constants $K_1$, $K_2>0$ such that $E\left[(1-p)^{-\xi a(T_m)}\right]<K_1 e^{K_2 c^2 \xi^2}$, for every $\xi>0$. Using this along with the Cauchy-Schwarz inequality yields
\begin{align} \label{taram}
&P\left(\sup_{s,t\in [0,1]}\frac{|\tilde{v}_m(s)-\tilde{v}_m(t)|}{|t-s|^{\alpha}}\ge K_{\varepsilon}\right) \nonumber \\
&\le \frac{P\left(\sup_{s,t\in [0,1]}\frac{|v_m(s)-v_m(t)|}{|t-s|^{\alpha}}\ge K_{\varepsilon}\right)^{1/2} \left(E\left[(1-p)^{-2 a(T_m)}\right]\right)^{1/2}}{E\left[(1-p)^{-a(T_m)}\right]}\le \frac{(\varepsilon K_1e^{4 K_2 c^2})^{1/2}}{E\left[(1-p)^{-a(T_m)}\right]}. 
\end{align}
Finally, recall that $a(T_m)=\sum_{i=0}^{m-1} X_m(i)$, where $(X_m(i): 0\le i\le m)$ is the depth-first walk associated with $T_m$ (for convenience we have put $X_m(m)=0$). From \cite[Theorem 3]{marckert2003depth} we know that, as $m\to \infty$
\[
(m^{-1/2} X_m(\lfloor m t \rfloor))_{t\in [0,1]}\to (e(t))_{t\in [0,1]},
\]
in distribution in $D([0,1],\mathbb{R}_{+})$, where $(e(t))_{t\in [0,1]}$ is a normalized Brownian excursion. Writing 
\[
(1-p)^{-a(T_m)}=(1-p)^{-\sum_{i=0}^{m-1} X_m(i)}=(1-p)^{-m^{3/2} \int_{0}^{1} m^{-1/2} X_m(\lfloor m t \rfloor) dt}
\]
and using that the sequence $(1-p)^{-a(T_m)}$ is uniformly integrable, we deduce that 
\[
E\left[(1-p)^{-a(T_m)}\right]\to E\left[\exp \left(\int_{0}^{1} e(u) du\right)\right]>0,
\]
as $m\to \infty$. The desired result follows from \eqref{taram} and the fact that the sequence of random variables $\{||\tilde{v}_m||_{H_{\alpha}}\}_{m\in I}$ is bounded in probability (tight) whenever $I$ is finite.

\end{proof}

It is now immediate to check that the local times $(L^m_t(x))_{x\in V(G_m^p),t\ge 0}$ of the corresponding simple random walk on $G_m^p$ are equicontinuous under the annealed law (which is defined similar to \eqref{exemplary}) after rescaling. The proof of the next lemma relies heavily on the same methods used to establish Proposition \ref{own}, and therefore we will make use of the parts that remain unchanged.

\begin{lemma} \label{tiltcomp}
Suppose that p=p(m) in such a way that $m p^{2/3}\to \zeta$, as $m\to \infty$. For every $\varepsilon>0$ and $T>0$,
\[
\lim_{\delta\rightarrow 0} \limsup_{m\to \infty} \mathbb{P}_{\rho^m} \left(\sup_{\substack{y,z\in V(G_m^p): \\ m^{-1/2} R_{G_m^p}(y,z)<\delta}} \sup_{t\in [0,T]} m^{-1/2} |L_{m^{3/2} t}^m(y)-L_{m^{3/2} t}^m(z)|\ge \varepsilon \Bigg | s(G_m^p)=s\right)=0.
\]
\end{lemma}

\begin{proof}
Again, for simplicity let $\zeta=1$. The general result follows by Brownian scaling. Call $\tilde{G}_m^p$ the graph generated by the process of adding $\text{Bin}(a(\tilde{T}_m^p),p)$ number of surplus edges to  $\tilde{T}_m^p$. We view $\tilde{G}_m^p$ as the metric space $\tilde{T}_m^p$ that includes the edges (of length 1) that have been added. Recall that $\tilde{G}_m^p$ has the same distribution as $G_m^p$. From Lemma \ref{tiltcont1}, there exists $C_1>0$ and $\alpha\in (0,1/2)$ such that 
\begin{equation} \label{tiltcont2}
\tilde{v}_m(s)+\tilde{v}_m(t)-2 \min_{r\in [s\wedge t,s\vee t]} \tilde{v}_m(r)\le C_1 |t-s|^{\alpha}, \qquad \forall s,t\in [0,1]
\end{equation}
with probability arbitrarily close to 1. Conditioned on $\tilde{v}_m$ satisfying \eqref{tiltcont2}, and given $t_1,t_2\in [0,1]$, with $2 m t_1$ and $2 m t_2$ integers, such that $p_{\tilde{v}_m}(t_1)=y$ and $p_{\tilde{v}_m}(t_2)=z$, the resistance on $\tilde{G}_m^p$ between $y$ and $z$ is smaller than the total length of the path between $y$ and $z$ on $\tilde{T}_m^p$. Therefore, using \eqref{tiltcont2}
\begin{align*}
R_{\tilde{G}_m^p}(y,z)&\le \tilde{V}_m(2 m t_1)+\tilde{V}_m(2 m t_2)-2 \min_{r\in [t_1\wedge t_2,t_1\vee t_2]} \tilde{V}_m(2 m r)\\
&=m^{1/2} (\tilde{v}_m(t_1)+\tilde{v}_m(t_2)-2 \min_{r\in [t_1\wedge t_2,t_1\vee t_2]} \tilde{v}_m(r))=m^{1/2} d_{\tilde{v}_m}(t_1,t_2)\le C_1 m^{1/2} |t_1-t_2|^{\alpha},
\end{align*}
which also indicates that, on the event that \eqref{tiltcont2} holds, the maximum resistance $r(\tilde{G}_m^p)$ is bounded above by $C m^{1/2} 2^{\alpha}$. Moreover, conditioning on $s(\tilde{G}_m^p)=s$, $m(\tilde{G}_m^p)=2 (m+s)$. Hence, by \eqref{proof4}, which still holds by replacing $\mathcal{T}_n$ by $\tilde{G}_m^p$, we can show that, conditioned on $\tilde{v}_m$ satisfying \eqref{tiltcont2}, for any fixed $q>0$,
\begin{align*}
\mathbf{E}_{\rho^m}^{\tilde{G}_m^p} \left[\sup_{t\in [0,T]} m^{-1/2} \left|L_{m^{3/2} t}^m(y)-L_{m^{3/2} t}^m(z) \right|^q\bigg| s(\tilde{G}_m^p)=s\right]&\le C_2  (m^{-1/2} R_{\tilde{G}_m^p}(y,z))^{q/2} 2^{-\alpha q}\\
&\le C_3 \left|(t_1/4)-(t_2/4)\right|^{\alpha q/2},
\end{align*}
where $t_1$ and $t_2$ are such that $2 m t_1$ and $2 m t_2$ are integers, such that $p_{\tilde{v}_m}(t_1)=y$ and $p_{\tilde{v}_m}(t_2)=z$. The rest of proof is finished in the manner of Proposition \ref{own}, and therefore we omit it.

\end{proof}

For notational simplicity, the next result is stated for the largest connected component on the critical window. In fact, it holds for the family of the $i$-th largest connected components, $i\ge 1$. In this case, let us denote by $\mathcal{C}_1^n$ the largest connected component of $G(n,n^{-1}+\lambda n^{-4/3})$, and by $(L_t^n(x))_{x\in V(\mathcal{C}_1^n),t\ge 0}$ the local times of the simple random walk on $\mathcal{C}_1^n$.

\begin{proposition} \label{ownown}
For every $\varepsilon>0$ and $T>0$, 
\[
\lim_{\delta\rightarrow 0} \limsup_{n\to \infty} \mathbb{P}_{\rho^n} \left(\sup_{\substack{y,z\in V(\mathcal{C}_1^n): \\ n^{-1/3} R_{\mathcal{C}_1^n}(y,z)<\delta}} \sup_{t\in [0,T]} n^{-1/3} |L_{n t}^n(y)-L_{n t}^n(z)|\ge \varepsilon\right)=0.
\]
\end{proposition}

\begin{proof}
Conditioning on the size and the surplus of $\mathcal{C}_1^n$, we observe that for every $\varepsilon_1>0$ and $T_1>0$,
\begin{align} \label{boring}
&\mathbb{P}_{\rho^n} \left(\sup_{\substack{y,z\in V(\mathcal{C}_1^n): \\ n^{-1/3} R_{\mathcal{C}_1^n}(y,z)<\delta_1}} \sup_{t\in [0,T_1]} n^{-1/3} |L_{n t}^n(y)-L_{n t}^n(z)|\ge \varepsilon_1\right) 
\\ \nonumber \le
&\sum_{m=1}^{\lfloor A n^{2/3} \rfloor} \sum_{s=0}^{S} \mathbb{P}_{\rho^n} \left(\sup_{\substack{y,z\in V(\mathcal{C}_1^n): \\ n^{-1/3} R_{\mathcal{C}_1^n}(y,z)<\delta_1}} \sup_{t\in [0,T_1]} n^{-1/3} |L_{n t}^{n}(y)-L_{n t}^{n}(z)|\ge \varepsilon_1\bigg ||V(\mathcal{C}_1^n)|=m,s(\mathcal{C}_1^n)=s\right) 
\\ \nonumber \cdot
&P(|V(\mathcal{C}_1^n)|=m,s(\mathcal{C}_1^n)=s)+P(|V(\mathcal{C}_1^n)|>\lfloor A n^{2/3} \rfloor)+P(s(\mathcal{C}_1^n)>S),
\end{align}
for large enough constants $A$ and $S$. Since the largest components on the critical window $p=n^{-1}+\lambda n^{-4/3}$, $\lambda\in \mathbb{R}$ have sizes of order $n^{2/3}$, we have that $\mathcal{C}_1^n$ conditioned to have size $m$ and surplus $s$ has the same distribution with $G_m^p$, where $m$ and $p=p(m)$ are in such a way that $m p^{2/3}\to \zeta$, as $m\to \infty$, for some $\zeta>0$. Therefore, we acan bound \eqref{boring} by 
\begin{align*}
&\sup_{m\ge 1} \mathbb{P}_{\rho^n} \left(\sup_{\substack{y,z\in V(G_m^p): \\ m^{-1/2} R_{G_m^p}(y,z)<\delta_2}} \sup_{t\in [0,T_2]} m^{-1/2} |L_{m^{3/2} t}^{m}(y)-L_{m^{3/2} t}^{m}(z)|\ge \varepsilon_2\bigg |s(G_m^p)=s\right) 
\\ \\+
&P(|V(\mathcal{C}_1^n)|>\lfloor A n^{2/3} \rfloor)+P(s(\mathcal{C}_1^n)>S).
\end{align*}
From Theorem \ref{PhT1}, 
\begin{equation} \label{boring2}
\lim_{A\to \infty} \limsup_{n\to \infty} P(n^{-2/3} |V(\mathcal{C}_1^n)|>A)=0.
\end{equation}
Furthermore, as $n\to \infty$, 
\[
s(\mathcal{C}_1^n)\to \text{Poi}\left(\int_{0}^{\zeta} \tilde{e}(u) du\right),
\]
where $\text{Poi}\left(\int_{0}^{\zeta} \tilde{e}(t) dt\right)$ denotes a Poisson random variable with mean the area under a tilted excursion of length $\sigma$ (see \cite[Corollary 23]{addario2012continuum}), and as a consequence tightness of process that encodes the surplus of $\mathcal{C}_1^n$ follows, i.e.
\begin{equation} \label{boring3}
\lim_{S\to \infty} \limsup_{n\to \infty} P(s(\mathcal{C}_1^n)>S).
\end{equation}
Now, the proof is finished by combining \eqref{boring2} and \eqref{boring3} with the equicontinuity result of Lemma \ref{tiltcomp}.

\end{proof}

\subsubsection{Continuity of blanket times of Brownian motion on \texorpdfstring{$\mathcal{M}$}{M}} \label{excm}

To prove continuity of the $\varepsilon$-blanket time of the Brownian motion on $\mathcal{M}$, we first define a $\sigma$-finite measure on the product space of positive excursions and random pointsets of $\mathbb{R}_{+}^2$. Throughout this section we denote the Lebesgue measure on $\mathbb{R}_{+}^2$ by $\ell$. We define the aforementioned measure $\mathbf{N}(d(e,\mathcal{P}))$ by setting 
\begin{equation} \label{meas1}
\mathbf{N}(d e,|\mathcal{P}|=k, (d x_1,...,d x_k)\in A_1\times...\times A_k):=\int_{0}^{\infty}  f_L(l) \mathbb{N}_l (d e) \frac{e^{-1}}{k!} \prod_{i=1}^{k} \frac{\ell(A_i\cap A_e)}{\ell(A_e)},
\end{equation} 
where $f_L(l):=d l/\sqrt{2 \pi l^3}$, $l\ge 0$ gives the density of the length of the excursion $e$ and $A_e:=\{(t,x): 0\le x\le e(t)\}$ denotes the area under its graph. In other words, the measure picks first an excursion length according to $f_L(l)$ and, given $L=l$, it picks a Brownian excursion of that length. Then, independently of $e$ it chooses $k$ points according to a Poisson with unit mean, which distributes uniformly on the area under the graph of $e$. 

It turns out that this is an easier measure to work with when applying our scaling argument to prove continuity of the blanket times. Also, as will see later, $\mathbf{N}$ is absolutely continuous with respect to the canonical measure $\mathbf{N}^{t,\lambda}(d(e,\mathcal{P}))$ that first at time $t$ picks a tilted Brownian excursion $e$ of a randomly chosen length $l$, and then independently of $e$ chooses $k$ points distributed as a Poisson random variable with mean $\int_{0}^{l} e(t) dt$, which as before are distributed uniformly on the area under the graph of $e$. To fully describe this measure let $\mathbb{N}^{t,\lambda}$ denote the measure (for excursions starting at time $t$) associated to $B^{\lambda}_t-\inf_{0\le s\le t} B^{\lambda}_s$, first stated by Aldous in \cite{aldous1997brownian}. We note that $\mathbb{N}^{t,\lambda}=\mathbb{N}^{0,\lambda-t}$ and thus it suffices to write $\mathbb{N}^{0,\lambda}$ for every $\lambda\in \mathbb{R}$. For every measurable $A\in E$,
\[
\mathbb{N}^{0,\lambda}(A)=\int_{0}^{\infty} \mathbb{N}^{0,\lambda}_l(A) f_{L}(l) F_{\lambda}(l) \mathbb{N}_l\left(\exp\left(\int_{0}^{l} e(u) du\right)\right),
\]
where $\mathbb{N}_{l}^{0,\lambda}$ is a shorthand for the excursion measure $\mathbb{N}^{0,\lambda}$, conditioned on the event $\{\tilde{L}=l\}$ and $F_{\lambda}(l):=\exp\left(-1/6\left(\lambda^3+(l-\lambda)^3\right)\right)$. For simplicity, let $g_{\tilde{L}}(l,\lambda):=f_{L}(\lambda) F_{\lambda}(l) \mathbb{N}_l\left(\exp\left(\int_{0}^{l} e(u) du\right)\right)$. In analogy with \eqref{meas1} we characterize $\mathbf{N}^{t,\lambda}(d(e,\mathcal{P}))$ by setting 
\begin{align} \label{meas3}
&\mathbf{N}^{t,\lambda}(d e,|\mathcal{P}|=k, (d x_1,...,d x_k)\in A_1\times...\times A_k)
\nonumber \\ :=
&\int_{0}^{\infty} g_{\tilde{L}}(l,\lambda-t) \mathbb{N}^{t,\lambda}_l (d e) \exp\left(-\int_{0}^{l} e(u) du\right)\frac{\left(\int_{0}^{l} e(u) du\right)^k}{k!} \prod_{i=1}^{k} \frac{\ell(A_i\cap A_e)}{\ell(A_e)}.
\end{align} 
After calculations that involve the use of the Cameron-Martin-Girsanov formula \cite[Chapter IX, (1.10) Theorem]{revuz1999continuous} (for the entirety of those calculations one can consult \cite[Section 5]{addario2012continuum}), one deduces that 
\[
\mathbb{N}_l^{t,\lambda}(d e)=\exp\left(\int_{0}^{l} e(u) du\right) \frac{\mathbb{N}_l(d e)}{\mathbb{N}_l\left(\exp\left(\int_{0}^{l} e(u) du\right)\right)},
\]
and as a consequence the following expression for the Radon-Nikodym derivative is valid.
\begin{equation} \label{meas4}
\frac{d \mathbf{N}^{t,\lambda}}{d \mathbf{N}}=\frac{F_{\lambda-t}(l) \left(\int_{0}^{l} e(u) du\right)^k/k!}{e^{-1}/k!}=\exp\left(1-\frac{1}{6}\left(\lambda^3+(l-\lambda+t)^3\right)\right) \left(\int_{0}^{l} e(u) du\right)^k.
\end{equation}

The main reason to consider $\mathbf{N}$ instead of $\mathbf{N}^{t,\lambda}$ is that it can easily be seen to enjoy the same scaling property with the normalized It\^o excursion measure. Recall that for every $b>0$, the mapping $\Theta_b:E\to E$ defined by setting $\Theta_b(e)(t):=\sqrt{b} e(t/b)$, for every $e\in E$, and $t\ge 0$. As we saw in subsection \ref{unes} acts on the real tree coded by $e$ scaling its distance and invariant probability measure appropriately. Let $\mathcal{M}_{e,\mathcal{P}}$, where $e$ is a Brownian excursion of length $\zeta$ and pointset $\mathcal{P}$, that is a Poisson process on $\mathbb{R}_{+}^2$ of intensity one with respect to the Lebesgue measure independent of $e$. If $|\mathcal{P}\cap e|$ is distributed as a Poisson random variable with mean $\int_{0}^{\zeta} e(u) du$, then $|\mathcal{P}\cap \Theta_b(e)|$ has law given by a Poisson distribution with mean $b^{3/2} \int_{0}^{\zeta} e(u) du$. Moreover, conditioned on $|\mathcal{P}\cap e|$, if the coordinates of a point $(u_{(t,x)},v_{(t,x)})$ in $\mathcal{P}\cap e$ have densities proportional to $e(u)$ for $u_{(t,x)}$ and, conditioned on $u_{(t,x)}$, uniformly on $[0,e(u_{(t,x)})]$ for $v_{(t,x)}$ respectively, then conditioned on $|\mathcal{P}\cap \Theta_b(e)|$, the coordinates of a point $(u^b_{(t,x)},v^b_{(t,x)})$ in $\mathcal{P}\cap \Theta_b(e)$ has coordinates distributed as $b u_{(t,x)}$ in the case of $u^b_{(t,x)}$, and conditioned on $u_{(t,x)}$, uniform on $[0,\sqrt{b} u_{(t,x)}]$ in the case of $v^b_{(t,x)}$ respectively. From now on we use $\Theta_b(e,\mathcal{P})$ to denote the mapping from the product space of positive excursions and pointsets of the upper half plance onto itself that rescales $e$ as $\Theta_b(e)$ and repositions the collection of points in $\mathcal{P}$ as described above.

Starting with how the distance rescales under the application of $\Theta_b$, $b>0$, it is immediate from the definition of the quasi-metric $d_{\mathcal{M}_{e,\mathcal{P}}}$ that $d_{\mathcal{M}_{\Theta_b(e,\mathcal{P})}}=b^{1/2} d_{\mathcal{M}_{e,\mathcal{P}}}$. Let $\mathcal{L}(\mathcal{T}_e)$ denote the set of leaves of $\mathcal{T}_e$. Then, $\pi^e$ has full support on $\mathcal{L}(\mathcal{T}_e)$ and  $\pi^e(\mathcal{L}(\mathcal{T}_e))=\zeta(e)$. Consider the set 
\[
I=\{\sigma\in \mathcal{L}(\mathcal{T}_e): p_{e,\mathcal{P}}(\sigma)\in A\},
\]
for a measurable subset $A$ of $\mathcal{M}_{e,\mathcal{P}}$. Then, from the definition of $\pi_{e,\mathcal{P}}$, we have that $\pi_{e,\mathcal{P}}(A)=\pi^e(I)$, and consequently $\pi_{\Theta_b(e,\mathcal{P})}(A)=\pi^{\Theta_b(e)}(I)$. As we examined before $\pi^{\Theta_b(e)}(I)=b \pi^e(I)$, and from this it follows that $\pi_{\Theta_b(e,\mathcal{P})}(A)=b \pi_{e,\mathcal{P}}(A)$, for every measurable subset $A$ of $\mathcal{M}_{e,\mathcal{P}}$. Finally, since $\mathbb{N}\circ \Theta_b^{-1}=\sqrt{b} \mathbb{N}$ and using the fact that $\ell(A_i\cap A_e)/\ell(A_e)$ in \eqref{meas1} is scale invariant under $\Theta_b$, it follows that 
$\mathbf{N}\circ \Theta_b^{-1}=\sqrt{b} \mathbf{N}$, i.e. $\mathbf{N}(d(e,\mathcal{P}))$ rescales exactly like $\mathbb{N}(d e)$. 

We now have all the ingredients to prove continuity of the blanket times of the Brownian motion on $\mathcal{M}$. We briefly describe the arguments that have been already used in establishing Proposition \ref{pr2}. Let $\tau_{\text{bl}}^{e,\mathcal{P}}(\varepsilon)$ denote the $\varepsilon$-blanket time of the Brownian motion $X^{e,\mathcal{P}}$ on $\mathcal{M}_{e,\mathcal{P}}$ started from a distinguished vertex $\bar{\rho}$, for some $\varepsilon\in (0,1)$. Taking the expectation of the law of $\tau_{\text{bl}}^{e,\mathcal{P}}(\varepsilon)$, $\varepsilon\in (0,1)$ against the $\sigma$-finite measure $\mathbf{N}$ (see \eqref{sc5}), using Fubini and the monotonicity of the blanket times, yields
\[
\mathbf{P}_{\bar{\rho}}^{e,\mathcal{P}}\left(\tau_{\text{bl}}^{e,\mathcal{P}}(\varepsilon-)=\tau_{\text{bl}}^{e,\mathcal{P}}(\varepsilon)\right)=1,
\]
$\lambda$-a.e. $\varepsilon$, $\mathbf{N}$-a.e. $(e,\mathcal{P})$, where $\mathbf{P}_{\bar{\rho}}^{e,\mathcal{P}}$ denotes the law of $X^{e,\mathcal{P}}$. The rest of the argument relies on impoving such a statement to hold for every $\varepsilon\in (0,1)$ by using scaling.

In the transformed glued metric space $\mathcal{M}_{\Theta_b(e,\mathcal{P})}$ the Brownian motion admits $\mathbf{P}_{\bar{\rho}}^{e,\mathcal{P}}$-a.s. jointly continuous local times $(\sqrt{b} L_{b^{-3/2} t}(x))_{x\in \mathcal{M}_{e,\mathcal{P}},t\ge 0}$. This enough to infer that, for every $\varepsilon\in (0,1)$ and $b>1$, the continuity of the $\varepsilon$-blanket time variable of $\mathcal{M}_{e,\mathcal{P}}$ is equivalent to the continuity of the $b^{-1} \varepsilon$-blanket time variable of $\mathcal{M}_{\Theta_b(e,\mathcal{P})}$, and consequently as in the proof of Proposition \ref{pr2} applying our scaling argument implies 
\[
\mathbf{P}_{\bar{\rho}}^{e,\mathcal{P}}\left(\tau_{\text{bl}}^{e,\mathcal{P}}(\varepsilon-)=\tau_{\text{bl}}^{e,\mathcal{P}}(\varepsilon)\right)=1,
\]
$\mathbf{N}$-a.e. $(e,\mathcal{P})$. Recall that, conditional on $C_1$, $\mathcal{M}\,{\buildrel (d) \over =}\,\mathcal{M}^{(C_1)}$, where $C_1$ is the length of the longest excursion of the process defined in \eqref{parabola}, which is distributed as a tilted excursion of that length. Then, applying again our scaling argument as in the end of the proof of Proposition \ref{pr2}, conditional on $C_1$, we deduce
\[
\mathbf{P}_{\rho}^{\mathcal{M}}\left(\tau_{\text{bl}}^{\mathcal{M}}(\varepsilon-)=\tau_{\text{bl}}^{\mathcal{M}}(\varepsilon)\right)=1,
\]
$\mathbf{N}_{C_1}$-a.e. $(e,\mathcal{P})$, where $\mathbf{N}_l$ is the version of $\mathbf{N}$ defined in \eqref{meas1} conditioned on the event $\{L=l\}$. Since the canonical measure $\mathbf{N}^{0,\lambda}_{C_1}$ is absolutely continuous with respect to $\mathbf{N}_{C_1}$ as it was shown in \eqref{meas4}, the above also yields that conditional on $C_1$, $\mathbf{N}^{0,\lambda}_{C_1}$-a.e. $(e,\mathcal{P})$, $\varepsilon\mapsto \tau_{\text{bl}}^{\mathcal{M}}(\varepsilon)$ is continuous $\mathbf{P}_{\rho}^{\mathcal{M}}$-a.s. 

Given this result and Proposition \eqref{ownown} the following convergence theorem follows from Theorem \ref{Mth}. Here, for a particular real value of $\lambda$ and conditional on $C_1$,

\begin{equation} \label{lastbits}
\mathbb{P}_{\rho}(\cdot):=\int \mathbf{P}_{\rho}^{\mathcal{M}}(\cdot ) \mathbf{N}_{C_1}^{0,\lambda}(d(e,\mathcal{P})),
\end{equation}
formally defines the annealed measure for suitable events.
\begin{theorem}
Fix $\varepsilon\in (0,1)$. If $\tau_{\textnormal{bl}}^{n}(\varepsilon)$ is the $\varepsilon$-blanket time variable of the random walk on $\mathcal{C}_1^n$, started from its root $\rho^n$, then 
\[
\mathbb{P}_{\rho^n}\left(n^{-1} \tau_{\textnormal{bl}}^n(\varepsilon)\le t\right)\to \mathbb{P}_{\rho}\left(\tau_{\textnormal{bl}}^{\mathcal{M}}(\varepsilon)\le t\right),
\]
for every $t\ge 0$, where $\tau_{\textnormal{bl}}^{\mathcal{M}}(\varepsilon)\in (0,\infty)$ is the $\varepsilon$-blanket time variable of the Brownian motion on $\mathcal{M}$, started from $\rho$. 
\end{theorem}

\subsection{Critical random graph with prescribed degrees}

Let $\text{G}_{n,d}$ denote the space of all simple graphs labelled by $[n]$ such that the $i$-th vertex has degree $d_i$, $i\ge 1$, for $1\le i\le n$. We denote the vector $(d_i: i\in [n])$ of the prescribed degree sequence by $d$, where $\ell_n:=\sum_{i\in [n]} d_i$ is assumed even. Write $\bar{\text{G}}_{n,d}$ for $\text{G}_{n,d}$ with the difference that we allow self-loops as well as the occurence of multiple edges between the same pair of vertices. Then, the configuration model is the random multigraph in $\text{G}_{n,d}$ constructed as follows. Assign each vertex $i$ with $d_i$ half-edges, labelling them arbitrarily by $1,...,\ell_n$. The multigraph $\text{M}^n(d)$ produced by a uniform random pairing of the half-edges to create full edges is called the configuration model. In particular, we look at prescribed degree sequences that satisfy the following assumptions.

\begin{assum} \label{gracias}
Let $D_n$ is a random variable with distribution given by 
\[
P(D_n=i)=\frac{\# \{j: d_j=i\}}{n}.
\] 
In other words, $D_n$ has the law of the degree of a vertex chosen uniformly at random from $[n]$. Suppose that $D_n\xrightarrow{(d)} D$, for a limiting random variable $D$ such that $P(D=1)>0$. Moreover, assume the following as $n\to \infty$,
\begin{align*}
&\text{(i) Convergence of the third moments: }E(D_n^3)\to E(D^3)<\infty,
\\
&\text{(ii) Scaling critical window: }\frac{E(D_n (D_n-1))}{E(D_n)}=1+\lambda n^{-1/3}+o(n^{-1/3}), \text{ for some } \lambda\in \mathbb{R}. \text{ In particular, }\\
&E(D^2)=2 E(D). 
\end{align*}
\end{assum}

\begin{remark}
We remark here that the configuration model with random i.i.d. degrees sampled from a distribution with $E(D^3)<\infty$ treated in \cite{joseph2014component} meets the assumptions introduced above almost surely. Especially, (ii) is satisfied for $\lambda=0$ and corresponds to the critical case, i.e. if $E(D^2)<2 E(D)$ there is no giant component with probability tending to 1, as $n\to \infty$. In addition, $E(D^2)>2 E(D)$ sees the emergence of a unique giant component with probability tending to 1, as $n\to \infty$.
\end{remark}

Write $c=(c_1,c_2,c_3)\in \mathbb{R}_{+}^3$ and define $(B^{c,\lambda}_t)_{t\ge 0}$, a Brownian motion with parabolic drift by 
\begin{equation} \label{pard}
B^{c,\lambda}_t:=\frac{\sqrt{c_2}}{c_1} B_t+\lambda t-\frac{c_2 t^2}{2 c_1^3},
\end{equation}
where $(B_t)_{t\ge 0}$ is a standard Brownian motion. The most general result under minimum assumptions for the joint convergence of the component sizes and the corresponding surpluses was proven in  \cite{dhara2017critical}. Fix $\lambda\in \mathbb{R}$ and let $(M_i^n)_{i\ge 1}$ and $(R_i^n)_{i\ge 1}$ denote the sequence of the sizes and surpluses of the components of $\text{M}^n(d)$ respectively.

\begin{theorem}[\textbf{Dhara, Hofstad, Leeuwaarden, Sen \cite{dhara2017critical}}] \label{Van}
As $n\to \infty$,
\begin{equation}
\left(n^{-2/3} (M_i^n)_{i\ge 1},(R_i^n)_{i\ge 1}\right)\longrightarrow \left((M_i^{c_D})_{i\ge 1},(R_i^{c_D})_{i\ge 1}\right)
\end{equation}
in distribution, where the convergence of the first sequence takes place in $\ell^2_{\downarrow}$ and for the second in the product topology.
\end{theorem}

The limiting sequence $(M_i^{c_D})_{i\ge 1}$ is distributed as the ordered sequence of lengths of excursions of the process $(B_t^{c_D,\lambda})_{t\ge 0}$ reflected upon its minimum, where $c_D$ has coordinates depending only on the first three moments of $D$ and are given exactly by $c_1^D=E(D)$, $c_2^D=E(D^3) E(D)-(E(D^2))^2$ and $c_3^D=1/E(D)$. Drawing the graph of the reflected process, scattering points on the plane according to a Poisson with rate $c_3^D$ and keeping only those that fell between the $x$-axis and the function, describes $R_i^{c_D}$, $i\ge 1$ as the number of those points that fell in the excursion with length $M_i^{c_D}$. Note that Theorem \ref{PhT} follows as a corollary of Theorem \ref{Van} for the special choice of $c_{\text{ER}}=(1,1,1)$.

In Section \ref{lalala} we introduced $\mathcal{M}^{(\zeta)}$ as the random real tree coded by a tilted Brownian excursion of length $\zeta$ to which a number of point identifications to create cycles is added. The number of point identifications is a Poisson random variable with mean given by the area under the tilted excursion. Given that number, say $k\ge 0$, $x_i$ is picked with a density proportional to the height of the tilted excursion in an i.i.d. fashion for every $1\le i\le k$ and $y_i$ is picked uniformly from the path that connects the root to $x_i$. Then, $x_i$ and $y_i$ are identified, $1\le i\le k$. Let $\mathcal{M}^{(\zeta,c_3)}$ denote $\mathcal{M}^{(\zeta)}$ if the number of point identifications is instead Poisson with mean given by the area under the tilted excursion multiplied by $c_3$, i.e. $c_3 \int_{0}^{\zeta} \tilde{e}(u) du$.

Then, the limit of the largest connected component of the configuration on the scaling critical window, say $\text{M}_1^n(d)$, can be written as a scalar of $\mathcal{M}^{(M_1^{c_D},c_3^D)}$, where $M_1^{c_D}$ is distributed according to the length of the longest excursion of the reflected upon its minimum parabolic Brownian motion with coefficients dependent on $c_D$ as defined in \eqref{pard}. This statement is made precise as a simplified version of \cite[Theorem 2.4]{bhamidi2016geometry}, which we quote. As $n\to \infty$, 
\begin{equation} \label{uni1}
\left(n^{-2/3} M_1^n,\left(V(\text{M}^n_1(d)),n^{-1/3} d_{\text{M}_1^n(d)},\rho^n\right)\right)\longrightarrow \left(M_1^{c_D},\left(\mathcal{M}_D,\frac{c_1^D}{\sqrt{c_2^D}}d_{\mathcal{M}_D},\rho\right)\right),
\end{equation}
in distribution, where conditional on $M_1^{c_D}$, $\mathcal{M}_D\,{\buildrel (d) \over =}\,\mathcal{M}^{(M_1^{c_D},c_3^D)}$. Actually \cite[Theorem 2.4]{bhamidi2016geometry} holds also by considering the largest connected component of a uniform element of $G_{n,d}$ with $d$ satisfying the minimum assumptions in (i) and (ii). Again, the limit of the maximal components of $G(n,n^{-1}+\lambda n^{-4/3})$ can be recover by considering $D_{\text{ER}}$ to be a mean 1 Poisson random variable ($c_1^{D_{\text{ER}}}=c_2^{D_{\text{ER}}}=c_3^{D_{\text{ER}}}=1$) and the fact that $G(n,n^{-1}+\lambda n^{-4/3})$ conditioned on the degree sequence being $d$ is uniformly distributed over $G_{n,d}$. The result in \eqref{uni1} is to be regarded as a stepping stone to a general program that seeks to prove universality for the metric structure of the maximal components at criticality for a number of random graphs models with their distances scaling like $n^{1/3}$. For more details see \cite{bhamidi2014scaling}. 

We turn now our interest on how to sample uniformly a connected component with a given degree sequence $(\tilde{d}_i: i\in [\tilde{m}])$ that satisfies the following assumption. 
\begin{assum} \label{assum4}
(i) Let $\tilde{d}_1=1$, and $\tilde{d}_i\ge 1$, for every $1\le i\le \tilde{m}$.
\\
(ii) There exists a probability mass function $(\tilde{p}_i)_{i\ge 1}$ with the properties 
\[
\tilde{p}_1>0, \qquad \sum_{i\ge 1} i \tilde{p}_i=2, \qquad \sum_{i\ge 1} i^2 \tilde{p}_i<\infty
\]
such that 
\[
\frac{1}{\tilde{m}} \# \{j: \tilde{d}_j=i\}\to \tilde{p}_i, \text{ for all } i\ge 1, \text{ and } \qquad \frac{1}{\tilde{m}} \sum_{i\ge 1} \tilde{d}_i^2\to \sum_{i\ge 1} i^2 \tilde{p}_i.
\]
In particular, $\max_{1\le 1\le \tilde{m}} \tilde{d}_i=o(\sqrt{m})$.
\end{assum}

For a given rooted plane tree $\theta$ with root $\rho$, let $\text{c}(\theta)=(\text{c}_v(\theta))_{v\in \theta}$, where $\text{c}_v(\theta)$ gives the number of children of $v$ in $\theta$ and let $\text{s}(\theta)=(\text{s}_i(\theta))_{i\ge 0}$ be the empirical children distribution (ECD) of $\theta$, i.e. $\text{s}_i(\theta):=\#  \{v\in \theta: c_v(\theta)=i\}$, for every $i\ge 0$. Note that $\text{s}_0(\theta)=\# \mathcal{L}(\theta)$ gives exactly the number of leaves of $\theta$. Now, given a sequence of integers $\text{s}=(\text{s}_i)_{i\ge 0}$, it is easy to check that it is tenable for a tree $\theta$ if and only if $\text{s}_0\ge 1$, $\text{s}_i\ge 0$ for every $i\ge 1$, and 
\[
\sum_{i\ge 0} \text{s}_i=1+\sum_{i\ge 1} i \text{s}_i<\infty.
\]
Given $\text{s}$, let $\mathbf{T}_{\text{s}}$ denote the collection of all plane trees having ECD $\text{s}$. 

Let $x,y\in \mathcal{L}({\theta})$. We say that the ordered pair of leaves $(x,y)$ is admissible if $\text{par}(x)<_{\text{DF}} \text{par}(y)$, i.e. if the parent of $x$ is explored before the parent of $y$ during a depth-first search of $\theta$, and if $\text{gpar}(y)\in [[\rho,\text{gpar}(x)]]$, i.e. if the grandparent of $y$ belongs to ancestral line that connects the root to the grandparent of $x$. Let $(\mathbf{A}(\theta),<<)$ denote the collection of pairs of admissible leaves of $\theta$ endowed with the linear order << that declares $(x_1,y_1)<<(x_2,y_2)$ if and only if $x_1=x_2$ or $x_1<_{\text{DF}} x_2$ and $y_1<_{\text{DF}} y_2$. For $k\ge 1$, we denote by $\mathbf{A}_k(\theta)$ the collection of admissible $k$-tuples of $2 k$ distinct leaves and by $\mathbf{T}_{\text{s}}^k$ the pairs $(\theta,\text{z})$ for which $\theta\in \mathbf{T}_{\text{s}}$ and $\text{z}\in \mathbf{A}_k(\theta)$. Finally, for a rooted plane tree $\theta$ and $\text{z}=\{(x_1,y_1),...,(x_k,y_k)\}\in \mathbf{A}_k(\theta)$, we denote by $L(\theta,\text{z})$ the rooted plane tree obtained from $\theta$ performing the following operation. Delete $x_i$, $y_i$ together with the edges adjacent to them for each $i=1,...,k$ and add an edge between $\text{par}(x_i)$ and $\text{par}(y_i)$. We equip $L(\theta,\text{z})$ with the shortest path distance and the uniform probability measure on its set of vertices. Finally, let 

We will work with connected graphs with a fixed surplus, so suppose that 
\[
\sum_{i\in [\tilde{m}]} \tilde{d}_i=2(\tilde{m}-1)+2 k,
\]
for some fixed $k\ge 0$. Under Assumption \ref{assum4}, the lowest labelled vertex has one descendant and for the remaining vertices $2,...,\tilde{m}$ we form the children sequence $\text{c}:=(\text{c}_i)_{i=2}^{\tilde{m}+2 k}$ via $\text{c}_i:=\tilde{d}_i-1$, for every $2\le i\le \tilde{m}$, and $\text{c}_{\tilde{m}+j}:=0$, for every $1\le j\le 2 k$. Observe that 
\begin{equation} \label{childish}
\sum_{i=2}^{\tilde{m}+2 k} c_i=\sum_{i=2}^{\tilde{m}} \tilde{d}_i-(\tilde{m}-1)=(\tilde{m}-1+2 k)-1,
\end{equation}
and thus $\text{c}$ can be seen as representing a children sequence for a plane tree with $m:=\tilde{m}-1+2 k$ vertices. Write $\text{s}:=(\text{s}_i)_{i\ge 0}$ for its ECD. The following lemma is due to Bhamidi and Sen.

\begin{lemma} [\textbf{Bhamidi, Sen \cite{bhamidi2016geometry}}] \label{assf}
To generate uniformly a connected graph with prescibed degree sequence $\tilde{d}$ satisfying Assumption \ref{assum4}
\begin{enumerate}
\item Generate first $(\tilde{T}_{\textnormal{s}},\tilde{Z})$ uniformly from $\mathbf{T}_{\textnormal{s}}^k$. If $\tilde{Z}=\{(x_1,y_1),...,(x_k,y_k)\}$ with $(x_1,y_1)<<...<<(x_k,y_k)$, label $x_i$ as $\tilde{m}+2 i-1$ and $y_i$ as $\tilde{m}+2 i$, $1\le i\le k$. Label the rest of the $\tilde{m}-1$ vertices uniformly using the remaining labels $2,...,\tilde{m}$ so that in the resulting labelled tree the vertex $j$ has exactly $\tilde{d}_j-1$ children. Call this labelled tree $\tilde{T}_{\textnormal{s}}^{\textnormal{lb}}$.
\item Construct $L(\tilde{T}_{\textnormal{s}}^{\textnormal{lb}},\tilde{Z})$, attach a vertex labelled 1 to the root and forget about the planar order and the root. Call $\mathcal{G}$ the resulting graph.
\end{enumerate}
Then, $\mathcal{G}$ is distributed uniformly over the set of connected graph with prescribed degree sequence $\tilde{d}$.
\end{lemma}

We use $\rho^m$ to denote the root of $\tilde{T}_{\text{s}}$. We denote by  $\tilde{V}_m^{\text{s}}=(\tilde{V}_m^{\text{s}}(i): 0\le i\le 2 m)$ the contour process of $\tilde{T}_{\text{s}}$, the random tree generated according to 1 in the statement of Lemma \ref{assf}, and by $\tilde{v}_m^{\text{s}}=(m^{-1/2} \tilde{V}_m^{\text{s}}(2 m s): 0\le s\le 1)$ the rescaled contour process as well. In the next lemma we show that, for some $\alpha>0$, the sequence $||\tilde{v}_m^{\text{s}}||_{H_{\alpha}}$ of H\"older norms is tight.

\begin{lemma} \label{highon}
There exists $\alpha\in (0,1/2)$ such that for every $\varepsilon>0$ there exists a finite real number $M_{\varepsilon}$ such that 
\begin{equation} 
P\left(\sup_{t\in [0,1]}\frac{|\tilde{v}_m^{\textnormal{s}}(s)-\tilde{v}_m^{\textnormal{s}}(t)|}{|t-s|^{\alpha}}\le M_{\varepsilon}\right)\ge 1-\varepsilon.
\end{equation}
\end{lemma}

\begin{proof}
From the definition of $(\tilde{T}_{\text{s}},\tilde{Z})$, it is clear that 
\[
P(\tilde{T}_{\text{s}}=\theta)=\frac{|\mathbf{A}_k(\theta)|}{|\mathbf{T}_{\text{s}}^k|},
\]
for any $\theta\in \mathbf{T}_{\text{s}}$, i.e. $\tilde{T}_{\text{s}}$ is a random tree that has ``tilted'' distribution which is biased in favor of trees with large collection of admissible $k$-tuples between $2 k$ distinct leaves. Hence for any $f: C([0,1],\mathbb{R}_{+})\to \mathbb{R}_{+}$ bounded and continuous function,
\begin{equation} \label{tiltme}
E\left(f\left(\tilde{v}_{m}^{\text{s}}\right)\right)=\frac{E\left(f\left(v_m^{\text{s}}\right)|\mathbf{A}_k(T_{\text{s}})|\right)}{E\left(|\mathbf{A}_k(T_{\text{s}})|\right)},
\end{equation}
where $T_{\text{s}}$ a uniform plane tree having ECD $\text{s}$, which is specified by the children sequence described in \ref{childish}. Here, $v_m^{\text{s}}$ is the normalized contour function that encodes $T_{\text{s}}$. Note that when $\tilde{d}$ satisfies Assumption \ref{assum4}, $\text{s}$ satisfies the following. 
\[
\sum_{i\ge 0} \text{s}_i=m, \qquad \frac{\text{s}_i}{m}\to p_i, \text{ for all } i\ge 0,\text{ and } \qquad \frac{1}{m} \sum_{i\ge 0}i^2 \text{s}_i\to \sum_{i\ge 0} i^2 p_i.
\]
In particular, $\max \{i: \text{s}_i\neq 0\}=o(\sqrt{m})$. Also, $p:=(p_i)_{i\ge 0}$ is a probability mass function with $p_i:=\tilde{p}_{i+1}$, for every $i\ge 0$, and therefore it satisfies the properties
\begin{equation} \label{sleeping}
p_0>0, \qquad \sum_{i\ge 0} i p_i=1, \qquad \sum_{i\ge 0} i^2 p_i<\infty.
\end{equation} 
Now, by \eqref{tiltme} along with the Cauchy-Schwarz inequality, if $K_{\varepsilon}$ is the finite real number for which \eqref{proof2} holds, we deduce
\begin{align} \label{sleep}
P\left(\sup_{s,t\in [0,1]}\frac{|\tilde{v}_m^{\text{s}}(s)-\tilde{v}_m^{\text{s}}(t)|}{|t-s|^{\alpha}}\ge K_{\varepsilon}\right)&=\frac{E\left[\one_{\left\{\sup_{t\in [0,1]}\frac{|v_m^{\text{s}}(s)-v_m^{\text{s}}(t)|}{|t-s|^{\alpha}}\ge K_{\varepsilon}\right\}} |\mathbf{A}_k(T_{\text{s}})|\right]}{E\left(|\mathbf{A}_k(T_{\text{s}})|\right)}\nonumber \\
&\le \frac{P\left(\sup_{s,t\in [0,1]}\frac{|v_m^{\text{s}}(s)-v_m^{\text{s}}(t)|}{|t-s|^{\alpha}}\ge K_{\varepsilon}\right)^{1/2} \left(E\left[\left(\frac{|\mathbf{A}_k(T_{\text{s}})|}{\text{s}_0^{k} m^{k/2}}\right)^{2}\right]\right)^{1/2} }{E\left[\frac{|\mathbf{A}_k(T_{\text{s}})|}{\text{s}_0^{k} m^{k/2}}\right]}.
\end{align}
Using \cite[Lemma 6.3(ii)]{bhamidi2016geometry}, we have that 
\[
\sup_{m} E\left[\left(\frac{|\mathbf{A}_k(T_{\text{s}})|}{\text{s}_0^k m^{k/2}}\right)^{2}\right]\le \frac{1}{k!} \sup_{m} E\left[\left(\frac{|\mathbf{A}(T_{\text{s}})|}{\text{s}_0 \sqrt{m}}\right)^{2 k}\right]<\infty,
\]
for every $k\ge 1$. Furthermore, using \cite[Lemma 6.3(iii)]{bhamidi2016geometry}, and \cite[Lemma 6.3(v)]{bhamidi2016geometry} together with the uniform integrability from above, we conclude that
\[
E\left[\frac{|\mathbf{A}_k(T_{\text{s}})|}{\text{s}_0^{k} m^{k/2}}\right]\to \left(\frac{p_0 \sigma}{2}\right)^k E\left[\left(\int_{0}^{1} 2 e(u) du\right)^k\right]>0,
\]
as $m\to \infty$, where $(e(t))_{t\in [0,1]}$ is a normalized Brownian excursion and $\sigma^2=\sum_{i\ge 0} i^2 p_i-1$. It only remains to deal with the quantity $P\left(||v_m^{\text{s}}||_{H_{\alpha}}\ge K_{\varepsilon}\right)$ that appears on the right-hand side of \eqref{sleep}. It turns out that plane trees chosen uniformly from $\mathbf{T}_{\text{s}}$ are related to G-W trees by a simple conditioning. The uniform distribution on $\mathbf{T}_{\text{s}}$ coincides with the distribution of a G-W tree $\theta$ with offspring distribution $\mu:=(\mu_i)_{i\ge 0}$, which must satisfy $\mu_i>0$ if $\text{s}_i>0$, conditioned on the event $\cap_{i\ge 0} \{\text{s}_i(\theta)=\text{s}_i\}$. Take $\mu=p$ as in \eqref{sleeping} to be the critical offspring distribution with finite variance of a G-W tree $\theta$. Then, if $P_p$ is the probability distribution of $\theta$,
\[
P\left(||v_m^{\text{s}}||_{H_{\alpha}}\ge K_{\varepsilon}\right)=P_p\left(||v_m||_{H_{\alpha}}\ge K_{\varepsilon}|\text{s}_i(\theta)=\text{s}_i, i\ge 0\right),
\]
where $||v_m||_{H_{\alpha}}$ denotes the  $\alpha$-H\"older norm of the normalized contour function $v_m$ that encodes $\theta$. The proof is completed as a result of Theorem \ref{prepar2}.

\end{proof}

In Lemma \ref{assf} we saw that $\mathcal{G}$, a uniformly chosen connected graph with prescribed degree sequence $\tilde{d}$ that satisfies Assumption \ref{assum4} is distributed as $L(\tilde{T}_{\text{s}}^{\text{lb}},\tilde{Z})$, where $(\tilde{T}_{\text{s}}^{\text{lb}},\tilde{Z})$ is a uniform labelled element of $\mathbf{T}_{\text{s}}^k$. Recall that to obtain $L(\tilde{T}_{\text{s}}^{\text{lb}},\tilde{Z})$ from $(\tilde{T}_{\text{s}}^{\text{lb}},\tilde{Z})$, where $\tilde{Z}=\{(x_i,y_i),...,(x_k,y_k)\}$ with $(x_1,y_1)<<...<<(x_k,y_k)$, for every pair of admissible leaves we add an edge between $\text{par}(x_i)$ and $\text{par}(y_i)$, and delete $x_i$ and $y_i$ and the two edges incident to them for $1\le i\le k$. The resistance on $L(\tilde{T}_{\text{s}}^{\text{lb}},\tilde{Z})$ between two vertices is smaller than the total length of the path between them on $\tilde{T}_{\text{s}}^{\text{lb}}$. This observation together with Lemma \ref{highon} is enough to establish equicontinuity of the rescaled local times $(L^{\mathcal{G}}_t(x))_{x\in V(\mathcal{G}),t\ge 0}$ of the simple random walk on $\mathcal{G}$ under the annealed law (which is defined similar to \eqref{exemplary}). Since the proof relies heavily on arguments that are present in the proof of Proposition \ref{own} we omit it.

\begin{lemma} \label{sizeb3}
For every $\varepsilon>0$ and $T>0$,
\[
\lim_{\delta\rightarrow 0} \limsup_{m\to \infty} \mathbb{P}_{\rho^m} \left(\sup_{\substack{y,z\in V(\mathcal{G}): \\ m^{-1/2} R_{\mathcal{G}}(y,z)<\delta}} \sup_{t\in [0,T]} m^{-1/2} |L_{m^{3/2} t}^{\mathcal{G}}(y)-L_{m^{3/2} t}^{\mathcal{G}}(z)|\ge \varepsilon\right)=0.
\]
\end{lemma}

Assume that $d$ satisfies Assumption \ref{gracias} with limiting random variable $D$, and let $D^*$ denote its size-biased distribution given by 
\[
p^{*}_{i}:=P(D^*=i)=\frac{i P(D=i)}{E(D)}, \qquad i\ge 1.
\]
Then, for $M_1^n(d)$, the largest connected component of the configuration model $M^n(d)$, 
\begin{equation} \label{sizeb1}
\frac{\# \{v\in M_1^n(d): d_v=i\}}{|V(M_1^n(d))|}\xrightarrow{\text{P}} p_i^{*}, \text{ for all } i\ge 1, \qquad \frac{1}{|V(M_1^n(d))|} \sum_{v\in M_1^n(d)} d_v^2\xrightarrow{\text{P}} \sum_{i\ge 1} i^2 p_i^{*}<\infty,
\end{equation} 
\begin{equation} \label{sizeb2}
P(M_1^n(d) \text{ is simple})\to 1.
\end{equation}
For a justification of \eqref{sizeb1} and \eqref{sizeb2} see \cite[Proposition 8.2]{bhamidi2016geometry}. Note that $P(D=1)>0$ under Assumption \ref{gracias}, and hence $p_1^{*}>0$. Furthermore, under Assumption \ref{gracias},
\[
\sum_{i\ge 1} i p_i^{*}=\frac{E(D^2)}{E(D)}=2,
\]
and this shows along with \eqref{sizeb1} that $(d_v: v\in M_1^n(d))$ satisfies Assumption \ref{assum4} (after a possible) with limiting probability mass function $p^{*}:=(p_i^{*})_{i\ge 1}$. Let $\mathcal{P}$ denote the partition of $M^n(d)$ into different components. Conditional on the event $\{M_1^n(d) \text{ is simple}\}\cap \{\mathcal{P}=P\}$, $M_1^n$ is uniformly distributed over the set of simple, connected graphs with degree sequence decided by the partition $P$ \cite[Proposition 7.7]{van2016random}. Since $P(M_1^n(d) \text{ is a multigraph})\to 0$ by \eqref{sizeb2}, the following Proposition simply follows as a combination of Theorem \ref{Van} and Lemma \ref{sizeb3}.

\begin{proposition}
Under Assumption \ref{gracias}, for every $\varepsilon>0$ and $T>0$, the rescaled local times of the simple random walk on $M_1^n(d)$ are equicontinuous under the annealed law, i.e.
\[
\lim_{\delta\rightarrow 0} \limsup_{n\to \infty} \mathbb{P}_{\rho^n} \left(\sup_{\substack{y,z\in V(M_1^n(d)): \\ n^{-1/3} R_{M_1^n(d)}(y,z)<\delta}} \sup_{t\in [0,T]} n^{-1/3} |L_{n t}^n(y)-L_{n t}^n(z)|\ge \varepsilon\right)=0.
\]
\end{proposition}

\subsubsection{Convergence of the walks} \label{newresult}

Croydon \cite{croydon2016scaling} used regular resistance forms to describe the scaling limit of the associated random walks on scaling limits of sequences of spaces equipped with resistance metrics and measures provided that they converge with respect to a suitable Gromov-Hausdorff topology, and under the assumption that a non-explosion condition is satisfied. For families of random graphs that are nearly trees and their scaling limit can be described as a tree `glued' at a finite number of pairs of points, a useful corollary of \cite[Theorem 1.2]{croydon2016scaling} combined with \cite[Proposition 8.4]{croydon2016scaling} yields the convergence of the processes associated with the fused spaces. 

To see that the conclusion of \cite[Proposition 8.4]{croydon2016scaling} holds, recall that under Assumption \ref{gracias}, jointly with Theorem \ref{Van},
\[
\left(V(\text{M}^n_1(d)),n^{-1/3} d_{\text{M}_1^n(d)},\rho^n\right)\longrightarrow \left(\mathcal{M}_D,\frac{c_1^D}{\sqrt{c_2^D}}d_{\mathcal{M}_D},\rho\right),
\]
as $n\to \infty$ in the Gromov-Hausdorff-Prokhorov sense. Let $\mathcal{P}$ denote the partition of $M^n(d)$ into different components. Conditional on the event $\{M_1^n(d) \text{ is simple}\}\cap \{\mathcal{P}=P\}$, $M_1^n$ is uniformly distributed over the set of simple, connected graphs with degree sequence decided by the partition $P$, and therefore the convergence above is valid with $M_1^n(d)$ replaced by $L(\tilde{T}_{\textnormal{s}},\tilde{Z})$ (see Lemma \ref{assf} for its construction). If $\tilde{Z}=\{(x_1,y_1),...,(x_{R_1^n},y_{R_1^n})\}$ with $(x_1,y_1)<<...<<(x_{R_1^n},y_{R_1^n})$, let $D(\tilde{T}_{\textnormal{s}},\tilde{Z})$ be the space obtained by fusing $x_i$ and $\text{gpar}(y_i)$, $1\le j\le R_1^n$, endowed with the graph distance and the push-forward of the uniform probability measure on $\tilde{T}_{\textnormal{s}}$, and observe that 
\[
d_{\text{GHP}}(L(\tilde{T}_{\textnormal{s}},\tilde{Z}),D(\tilde{T}_{\textnormal{s}},\tilde{Z}))\le 5 R_1^n.
\]
Thus, jointly with Theorem \ref{Van}, the convergence above is valid with $M_1^n(d)$ replaced with the `glued' tree $D(\tilde{T}_{\textnormal{s}},\tilde{Z})$, and since $\mathcal{M}_D$ is also a `glued' tree, this shows that the conclusion of \cite[Proposition 8.4]{croydon2016scaling} is valid.

Fix $r\ge 0$. It remains to show that
\[
\lim_{r\to \infty} \liminf_{n\to \infty} P\left(n^{-1/3} d_{M_1^n(d)}(\rho_n,B_n(\rho_n,r)^c)\ge \lambda\right)=1,
\]
for every $\lambda\ge 0$. Indeed,
\begin{align*}
P\left(n^{-1/3} d_{M_1^n(d)}(\rho_n,B_n(\rho_n,r)^c)\ge \lambda\right)&\ge P\left(n^{-1/3} d_{M_1^n(d)}(\rho_n,B_n(\rho_n,r)^c)\ge \lambda,n^{-1/3} D_1^n(d)\le r\right)
\\
&=P(n^{-1/3} D_1^n(d)\le r),
\end{align*}
where $D_1^n(d):=\text{diam}(M_1^n(d))$. Letting $D_1(d):=\text{diam}(\mathcal{M}_D)$, since
\[
|n^{-1/3} D_1^n(d)-D_1(d)|\le 2 d_{\textnormal{GH}}(n^{-1/3} M_1^n(d),\mathcal{M}_D),
\]
we deduce
\[
\liminf_{n\to \infty} P\left(n^{-1/3} d_{M_1^n(d)}(\rho_n,B_n(\rho_n,r)^c)\ge \lambda\right)\ge \liminf_{n\to \infty} P(n^{-1/3} D_1^n(d)\le r)=P(D_1(d)\le r).
\]
As $r\to \infty$ the right-hand-side tends to 1, and this shows that the non-explosion condition is fulfilled. As a consequence we have the convergence of the processes associated with the fused spaces.

\begin{theorem}
It is possible to isometrically embed $(M_1^n(d),d_{M_1^n(d)})$, $n\ge 1$ and $(\mathcal{M}_D,d_{\mathcal{M}_D})$ into a common metric space $(F,d_F)$ such that 
\begin{equation}
\mathbf{P}^{M_1^n(d)}_{\rho^n}\left(n^{-1/3} (X_{\lfloor n t\rfloor}^{M_1^n(d)})_{t\ge 0}\in \cdot \right)\to \mathbf{P}^{\mathcal{M}_D}_{\rho}\left((X_t)^{\mathcal{M}_D}_{t\ge 0}\in \cdot \right)
\end{equation}
weakly as probability measures in $D(\mathbb{R}_{+},F)$.
\end{theorem}

\subsubsection{Continuity of blanket times of Brownian motion on \texorpdfstring{$\mathcal{M}_{D}$}{MD}}

In Section \ref{excm} we presented $\mathbb{N}^{t,\lambda}$, the inhomogeneous excursion (for excursions starting at time $t$) measure associated with a Brownian motion with parabolic drift as defined in \eqref{parabola}. Denote by $\mathbb{N}^{c,\lambda}_{t}$ the excursion measure associated with $B^{c,\lambda}$ as defined in \eqref{pard}. Write $(e(u): 0\le u\le t)$ for the canonical process under $\mathbb{N}$. By the Cameron-Martin-Girsanov formula \cite[Chapter IX, (1.10) Theorem]{revuz1999continuous}, applied under $\mathbb{N}$,
\[
\frac{d \mathbb{N}^{c,\lambda}_{0}}{d \mathbb{N}}=\exp\left(\frac{\sqrt{c_2}}{c_1} \int_{0}^{t}\gamma(u) d e(u)-\frac{1}{2} \int_{0}^{t} \gamma^2(u) du\right),
\]
where $\gamma(u):=\lambda-\frac{c_2}{c_1^3} u$ is the drift. On the sets of excursions of length $t$, using integration by parts we have that 
\[
\frac{\sqrt{c_2}}{c_1} \int_{0}^{t} \left(\lambda-\frac{c_2}{c_1^3}u\right)d e(u)=\frac{c_2^{3/2}}{c_1^4} \int_{0}^{t} e(u) du,
\]
a multiplicative of the area under the excursion of length $t$. So, the density becomes
\[
\frac{d \mathbb{N}^{c,\lambda}_{0}}{d \mathbb{N}}=\exp\left(\frac{c_2^{3/2}}{c_1^4} \int_{0}^{t} e(u) du-\frac{1}{6}\left(\left(\frac{c_2}{c_1^3} t-\lambda\right)^3+\lambda^3\right)\right).
\]
There is a corresponding probability measure $\mathbb{N}_{0,l}^{c,\lambda}:=\mathbb{N}_0^{c,\lambda}(\cdot |\tilde{L}=l)$, which for a Borel set $\mathcal{B}$ on the space of positive excursions of finite length, is determined by 
\[
\mathbb{N}^{c,\lambda}_{0,l}(\one_{\mathcal{B}})=\frac{\mathbb{N}_l\left(\exp\left(\frac{c_2^{3/2}}{c_1^4} \int_{0}^{l} e(u) du\right) \one_{\mathcal{B}} \right)}{\mathbb{N}_l\left(\exp\left(\frac{c_2^{3/2}}{c_1^4} \int_{0}^{l} e(u) du\right)\right)}.
\]
To determine $\mathbb{N}^{c,\lambda}_0(\tilde{L}\in dl)$ recall that $\mathbb{N}(L\in dl)=f_L(\lambda)=dl/\sqrt{2 \pi l^3}$, $l\ge 0$, and therefore
\[
\mathbb{N}^{c,\lambda}_0(\tilde{L}\in dl)=f_L(l) \exp\left(-\frac{1}{6}\left(\left(\frac{c_2}{c_1^3}l-\lambda\right)^3+\lambda^3\right)\right) \mathbb{N}_l\left(\exp\left(\frac{c_2^{3/2}}{c_1^4} \int_{0}^{l} e(u) du\right)\right).
\]
Let $\mathbf{N}^{c,\lambda}_t$  denote the canonical measure that first at time $t$ picks a tilted Brownian excursion of a randomly chosen length $l$, and then independently of $e$ chooses a number of points according to a Poisson random variable with mean $c_3 \int_{0}^{l} e(t) dt$, which subsequently are distributed uniformly on the area under the graph of $e$. In comparison with \eqref{meas3} we characterize $\mathbf{N}^{c,\lambda}_t(d(e,\mathcal{P}))$ by setting
\begin{align*}
&\mathbf{N}^{c,\lambda}_t(d e,|\mathcal{P}|=k, (d x_1,...,d x_k)\in A_1\times...\times A_k)
\nonumber \\ :=
&\int_{0}^{\infty} \mathbb{N}^{c,\lambda-t}_0(\tilde{L}\in dl) \mathbb{N}^{c,\lambda}_{t,l} (d e) \exp\left(-c_3 \int_{0}^{l} e(u) du\right)\frac{\left(c_3 \int_{0}^{l} e(u) du\right)^k}{k!} \prod_{i=1}^{k} \frac{\ell(A_i\cap A_e)}{\ell(A_e)}.
\end{align*} 
It is easy to see that $\mathbf{N}^{c,\lambda}_t$ is absolutely continuous with respect to $\mathbf{N}$ as defined in \eqref{meas1}. More specifically, 
\[
\frac{d \mathbf{N}_t^{c,\lambda}}{d \mathbf{N}}=\exp \left(1-\frac{1}{6}\left(\lambda^3+\left(\frac{c_2}{c_1^3 }l-\lambda+t\right)^3\right)\right) \left(c_3 \int_{0}^{l} e(u) du\right)^k.
\]
Using the same methods as in Section \ref{excm}, we can prove that 
\begin{theorem} 
Fix $\varepsilon\in (0,1)$. If $\tau_{\textnormal{bl}}^{n}(\varepsilon)$ is the $\varepsilon$-blanket time variable of the random walk on $M_1^n(d)$, started from its root $\rho^n$, then 
\[
\mathbb{P}_{\rho^n}\left(n^{-1} \tau_{\textnormal{bl}}^n(\varepsilon)\le t\right)\to \mathbb{P}_{\rho}\left(\tau_{\textnormal{bl}}^{\mathcal{M}_D}(\varepsilon)\le t\right),
\]
for every $t\ge 0$, where $\tau_{\textnormal{bl}}^{\mathcal{M}_D}(\varepsilon)\in (0,\infty)$ is the $\varepsilon$-blanket time variable of the Brownian motion on $\mathcal{M}_D$, started from $\rho$ and $\mathbb{P}_{\rho}$ is the annealed law defined similarly as in \eqref{lastbits}.
\end{theorem}

\section*{Acknowledgements}

I would like to thank my supervisor Dr David Croydon for suggesting the problem, his support and many useful discussions.

\bibliographystyle{abbrv}
\bibliography{BlanketConvergenceAndriopoulos}

\end{document}